\documentclass[11pt]{amsart}
\usepackage{amsmath, amscd, amsthm, amssymb, amsxtra,epic,eepic}
\usepackage{graphics}
\usepackage{url}
\usepackage{mathtools} 
\usepackage[margin=3cm]{geometry} 
\usepackage[hypertex]{hyperref}
\usepackage{amsrefs}
\usepackage[all]{xy}

\newtheorem{thm}{Theorem}[section] 
\newtheorem{pro}[thm]{Proposition} 
\newtheorem{lem}[thm]{Lemma} 
\newtheorem{cor}[thm]{Corollary} 

\theoremstyle{definition} 
\newtheorem{defn}[thm]{Definition} 

\theoremstyle{remark} 
\newtheorem{rem}[thm]{Remark}
\newtheorem{exa}[thm]{Example}

\newcommand{\CC}{\mathbb C}

\newcommand{\NN}{\mathbb N}
\newcommand{\PP}{\mathbb P}
\newcommand{\QQ}{\mathbb Q}
\newcommand{\RR}{\mathbb R}
\newcommand{\UU}{\mathbb U}
\newcommand{\VV}{\mathbb V}
\newcommand{\ZZ}{\mathbb Z}

\newcommand{\Ocal}{\mathcal O}
\newcommand{\Ccal}{\mathcal C}

\newcommand{\Tcal}{\mathcal T}

\newcommand{\Mcal}{\mathcal M}

\newcommand{\rest}[1]{{\textstyle|}_{#1}}
\newcommand{\id}{\operatorname{id}}
\newcommand{\Aut}{\operatorname{Aut}}

\newcommand{\pt}{\operatorname{pt}}

\newcommand{\nar}{\operatorname{nar}}

\newcommand{\age}{\operatorname{age}}

\newcommand{\gr}{\operatorname{gr-}}

\newcommand{\End}{\operatorname{End}}

\newcommand{\Hom}{\operatorname{Hom}}

\newcommand{\Spec}{\operatorname{Spec}}

\newcommand{\Res}{\operatorname{Res}}
\newcommand{\al}{\alpha}

\newcommand{\la}{\lambda}

\newcommand{\ev}{{\rm ev}}

\newcommand{\ctop}{{c}_{\mathrm{top}}}

\newcommand{\tw}{{\rm tw}}
\newcommand{\un}{{\rm un}}
\newcommand{\ch}{{\rm ch}}

\newcommand{\textsum}{{\textstyle{\sum}}}

\providecommand{\abs}[1]{\lvert#1\rvert}

\def\vir{{\rm vir}}

\def\nn{N}

\def\MFhom{\operatorname{{MF}}}
\def\GW{\operatorname{\text{GW}}}

\def\CR{\operatorname{\text{CR}}}

\def\ol{\overline}

\def\wt{\widetilde}

\def\parfrac#1#2{\frac{\partial #1}{\partial #2}}
\def\corr#1{\left\langle #1 \right\rangle} 
\def\fracp#1{\left\langle #1 \right\rangle}
\def\pair#1#2{\langle #1, #2\rangle} 
\def\Pair#1#2{\left\langle #1, #2\right \rangle}  
\def\floor#1{\left\lfloor #1 \right\rfloor}
\def\ceil#1{\left \lceil #1 \right \rceil}
\def\miyo#1{#1}

\font\rsfs=rsfs10
\newcommand{\marsfs}[1]{\mbox{\rsfs #1}}

\newcommand{\bmu}{\boldsymbol{\mu}} 
\newcommand{\bs}{\mathbf{s}} 
\newcommand{\bp}{\mathbf{p}} 
\newcommand{\bq}{\mathbf{q}} 
\newcommand{\bk}{\mathbf{k}} 
\newcommand{\bF}{\mathbf{F}}

\newcommand{\tW}{\widetilde{W}} 
\newcommand{\cO}{\mathcal{O}} 
\newcommand{\cM}{\mathcal{M}} 
\newcommand{\cMo}{\mathcal{M}^\circ}
\newcommand{\hcM}{\widehat{\cM}}
\newcommand{\tcMo}{\widetilde{\cM}^\circ}
\newcommand{\cC}{\mathcal{C}} 
\newcommand{\cL}{\mathcal{L}} 
\newcommand{\cI}{\mathcal{I}} 
\newcommand{\cE}{\mathcal{E}} 
\newcommand{\cF}{\mathcal{F}} 
\newcommand{\cG}{\mathcal{G}} 
\newcommand{\tcG}{\widetilde{\cG}} 
 
\newcommand{\hcF}{\widehat{\cF}}
\newcommand{\cR}{\mathcal{R}} 
\newcommand{\hcR}{\widehat{\cR}}
\newcommand{\cD}{\mathcal{D}} 
\newcommand{\cH}{\mathcal{H}} 
\newcommand{\cB}{\mathcal{B}} 

\newcommand{\frF}{\mathfrak{F}} 
\newcommand{\frs}{\mathfrak{s}} 
\newcommand{\frf}{\mathfrak{f}}

\newcommand{\Yo}{Y^\circ} 

\newcommand{\dd}{\mathsf{d}} 

\newcommand{\hc}{\hat{c}} 
\newcommand{\hrho}{\hat{\rho}} 
\newcommand{\hdelta}{\hat{\delta}}

\newcommand{\unit}{\boldsymbol{1}}

\newcommand{\Nar}{\operatorname{\mathsf{Nar}}}
\newcommand{\amb}{\operatorname{amb}}
\newcommand{\pri}{\operatorname{pri}} 
\newcommand{\ext}{{\operatorname{ext}}} 
 
\newcommand{\FJRW}{{\operatorname{FJRW}}} 
\newcommand{\sFJRW}{\operatorname{_{FJRW}}}
\renewcommand{\GW}{{\operatorname{GW}}}
\renewcommand{\CR}{{\operatorname{CR}}}
\newcommand{\sGW}{\operatorname{_{GW}}}
\newcommand{\Auteq}{\operatorname{Auteq}}
\newcommand{\Spin}{\operatorname{Spin}} 
\newcommand{\bro}{\operatorname{bro}} 
\newcommand{\Ker}{\operatorname{Ker}} 
\newcommand{\Image}{\operatorname{Im}}
\newcommand{\Eff}{\operatorname{Eff}} 
\newcommand{\grading}{\operatorname{\mathsf{Gr}}} 
\newcommand{\Gr}{\operatorname{Gr}} 
\newcommand{\Grass}{\operatorname{\mathit{Gr}}} 
\newcommand{\Hess}{\operatorname{Hess}} 
 
\newcommand{\inv}{\operatorname{inv}} 
\newcommand{\sfHom}{\operatorname{\mathsf{Hom}}} 
\newcommand{\str}{\operatorname{str}} 
\newcommand{\tr}{\operatorname{tr}} 
\newcommand{\pr}{\operatorname{pr}} 
 
\newcommand{\Mir}{\operatorname{Mir}} 
\newcommand{\Cone}{\operatorname{Cone}} 
\newcommand{\CY}{\operatorname{CY}} 
\newcommand{\LG}{\operatorname{LG}} 
\newcommand{\con}{\operatorname{con}} 
\newcommand{\rank}{\operatorname{rank}} 
\newcommand{\ord}{\operatorname{ord}} 
\newcommand{\Todd}{\operatorname{Td}} 

\newcommand{\hGamma}{\widehat{\Gamma}}
\newcommand{\hgamma}{\hat{\gamma}} 
\newcommand{\hvarphi}{\hat{\varphi}} 
\newcommand{\hR}{\widehat{R}} 
\newcommand{\hE}{\widehat{E}} 
\newcommand{\hF}{\widehat{F}} 
\newcommand{\hU}{\widehat{U}} 

\newcommand{\tU}{\widetilde{U}}

\newcommand{\tch}{\widetilde{\ch}} 

\newcommand{\ov}{\overline}

\newcommand{\iu}{\mathtt{i}}

\newcommand{\vc}{v_{\rm c}} 
\newcommand{\sfx}{\mathsf{x}} 
\newcommand{\mW}{\mathsf{W}} 

\newcommand{\tUU}{\widetilde{\UU}} 

\newcommand{\Hf}{{\mathfrak{H}}}

\newcommand{\rsF}{\marsfs{F}} 
\newcommand{\rsW}{\marsfs{W}} 

\newcommand{\uw}{\underline{w}} 
\newcommand{\ua}{\underline{a}} 
\newcommand{\ub}{\underline{b}}

\begin{document}

\title[LG/CY Correspondence, Global Mirror Symmetry 
and Orlov Equivalence]
{Landau-Ginzburg/Calabi-Yau Correspondence, \\
Global Mirror Symmetry and Orlov Equivalence}

\author{Alessandro Chiodo}
\address{Institut de Math\'{e}matiques de Jussieu, 
UMR 7586 CNRS, Universit\'{e} Pierre et Marie Curie, Case 247, 
4 Place Jussieu, 
75252 Paris cedex 05, 
France} 
\email{chiodoa@math.jussieu.fr} 


\author{Hiroshi Iritani}
\address{Department of Mathematics, Graduate School of Science, 
Kyoto University, Kitashirakawa-Oiwake-cho, Sakyo-ku, 
Kyoto, 606-8502, Japan}
\email{iritani@math.kyoto-u.ac.jp}

\author{Yongbin Ruan}
\address{Department of Mathematics,
University of Michigan, Ann Arbor, MI 48109-1109, USA}
\email{ruan@umich.edu}

\begin{abstract}
We show that the Gromov-Witten theory of Calabi-Yau hypersurfaces
matches, in genus zero and after an analytic continuation, 
the \miyo{quantum singularity theory (FJRW theory)}  
recently introduced by Fan, Jarvis and Ruan following a proposal of Witten. 
Moreover, on both sides, 
we highlight two remarkable integral local systems 
arising from the common formalism of 
\miyo{$\hGamma$-integral structures} 
applied to the derived category of the hypersurface $\{W=0\}$ 
and to the category of graded matrix factorizations of $W$.
In this setup, we prove that the analytic continuation matches
Orlov equivalence between the two above categories.
\end{abstract}

\maketitle

\let\oldtocsection=\tocsection
\let\oldtocsubsection=\tocsubsection

\renewcommand{\tocsection}[2]{\hspace{0em}\oldtocsection{#1}{#2}}
\renewcommand{\tocsubsection}[2]{\hspace{1em}\oldtocsubsection{#1}{#2}}
{
\tableofcontents
} 

\section{Introduction}
\label{sect:intro}

\subsection{Overview} 
\label{subsec:overview} 

The so-called Landau-Ginzburg/Calabi-Yau correspondence 
(LG/CY correspondence for short) in string 
theory \cite{Martinec, Vafa-Warner, Greene-Vafa-Warner} 
describes a relationship between the sigma model on a 
Calabi-Yau hypersurface and the Landau-Ginzburg model 
whose potential is the defining equation of the Calabi-Yau. 
In Witten's gauged linear sigma model \cite{Witten:phases}, 
the LG/CY correspondence arises, roughly speaking, from 
a variation of GIT quotient.  

Let $w_1,\dots, w_N$ be coprime positive integers 
and $x_1,\dots, x_N$ be variables of degree $w_1,\dots, w_N$.  
Let $W(x_1,\dots,x_N)$ be a weighted homogeneous polynomial 
of degree $d$ which has an isolated critical point only at the origin. 
We assume (i) the Calabi-Yau condition $d=w_1+\cdots+w_N$ 
and (ii) the Gorenstein condition\footnote
{This means that the ambient space $\PP(\uw)$ is Gorenstein. 
In this case, we can take $W$ to be the Fermat type polynomial 
$W=x_1^{d/w_1} + \cdots + x_N^{d/w_N}$.}: 
$w_j$ divides $d$ for all $1\le j\le N$. 
In this paper, we discuss two objects: 
\begin{itemize} 
\item The Calabi-Yau hypersurface $X_W=\{ W=0\}$ in 
the weighted projective space $\PP(\uw) = 
\PP(w_1,\dots, w_N)$. 
This is quasi-smooth (i.e.\ smooth in the sense of 
stacks) by the assumption on $W$ above. 

\item The Landau-Ginzburg orbifold $(\CC^N, W,\bmu_d)$. 
It consists of the space $\CC^N$ 
equipped with an action of 
$\bmu_d = \{g \in \CC^\times\, |\, g^d =1\}$, 
$(x_1,\dots, x_N) \mapsto (g^{w_1} x_1,\dots, g^{w_N} x_N)$ 
and a $\bmu_d$-invariant function $W\colon \CC^N \to \CC$.   
\end{itemize} 
These two models arise from the following GIT quotient.  
Consider the $\CC^\times$-action on the vector space 
$\CC^N \times \CC$ with co-ordinates $(x_1,\dots,x_N, p)$: 
\[
(x_1,\dots, x_N, p) \mapsto (t^{w_1} x_1,\dots, t^{w_N} x_N , t^{-d} p), 
\quad t \in \CC^\times.   
\]
We endow the space $\CC^N \times \CC$ with 
the $\CC^\times$-invariant potential $\tW(x,p) := p W(x)$. 
There are two possible GIT quotients of this space: 
one is the quotient of $(\CC^N \setminus \{0\}) \times \CC$ 
and the other is the quotient of 
$\CC^N \times (\CC\setminus \{0\})$. 
In the former case, we get the total space 
of the line bundle $\cO(-d)\to \PP(\uw)$ 
endowed with the function $\tW$. 
This should reduce to 
the sigma model on the Calabi-Yau hypersurface $X_W$. 
In the latter case, 
we get the Landau-Ginzburg orbifold $(\CC^N,W,\bmu_d)$.  

The GIT quotient itself does not change inside 
a ``chamber" of stability parameters,   
but the actual physical theory depends on 
a \emph{continuous and complexified} 
stability parameter $r + \iu \theta \in \CC$.  
The CY theory arises for $r\to \infty$ 
and the LG theory for $r \to -\infty$. 
The stability parameter $r+\iu \theta$ varies 
along a complex manifold $\cM$ called the global K\"{a}hler moduli space. 
In the case at hand, it is identified with (a Zariski open subset of) 
the weighted projective line $\PP(1,d)$.
The local picture near the $\bmu_d$-point in $\PP(1,d)$ 
corresponds to the LG model above 
and the $\bmu_d$-point is called LG point. 
The local picture near the antipodal point 
corresponds to the CY geometry 
and the antipodal point is called 
large radius limit point.  
These points are interesting asymptotically: 
we often work on punctured discs centered on them 
and refer to them as limit points (see Figure \ref{fig:Kaehlermoduli}). 
There is another limit point in the K\"{a}hler moduli space 
where the mirror has a conifold singularity; by abuse of 
terminology, we refer to this limit point as 
a conifold point. 

\begin{figure}[htbp]
\begin{center} 
\includegraphics{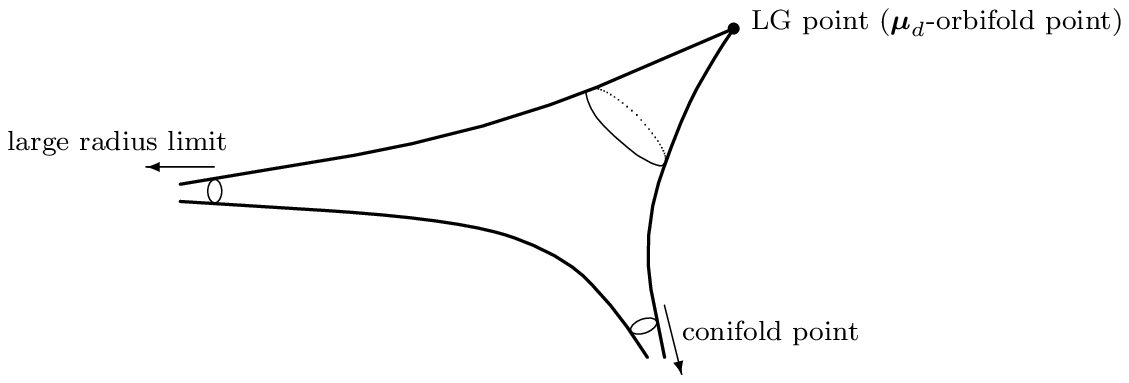}
%
%
%
%
\end{center} 







\caption{K\"{a}hler moduli space 
$\cM \cong \PP(1,d) \setminus \{\text{two points}\}$}
\label{fig:Kaehlermoduli} 
\end{figure} 

This paper is concerned with two aspects 
of topological string theory: 
the category of D-branes of type B (B-branes) and 
the closed string theory of type A (A-model). 
In this paper, the term ``B-brane" (or ``brane") 
has a precise mathematical meaning. 
On the CY side, the category of B-branes 
is the derived category of coherent sheaves on 
the Calabi-Yau hypersurface $X_W$. 
On the LG side, the category of B-branes 
is identified with the category 
$\MFhom^{\rm gr}_{\bmu_d}(W)$ 
of graded matrix factorizations of $W$ 
\cite{Orlov:cat-sing,Walcher, Hori-Walcher:F-term}. 
On the other hand, the mathematical A-model 
on the Calabi-Yau $X_W$ is given by GW (Gromov-Witten) theory. 
The mathematical A-model for the Landau-Ginzburg orbifold 
was formulated recently by Fan-Jarvis-Ruan \cite{FJR1} 
as the intersection theory on the moduli space 
of $W$-spin curves. 
This is called FJRW (Fan-Jarvis-Ruan-Witten) theory. 
About these theories, 
the following LG/CY correspondences  
are known in mathematics: 
\begin{enumerate} 
\item 
\label{item:B-brane} 
\textbf{B-brane LG/CY correspondence}: 
Orlov \cite{Orlov:equivalence} constructed derived equivalences 
$\Phi_l$ between the categories of B-branes 
indexed by an integer $l\in\ZZ$: 
\[
\Phi_l \colon D^b(X_W) \cong \MFhom^{\rm gr}_{\bmu_d}(W). 
\]

\item 
\label{item:A-model} 
\textbf{A-model LG/CY correspondence}:  
Chiodo-Ruan \cite{ChiRquintic} showed that 
for a quintic polynomial $W(x_1,\dots,x_5)$, 
GW theory of $X_W$ is analytically continued 
to FJRW theory of $(\CC^5, W,\bmu_5)$ at genus zero. 
Schematically, we write 
\[
\GW_{g=0}(X_W) \cong \FJRW_{g=0}(\CC^5,W,\bmu_5). 
\]
\end{enumerate} 
The purpose of this paper is to extend 
the correspondence \eqref{item:A-model} 
to a general weighted homogeneous polynomial $W$ 
and to describe a relationship between 
\eqref{item:B-brane} and \eqref{item:A-model}.

More precisely, ``analytic continuation"  
in \eqref{item:A-model} means the following. 
In genus zero, GW theory and FJRW theory 
yield \emph{quantum $D$-modules} 
over small neighbourhoods of the corresponding limit points; 
we show that these are restrictions of a certain \emph{global} 
$D$-module over the K\"{a}hler moduli space $\cM$. 
Note that the category of B-branes 
should be independent of the K\"{a}hler structure on $X_W$. 
Hence B-branes are ``locally constant" data 
over the K\"{a}hler moduli space 
around the limit point in each theory. 
In fact, we associate to each B-brane 
a \emph{flat section} of the quantum $D$-module. 
The flat section here is asymptotic (in the limit point)  
to the Chern character of the brane multiplied 
by the $\hGamma$-class. 
This defines a $\ZZ$-local system 
underlying the quantum $D$-module 
whose fibre is the numerical $K$-group 
of the category of B-branes. We call it  
the \emph{$\hGamma$-integral structure} of the 
quantum $D$-module. 
This has been introduced for GW theory 
by Iritani \cite{Ir} and Katzarkov-Kontsevich-Pantev \cite{KKP}. 
Here the role of the $\hGamma$-class (see Definition 
\ref{defn:Gamma}) is to preserve the Euler pairing 
$\chi(\cE,\cF) := \sum_{i\in \ZZ} \dim \Hom(\cE,\cF[i])$ 
in B-brane categories. 
Indeed, $\hGamma$ can be regarded 
as a square root of the Todd class. When $X_W$ is a manifold, 
we have 
\[
\left( (-1)^{\frac\deg 2} \hGamma_{X_W} \right) \cdot \hGamma_{X_W} 
=  (2\pi\iu)^{\frac\deg 2} \Todd_{X_W}  
\]
thanks to the functional equation 
$\Gamma(1-z) \Gamma(1+z) = \pi z/\sin(\pi z)$. 
Thus one can think of our flat section associated to a B-brane 
as a quantum version of the Mukai vector\footnote 
{In fact, it coincides with the Mukai vector for K3 surfaces.}.  
In this paper, we extend the $\hGamma$-integral structure 
to FJRW theory. 
Our main results are stated as follows. 
We refer the reader to Theorems \ref{thm:main}, \ref{thm:big} 
for more precise statements.  

\begin{thm} 
\label{thm:overview_main} 
{\rm (i)} The ambient part quantum $D$-module
of $X_W$ 
and the narrow part quantum $D$-module  
of $(\CC^N, W, \bmu_d)$ are analytically continued 
to each other\footnote
{See \S \ref{subsubsec:QDM} for the definition of 
the quantum $D$-modules. We mean by ``ambient part" 
the cohomology classes pulled back from the ambient 
space, see \S \ref{subsubsec:GW-statespace}; 
in FJRW side, this has a counterpart called 
``narrow part", see \S \ref{subsubsec:FJRW-statespace}. 
}, 
i.e.\ both of them are the restrictions of a global $D$-module over the 
K\"{a}hler moduli space $\cM$.  

{\rm (ii)} The analytic continuation in {\rm (i)} 
matches up the $\hGamma$-integral structures  
on both quantum $D$-modules. 
Moreover, the induced isomorphism between the numerical $K$-groups 
of the categories of B-branes coincide with the one induced 
by the Orlov derived equivalence. 
\end{thm} 


A prototype of our result is the work of 
Borisov-Horja \cite{Borisov-Horja:FM}, where 
they showed that the analytic continuation of 
the GKZ hypergeometric system is induced from 
a Fourier-Mukai transformation between the $K$-groups 
of toric Calabi-Yau orbifolds, 
under a suitable identification of the spaces 
of local solutions with the $K$-groups. 
In our case, the GKZ system is replaced with the quantum 
$D$-modules of GW/FJRW theories and the Fourier-Mukai transformation 
is replaced with the Orlov equivalence. 

Since the global $D$-module over $\cM$ 
have nontrivial monodromies, the analytic continuation of 
flat sections depends on the choice of 
(a homotopy type of) a path. On the other hand,  
the Orlov equivalence $\Phi_l$ depends on an 
integer $l\in \ZZ$. 
The recent physics paper \cite{HHP} by Herbst-Hori-Page 
clarified (by a physical argument) 
the dependence of a derived equivalence on the choice of a path.   
We confirm the prediction of \cite{HHP} that a path $\gamma_l$ 
passing through the $l$th ``window" corresponds to 
the $l$th Orlov equivalence $\Phi_l$. 
Moreover, we check that the monodromy representation 
of the fundamental group of 
$\cM = \PP(1,d) \setminus \{\text{2 points}\}$ factors through 
the group of autoequivalences of $D^b(X_W)$. 
The following theorem refines Part (ii) of Theorem 
\ref{thm:overview_main}. 

\begin{thm}[Theorem \ref{thm:monodromyrep}] 
\label{thm:overview-autoeq}
{\rm (i)} For each integer $l\in \ZZ$, there exists a path 
$\gamma_l$ from a neighbourhood of the large radius limit point to 
a neighbourhood of the LG point 
such that the analytic continuation along $\gamma_l^{-1}$ 
is induced by the Orlov equivalence $\Phi_l$. 

{\rm (ii)} 
Let $N'(X_W)$ be the subgroup \eqref{eq:Nprime} of the 
numerical $K$-group of $X_W$ consisting of $K$-classes 
whose Chern characters lie in the ambient cohomology 
$H_{\amb}(X_W)$ 
and let $\chi$ be the Euler pairing. 
The monodromy representation of the 
global $D$-module in Theorem \ref{thm:overview_main}
\[
\rho \colon \pi_1(\cM) \to \Aut(N'(X_W),\chi)
\]
can be lifted to a group homomorphism 
\[
\hrho \colon \pi_1(\cM)
\to \Auteq(D^b(X_W))/[2], 
\] 
where $[2]$ is the $2$-shift functor. 
The homomorphism $\hrho$ sends a (clockwise) loop around the conifold 
point to the spherical twist by the structure sheaf. 
\end{thm} 

Since the work of Seidel-Thomas \cite{Seidel-Thomas}, 
the monodromy group action on $D^b(X)$ has been 
widely studied. 
Horja \cite{Ho} identified the conifold monodromy of the GKZ 
system with the spherical twist. 
Aspinwall \cite[\S 7.1.4]{Aspinwall:Dbrane} 
observed that the $5$th power of 
the monodromy around the LG point corresponds 
to the 2-shift (for a quintic). 
We deduce the existence of the lift $\hrho$ 
from a result of Canonaco-Karp \cite{Cano-Karp}. 
The above theorem suggests an autoequivalence 
group action on GW theory. 
However we do not know if $\hrho$ is injective. 
The induced homomorphism $\rho$ is never injective 
when $\dim X_W$ is even (since the conifold monodromy is 
involutive), but it is still possible that $\rho$ 
is injective for an odd-dimensional Calabi-Yau $X_W$. 

\subsection{Mirror symmetry} 
The interaction between B-branes and the A-model above 
can be explained most clearly via mirror symmetry. 
Here we consider Hodge-theoretic mirror symmetry, 
Kontsevich's homological mirror symmetry \cite{Kontsevich:ICM} 
and their mutual relationships. 
See also \cite{CR:LGconference} for the discussion 
on global mirror symmetry for finite group quotients 
of Calabi-Yau hypersurfaces. 

The mirror of $X_W$ is given by a certain Calabi-Yau 
compactification $Y_v$ (Batyrev's mirror \cite{Batyrev}) 
of the affine variety 
\[
\Yo_v := \{(\sfx_1,\dots,\sfx_N) \in (\CC^\times)^N 
\,|\, \sfx_1 + \cdots + \sfx_N=1, \ 
\sfx_1^{w_1} \sfx_2^{w_2} \cdots \sfx_N^{w_N} =v\} 
\]
where the parameter $v$ is identified 
with an inhomogeneous co-ordinate 
of $\cM = \PP(1,d) \setminus \{\text{2 points}\}$ 
such that $v=0$ is the large radius limit 
and that $v=\infty$ is the LG point. 
The mirror $Y_v$ may have Gorenstein terminal quotient 
singularities. 
Note that $\cM$ now plays a role of the complex moduli of $Y_v$. 
Under mirror symmetry, the category of B-branes should be 
equivalent to the category of A-branes of the mirror. 
Mathematically, the category of A-branes is the 
derived Fukaya category whose objects are (twisted complexes of) 
graded Lagrangian submanifolds.  
Likewise, the A-model theory should be equivalent 
the B-model theory of the mirror, which is, at genus zero,  
the variation of Hodge structure 
associated to the deformation of the complex structure.   
We get the mirror statements of \eqref{item:B-brane} 
and \eqref{item:A-model}. 
\begin{enumerate} 
\item[(\hypertarget{item:A-brane}{1'})] 
{\bf A-brane ``mirror LG/CY" correspondence:}
The derived Fukaya category of $Y_v$ is independent of $v\in \cM$. 

\item[\hypertarget{item:B-model}{(2')}] 
{\bf B-model ``mirror LG/CY" correspondence:} 
There exists a global variation of Hodge structure (VHS) 
$H^{N-2}(Y_v)=\bigoplus_{p+q=N-2} H^{p,q}(Y_v)$ over $\cM$. 
\end{enumerate} 
Because $Y_v$ does not change as a symplectic manifold 
(or orbifold) as $v$ varies, 
the Fukaya category should be independent of $v$ (if it is defined). 
The B-model VHS is tautologically ``analytically continued" over $\cM$. 
Moreover, the category of A-branes and the B-model 
have a natural integration pairing.  
Namely, one can integrate a de Rham cohomology class on $Y_v$ 
over a Lagrangian submanifold.  
By this pairing, an A-brane (a Lagrangian submanifold) 
gives rise to a dual flat section of the B-model VHS,  
i.e.\ a middle homology class in $H_{N-2}(Y_v,\ZZ)$ 
represented by the brane. 
This is exactly dual to the phenomenon we described in \S \ref{subsec:overview}. 

The $\hGamma$-integral structure in GW theory for $X_W$ 
is actually mirrored from the natural integral structure 
in the B-model of $Y_v$ 
(see also \cite[Conjecture 4.2.10]{CR:LGconference}). 
\begin{thm}[{\cite[Theorem 6.9]{Iritani:period}}]
The ambient A-model VHS of a Calabi-Yau hypersurface $X_W$ 
equipped with  the ambient $\hGamma$-integral structure 
is isomorphic to the residual B-model VHS of $Y_v$ 
equipped with the vanishing cycle integral structure 
near $v=0$ under the mirror map $\tau_{\sGW} \colon \{|v|<\epsilon\} 
\to H^2_{\amb}(X_W)/\langle G\rangle$ in Theorem \ref{thm:main}.  
\end{thm} 

Here, the ambient A-model VHS is the ambient part 
quantum $D$-module in Theorem \ref{thm:overview_main}
restricted to $z=1$ (see Remark \ref{rem:A-VHS}); 
the ambient $\hGamma$-integral structure is the $\ZZ$-local 
system consisting of flat sections associated 
to vector bundles on $X_W$ 
which are restricted from the ambient space $\PP(\uw)$. 
The residual B-model VHS is defined to be   
the pure part $\Gr^{\rsW}_{N-2} H^{N-2}(Y_v^\circ) \subset H^{N-2}(Y_v)$ 
of the Deligne mixed Hodge structure 
of the affine variety $Y_v^\circ$; 
the vanishing cycle integral structure on it 
is spanned by the Poincar\'{e} duals of vanishing cycles 
of the function $\sfx_1+ \cdots + \sfx_N$ on the torus 
$\{(\sfx_1,\dots,\sfx_N)\in (\CC^\times)^N\,|\, 
\prod_{i=1}^N \sfx_i^{w_i} = v\}$. 
See \cite{Iritani:period} for the details. 
Because $K$-classes of vector bundles restricted from $\PP(\uw)$ correspond, 
under Orlov equivalence, to the $K$-classes of 
graded Koszul matrix factorizations (Proposition \ref{pro:OrlovKoszul}), 
we have the following corollary 
(see also \cite[Conjecture 4.2.11]{CR:LGconference}). 

\begin{cor} 
The narrow A-model VHS of the Landau-Ginzburg model 
$(\CC^N,W,\bmu_d)$ equipped with the subsystem of the 
$\hGamma$-integral structure 
spanned by $K$-classes of graded Koszul matrix factorizations 
is isomorphic to the residual B-model VHS of $Y_v$ 
equipped with the vanishing cycle integral structure 
near $v=\infty$ under the mirror map 
$\tau_{\sFJRW} \colon 
\{|v|^{-1/d} <\epsilon\} \to H_{\nar}^2(W,\bmu_d)/\langle G\rangle$ 
in Theorem \ref{thm:main}.  
\end{cor} 

In particular, both the quantum $D$-module of $X_W$ 
and of $(\CC^N,W,\bmu_d)$ over the image of the mirror map 
give a polarized variation of $\ZZ$-Hodge structure. 

\subsection{Plan of the paper} 
In Section \ref{sec:setup}, we introduce 
the $\hGamma$-integral structure on the quantum $D$-module
associated to FJRW and GW theories. 
Then we state our main theorems in a precise way. 
In Section \ref{sect:computing}, we introduce twisted 
FJRW invariants and calculate 
the (twisted) $I$-function of FJRW theory. 
This provides the main ingredients of the paper. 
In Section \ref{sec:Orlov=MB}, we calculate 
the analytic continuation of the $I$-function 
and show that the connection matrix 
matches the Orlov equivalence. 
In Section \ref{sec:globalDmod}, we construct a 
global $D$-module over the K\"{a}hler moduli 
and prove the main theorems. 

\vspace{10pt} 
\noindent 
{\bf Acknowledgments.} 
We would like to thank the Institut Fourier, 
where a major portion of this work was done. 
A.C. is supported by the ANR grant ``Nouvelles sym{\'e}tries 
pour la th{\'e}orie de Gromov-Witten", ANR-09-JCJC-0104-01;  
he thanks Claire Voisin for a motivating question on the 
possibility of a direct proof of the LG/CY correspondence 
without passing through mirror symmetry 
and Jos{\'e} Bertin for dozens of conversations 
and for his notes on matrix factorizations. 
H.I.\ is supported by Grant-in-Aid 
for Young Scientists (B) 22740042; 
he thanks Ed Segal for useful discussions 
on matrix factorizations and Orlov equivalence. 
Y.R.\ is supported by a grant from NSF; his two visits 
to Grenoble were supported by the Institut Fourier 
and by a ``chaire d'excellence ENS Lyon/UJF Grenoble''. 

\vspace{10pt} 
\noindent 
{\bf Notation}

\vspace{10pt} 
\begin{tabular}{ll}
$W(x_1,\dots,x_N)$ & weighted homogeneous polynomial 
(\S \ref{subsec:overview}) \\
$X_W$ & Calabi-Yau hypersurface defined by $W$ in $\PP(\uw)$ 
(\S \ref{subsec:overview}) \\
$\bmu_d$ & the group of $d$-th roots of unity \\
$H$  & state space 
($H(W,\bmu_d)$ or $H_{\CR}(X_W)$, 
\S \ref{subsubsec:FJRW-statespace}, \S \ref{subsubsec:GW-statespace}) 
\\
$H'$ & narrow/ambient part 
(\S \ref{subsubsec:restr-narrow}) \\ 
$H''$ & broad/primitive part 
(\S \ref{subsubsec:restr-narrow}) \\ 
$\{t^i\}$ & linear co-ordinates on the state space associated to 
a basis $\{T_i\}$ \\ 
& (\S \ref{subsec:qcoh_conn}, 
\S \ref{subsec:qdm_int}, \S \ref{subsubsec:tw-quantum}) \\  
$\hGamma$ & Gamma class 
(\S \ref{subsubsec:intstr})  \\ 
$\ov{H}$  & state space of twisted theories  
($H_\ext$ or $H_\CR(\PP(\uw))$, \S \ref{subsubsec:tw-quantum}) \\ 
$\cM$  & global 
K\"{a}hler moduli space ($\PP(1,d) \setminus \{0,\vc\}$, 
\S \ref{subsec:mGKZ}) 
\end{tabular} 

\section{$\hGamma$-integral structure 
and main statements} 
\label{sec:setup} 

We briefly review FJRW (Fan-Jarvis-Ruan-Witten) theory for $(W,\bmu_d)$ 
and GW (Gromov-Witten) theory for $X_W$ 
and introduce the $\hGamma$-integral structure
on the quantum $D$-modules of both theories. 
Then we give a precise statement of the main results.

\subsection{FJRW theory} 
The FJRW invariants ``count" the number of 
solutions to a non-linear PDE, the so-called Witten equation. 
These define a cohomological field theory on the 
FJRW state space. 
In this paper, we restrict ourselves to the genus zero 
FJRW invariants from the ``narrow part".  
In this case, the Witten equation has only trivial 
solutions and the invariants  
reduce to intersection numbers of tautological 
classes on the moduli space of $d$-spin curves. 
For the details of the full FJRW theory, we refer the 
reader to the original articles \cite{FJR:virt, FJR1}. 

\subsubsection{State space}
\label{subsubsec:FJRW-statespace} 
Let $(\CC^N, W,\bmu_d)$ be the Landau-Ginzburg orbifold 
in the previous section \S \ref{subsec:overview}. 
Let $\zeta := \exp(2\pi\iu/d)\in \bmu_d$ denote 
a primitive $d$th root of unity.  
Let $(\CC^N)_k$ denote the $\zeta^k$-fixed 
subspace of $\CC^N$ 
and $W_k \colon (\CC^N)_k \to \CC$ denote 
the restriction of $W$. 
We also write $N_k = \dim_\CC (\CC^N)_k$. 
The FJRW state space $H(W,\bmu_d)$ is defined to be 
\[
H(W,\bmu_d) := 
\bigoplus_{k=0}^{d-1} H(W,\bmu_d)_k   
\]
where the sector $H(W,\bmu_d)_k$ associated to 
$\zeta^k \in \bmu_d$ is given by 
\begin{align*} 
H(W,\bmu_d)_k &:= H^{N_k} 
\left((\CC^N)_k, W_k^{+\infty};\CC 
\right)^{\bmu_d},  
\\
W_k^{\pm \infty} &:= \{ x \in (\CC^N)_k \,:\, 
\pm \Re(W_k(x)) \gg 0\}.  
\end{align*} 
The degree of an element $\phi\in H(W,\bmu_d)_k$ is 
defined to be 
\begin{equation} 
\label{eq:deg-FJRW} 
\deg \phi := N_k + 
2\sum_{i=1}^N \langle k q_i \rangle  -2 
\end{equation} 
where $q_i := w_i/d$. This is an integer since $\sum_{i=1}^N q_i =1$. 
Let $\pair{\cdot}{\cdot}$ denote the 
natural intersection pairing 
\begin{equation} 
\label{eq:intersection-pairing} 
\pair{\cdot}{\cdot} \colon 
H^{N_k}\left((\CC^N)_k, W_k^{+\infty};\CC\right)
\times 
H^{N_k}\left((\CC^N)_k, W_k^{-\infty};\CC\right) 
\to \CC 
\end{equation} 
and $I\colon \CC^N \to \CC^N$ denote the map 
$(x_1,\dots, x_N) \mapsto (\tilde\zeta^{w_1} x_1,\dots, 
\tilde\zeta^{w_N} x_N)$ for $\tilde\zeta = \exp(\pi\iu/d)$. 
Because $W(I(x)) = - W(x)$, we have a map 
\begin{equation} 
\label{eq:Istar} 
I^* \colon H(W,\bmu_d)_{d-k} \cong 
H^{N_k}\left((\CC^N)_k, W_k^{+\infty};\CC\right)^{\bmu_d}
\to 
H^{N_k}\left((\CC^N)_k, W_k^{-\infty};\CC\right)^{\bmu_d}.   
\end{equation} 
We define the pairing between $\alpha_1 \in H(W,\bmu_d)_k$ and 
$\alpha_2 \in H(W,\bmu_d)_{d-k}$ by 
\begin{equation} 
\label{eq:FJRW-pairing}
(\alpha_1,\alpha_2) := \frac{1}{d} 
\pair{\alpha_1}{I^* \alpha_2}. 
\end{equation} 
Setting $(\alpha_1,\alpha_2)=0$ for 
$\alpha_1\in H(W,\bmu_d)_k$, $\alpha_2 \in H(W,\bmu_d)_l$ with 
$k+l \neq d$, we obtain a graded symmetric 
non-degenerate pairing $(\cdot,\cdot)$ on the 
state space $H(W,\bmu_d)$. 
The pairing in this paper differs from that in 
\cite{FJR1} by the factor $1/d = 1/|\bmu_d|$. 
See Appendix \ref{sec:factorofd} for this convention. 

We say that a sector $H(W,\bmu_d)_k$ is 
\emph{narrow}
if $(\CC^N)_k =\{0\}$ and 
\emph{broad} otherwise\footnote
{Fan-Jarvis-Ruan originally called these sectors 
``Neveu-Schwarz" and ``Ramond" respectively, but they 
changed the names later.}. 
Each narrow sector $H(W,\bmu_d)_k$ is one-dimensional 
and we denote by $\phi_{k-1}\in H(W,\bmu_d)_k$ 
the standard basis given as 
the identity class on $(\CC^N)_k=\{0\}$\footnote
{Note the shift of the index $k$ by one. An element $\phi_k$ with 
$k+1\notin \Nar$ will be introduced later in \S \ref{sect:computing} 
in the context of ``extended theory", 
but notice that $\phi_k$ with $k+1\notin \Nar$ 
is \emph{not} an element of the original FJRW state space.}. 
We set 
\[
\Nar := \{0\le k\le d-1\,:\, (\CC^N)_k =\{0\} \ 
\text{i.e.~$kq_i\notin \ZZ$ for all $i$}\}  
\]
and define the \emph{narrow part} as 
\[
H_{\nar}(W,\bmu_d) = \bigoplus_{k\in \Nar} H(W,\bmu_d)_k 
= \bigoplus_{k\in \Nar} \CC \phi_{k-1}. 
\]
Note that $\deg \phi_k = 2\sum_{j=1}^N \fracp{kq_j}$ for 
$k+1 \in \Nar$.  
The degree zero element $\phi_0 \in H(W,\bmu_d)_1$ plays the role of 
the identity in the FJRW theory. 
The pairing $(\cdot,\cdot)$ 
restricts to a non-degenerate pairing on $H_{\nar}(W,\bmu_d)$ 
\begin{equation} 
\label{eq:pairing_narrow} 
(\phi_k,\phi_l) = \frac{1}{d} \delta_{d, (k+1)+(l+1)}, \quad 
k+1, l+1 \in \Nar  
\end{equation} 
and $H_{\nar}(W,\bmu_d)$ is orthogonal to the  
\emph{broad part}
$H_{\bro}(W,\bmu_d):=\bigoplus_{k\notin \Nar} H(W,\bmu_d)_k$.

For a polynomial $f$ on $\CC^n$, we define the 
\emph{Jacobi space}\footnote
{It is isomorphic to the Jacobi ring $\CC[y_1,\dots,y_n]/
(\partial_1 f, \dots, \partial_n f)$, 
but notice that $\Omega(f)$ is not a ring.} 
$\Omega(f)$ by 
\begin{equation} 
\label{eq:Jacobiring} 
\Omega(f) := 
\Omega^{n}_{\CC^n} \big/ \dd f \wedge \Omega^{n-1}_{\CC^n}, 
\end{equation} 
where $\Omega^k_{\CC^n}$ denotes the space of algebraic  
$k$-forms on $\CC^n$. 
When $f$ has an isolated critical point at the origin, we 
have the Grothendieck residue pairing $\Res_{f} 
\colon \Omega(f) \otimes \Omega(f) \to 
\CC$ (see \cite{Griffiths-Harris}): 
\[
\Res_{f}\left( [a(y) \dd y], 
[b(y) \dd y] \right) 
:= \Res \left[ 
\frac{a(y) b(y) \dd y}{\partial_1 f, \dots,\partial_n f}
\right],  
\]
where $y=(y_1,\dots,y_n)$ is a co-ordinate system on $\CC^n$ and 
$\dd y = \dd y_1\wedge \cdots \wedge \dd y_n$. 
The residue pairing is independent 
of the choice of co-ordinates. 
We shall consider the Jacobi space $\Omega(W_k)$ 
associated to the homogeneous 
polynomial $W_k$ on $(\CC^N)_k$. 
The space $\Omega(W_k)$ is graded by the usual degree:  
$\deg x_j = \deg \dd x_j = w_j$. 
\begin{pro} 
\label{pro:FJRWstatespace_Jacobi} 
We have a canonical isomorphism 
\begin{equation} 
\label{eq:FJRWstatespace_Jacobi}
H(W,\bmu_d)_k \cong \, \Omega(W_k)^{\bmu_d}.  
\end{equation} 
Under this isomorphism, the pairing $(\cdot,\cdot) \colon 
H(W,\bmu_d)_k\times H(W,\bmu_d)_{d-k} \to \CC$ 
translates into the Grothendieck residue pairing 
between $\Omega(W_k)^{\bmu_d}$ and $\Omega(W_{d-k})^{\bmu_d} 
\cong \Omega(W_k)^{\bmu_d}$:  
\[
[\varphi] \otimes [\psi] \longmapsto  
(-1)^{\frac{N_k(N_k-1)}{2}} 
(2\pi\iu)^{N_k} \frac{1}{d}  
\Res_{W_k}\left( [\varphi], (-1)^{|\psi|} [\psi] \right), 
\]
where $|\psi|$ is the degree of $\psi\in \Omega(W_k)^{\bmu_d}$ 
divided by $d$. 
Notice that the above isomorphism does not match up the grading 
on the FJRW state space $H(W,\bmu_d)_k$ (which is homogeneous 
of degree $N_k + 2 (\sum_{j=1}^N \fracp{kq_j} -1)$)  
and the grading on the Jacobi space $\Omega(W_k)^{\bmu_d}$. 
\end{pro} 
The isomorphism \eqref{eq:FJRWstatespace_Jacobi} 
is given by the Hodge decomposition \eqref{eq:Hodge-decomp-Jacobi}. 
See Appendix \ref{sec:proof-Jacobi} for the proof. 
This description is used in \S \ref{subsubsec:HRR} 
(and in \S \ref{subsubsec:intstr})  
to discuss the Chern character and 
Riemann-Roch for matrix factorizations. 

\subsubsection{Moduli of W-curves}
In this paper all stacks are defined over $\CC$. 
By a \emph{pointed orbicurve}, we mean a proper 
and connected one-dimensional Deligne-Mumford stack 
$C$ which has only nodes as singularities, which is equipped 
with distinct marked points $\sigma_1,\dots,\sigma_n$ 
on the smooth locus, and which has possibly non-trivial 
stabilizers only at the marked points and the nodes. 
For a positive integer $d$, a pointed orbicurve $C$ 
is called \emph{$d$-stable} \cite{Chistab} 
if the associated pointed coarse curve 
$|C|$ is stable and if all the stabilizers at the nodes and 
the marked points are isomorphic to $\bmu_d$. 
We always assume that every node of an orbicurve 
is balanced \cite{Abramovich-Vistoli}, 
i.e.~formally locally near a node, the curve is 
isomorphic to $[\Spec (\CC[x,y]/xy)/\bmu_d]$, 
where the action of $\bmu_d$ is given by 
$(x,y) \mapsto (\zeta x, \zeta^{-1} y)$. 
For a pointed orbicurve $(C,\sigma_1,\dots,\sigma_n)$, 
we introduce an invertible sheaf $\omega_{\log}$ on $C$ 
which may be regarded as the dualizing sheaf
$\omega_C$ twisted at the stacky markings $\sigma_1, \dots \sigma_n$ 
(each of which is of degree $1/d$) 
or equivalently as the pullback of the dualizing sheaf 
$\omega_{|C|}$ twisted at the coarse markings 
$|\sigma_1|, \dots |\sigma_n|$ (each of which is of 
degree one) 
 \[
\omega_{\log} = \omega_{C}(\sigma_1 + \cdots + \sigma_n)=
q^* \left(\omega_{|C|}(|\sigma_1| + \cdots + |\sigma_n|) \right)
\]
where $q\colon C \to |C|$ is the natural map. 
In other words, $\omega_{\log}$ is the sheaf of 
logarithmic differential forms on $C$ with poles 
only at marked points and 
nodes, and such that 
the sum of the residues at each node is zero. 

A \emph{$d$-spin structure} on a pointed orbicurve $C$ 
is a line bundle $L \to C$ 
together with an isomorphism 
$\varphi\colon L^{\otimes d} \cong \omega_{\log}$. 
Write $W(x_1,\dots,x_N) = \sum_{i=1}^l c_i \prod_{j=1}^N x_j^{m_{ij}}$, 
where $\prod_{j=1}^N x_j^{m_{ij}}$, $i=1,\dots,l$ 
are mutually distinct monomials and $c_i\neq 0$. 
A \emph{$W$-structure} on a pointed orbicurve 
$C$ is a collection of line bundles 
$L_1,\dots,L_N$ (corresponding to the variables $x_1,\dots,x_N$) 
on $C$ together with isomorphisms 
\begin{equation} 
\label{eq:W-str}
\varphi_i \colon 
\bigotimes_{j=1}^N L_j^{\otimes m_{ij}} \cong \omega_{\log}, 
\quad i=1,\dots, l. 
\end{equation} 
This generalizes the notion of a $d$-spin structure 
(see Remark \ref{rem:Wstr-origin}).  
Since $W$ is weighted homogeneous of degree $d$, 
a $d$-spin structure $L\to C$ gives rise to a 
$W$-structure by setting 
\[
L_i = L^{\otimes w_i}, \quad i=1,\dots, N.  
\] 
A $W$-structure does not necessarily arise from a 
$d$-spin structure in this way. In this paper, however,  
we restrict our attention to a $W$-structure coming from a 
$d$-spin structure\footnote
{This is because we are interested in a \emph{generic}  
weighted homogeneous polynomial $W$. More precisely, 
we can add to $W$ a weighted homogeneous \emph{Laurent} 
polynomial $Z$ so that the group of diagonal symmetries 
preserving $W+Z$ is exactly $\bmu_d$; 
then every $(W+Z)$-structure comes from a $d$-spin structure. 
This means that the group $\bmu_d$ is \emph{admissible} 
in the sense of \cite[\S 2.3]{FJR1}.}.  

Let $L$ be a $d$-spin structure on a $d$-stable 
pointed orbicurve $C$. 
The stabilizer $\bmu_d$ at a marked point $\sigma$ 
acts on the fibre $L_\sigma$ via a homomorphism $\bmu_d \to \CC^\times$, 
which is of the form $t\mapsto t^k$ for a unique 
$0\le k<d$. 
We call the rational number $\age_\sigma(L) := k/d 
\in [0,1)$ the \emph{age} of $L$ at $\sigma$. 
The generator $\zeta\in \bmu_d$ acts on 
the fibre of the associated $W$-structure 
$(L^{\otimes w_1},\dots,L^{\otimes w_N})$ at $\sigma$ 
by $(\zeta^{kw_1},\dots,\zeta^{k w_N})$; 
hence in this case we regard the marked point $\sigma$ 
as corresponding to the sector $H(W,\bmu_d)_k$. 
For $0\le k_1,\dots,k_n\le d-1$, 
let $\Spin^d_{0,n}(k_1,\dots,k_n)$ denote the moduli stack of 
$d$-stable orbicurves $C$ of genus zero and 
with $n$ marked points $\sigma_1,\dots,\sigma_n$ 
endowed with a $d$-spin structure $L\to C$ 
such that $\age_{\sigma_i}(L) = \fracp{(k_i+1)/d}$. 
\begin{equation}
\label{eq:Spin}
\Spin^d_{0,n} (k_1,\dots,k_n) = 
\left\{\left(C;\sigma_1,\dots,\sigma_n; L;
\varphi \colon L^{\otimes d}
\cong \omega_{\log}\right) \, \Big|\, 
\age_{\sigma_i}(L) =\fracp{\tfrac{k_i+1}{d}} \right\}\Big/ 
\text{isom}.  
\end{equation} 
See \cite[Appendix]{ChiRquintic} for a precise 
definition of $\Spin^d_{0,n}$ as a fibred category, 
where it is denoted by $\cR_d$.  
More precisely the substack denoted by
$\cR_d(e^{2\pi \iu \Theta_1},\dots e^{2\pi \iu \Theta_n})$
in \cite{ChiRquintic}
corresponds to \eqref{eq:Spin} as soon as $\Theta_j=(k_j+1)/d$;
the present choice of notation allows a more straightforward
formula for the  FJRW invariants, see
\eqref{eq:narrow_FJRWinv} and the following footnote
\ref{footnote:reltoChiR}. 

The moduli stack $\Spin^d_{0,n}(k_1,\dots,k_n)$ 
is smooth, proper and of Deligne-Mumford type \cite{Chistab}. 
It is non-empty if and only if 
$\chi(L) = \deg L - \sum_{i=1}^n \age_{\sigma_i}(L) -1 = 
(n-2)/d - \sum_{i=1}^n \fracp{(k_i+1)/d}-1$
is an integer (see \cite{AGV,Kawasaki,Toen} for Riemann-Roch for orbicurves). 
When non-emtpy, it is of dimension $n-3$. 
(It differs from the Deligne-Mumford space $\ov{\cM}_{0,n}$ 
of stable curves only because of the stabilizers.) 
It is clear that the definition \eqref{eq:Spin} extends
{verbatim} to higher genera; we write
$\Spin^d_{g,n}(k_1,\dots,k_n)$ for the similar moduli 
space at genus $g$. 
We use it only in \S\ref{subsubsect:fullFJRW} 
where we recall the general setup of \cite{FJR1}.

\begin{rem} 
\label{rem:Wstr-origin} 
The definition \eqref{eq:W-str} of a $W$-structure 
originates from the Witten equation, 
which is a system of PDE for 
sections $s_i \in C^\infty(C,L_i)$, 
$i=1,\dots, N$: 
\begin{equation}\label{eq:WittenPDE}
\ov{\partial} s_j + \ov{\partial_j W(s_1,\dots,s_N)} = 0. 
\end{equation} 
The equation makes sense under \eqref{eq:W-str}  
and a suitable choice of a Hermitian metric on $L_i$ 
(see \cite{FJR:virt}). 
When all the marked points correspond to the narrow sector, 
the zero sections are only possible solutions to 
the Witten equation 
\cite[Theorem 3.3.8]{FJR:virt} 
and the FJRW invariant is the Euler class of the obstruction 
bundle over the moduli space of $W$-curves. We briefly review the general case.
\end{rem}

\subsubsection{FJRW invariants} \label{subsubsect:fullFJRW}
We review the virtual fundamental class defining 
the invariants of ``full" FJRW theory following \cite{FJR1}. 
The general FJRW invariants are defined by an analytic moduli 
space of solutions to the Witten equation \eqref{eq:WittenPDE}. 
We shall see in the next section \S \ref{subsubsec:narrowFJRW} 
that the narrow sector invariants have an algebro-geometric definition. 
We stress that the invariants introduced later in \S \ref{subsubsect:extending} 
are different from the invariants of the full theory; they are a 
natural and computable extension of the narrow sector invariants. 


The virtual cycle\footnote{
$[\Spin_{g,n}^d(k_1,\dots,k_n)]^{\rm vir}$ is denoted by  
$[\ol{\mathcal{W}}_{g,n,\bmu_d}(W,(k_1,\dots,k_n))]^{\rm vir}$ 
in \cite{FJR1}.} 
\[
\bigoplus_{0\le k_1,\dots,k_n \le d-1} 
[\Spin_{g,n}^d(k_1,\dots,k_n)]^{\rm vir}
\] 
of the fully developed FJRW theory is defined in all genera and lies in
\[
\bigoplus_{0\le k_1,\dots, k_n\le d-1} \left(  
H_{2D(k_1,\dots,k_n)}
(\Spin^d_{g,n}(k_1,\dots,k_n);\CC)\otimes \bigotimes_{i=1}^n
H_{N_{k_i}} 
\left((\CC^N)_{k_i}, W_{k_i}^{+\infty};\CC\right)^{\bmu_d}\right),
\]
where $D(k_1,\dots,k_n)=(3g-3+n)+\sum_{j=1}^N\chi(L^{\otimes w_j})$ 
for a $d$-spin structure $L$ from $\Spin^d_{g,n}(k_1,\dots,k_n)$. 
Regard the relative homology group 
$H_{N_{k}}((\CC^N)_k, W_k^{+\infty};\CC)^{\bmu_d}$ above 
as the dual of 
the sector $H(W,\bmu_d)_k$ of the state space. 
Then the  
virtual cycle defines a linear operator (see \cite[Eqn (63)]{FJR1}): 
\begin{align*}
H(W,\bmu_d)^{\otimes n}
\xrightarrow{\ \ [\Spin_{g,n}^d (k_1,\dots k_n)]^{\vir}\cap\ \ } 
H_{2D(k_1,\dots,k_n)}(\Spin^d_{g,n}(k_1,\dots,k_n);\CC)
\end{align*}
 assigning to $\otimes_{i=1}^n \alpha_i$ a nonzero cycle only if 
 $\alpha_i$ lies in   
$H^{N_{k_i}} 
((\CC^N)_{k_i}, W_{k_i}^{+\infty};\CC)^{\bmu_d}$. 
In FJRW theory, such a cycle is 
pushed forward via the natural forgetful morphism 
\[
\mathrm{st} \colon \Spin^d_{g,n}(k_1,...,k_n) \to 
\ov{\mathcal M}_{g,n}
\] 
forgetting the data $L$ and $\varphi$ and 
passing to the the coarse stable curve corresponding to 
$(C; \sigma_1,\dots,\sigma_n)$. 
After Poincare duality and multiplication by a factor 
$f_g = |G|^g/\deg({\rm st}) = d/g$, 
Fan, Jarvis and Ruan obtain a cohomological field theory. 
We show in Appendix \ref{sec:factorofd} that 
the multiplication by $f_g$ can be 
removed in all genera once we use the natural pairing 
\eqref{eq:FJRW-pairing} from Chen-Ruan cohomology. 
In genus zero, this simply 
amounts to the fact that our invariants are $1/f_0=1/d$ times the 
FJRW genus-zero invariants of \cite{FJR1}. 
By making explicit the cohomological field theory of \cite[Definition 4.2.1]{FJR1} 
\emph{after removing} $f_g$ 
and by applying the definition of the correlators of \cite[Definition 4.2.6]{FJR1},
we obtain 
\begin{equation}\label{eq:fullFJRWinv}
\corr{\tau_{b_1} (\alpha_1),\dots, \tau_{b_n} (\alpha_n)}_{g,n}^{\FJRW} 
 := \left([\Spin_{g,n}^d(k_1,\dots k_n)]^{\vir}\cap 
\prod_{i=1}^n \alpha_i \right)\cap \prod_{i=1}^{n} \psi_i^{b_i}
\end{equation}		
for $\alpha_i\in H^{N_{k_i}} 
((\CC^N)_{k_i}, W_{k_i}^{+\infty};\CC)^{\bmu_d}$ and 
$b_i\ge 0$. 
The class $\psi_i$ is the first Chern class of the 
line bundle on $\Spin^d_{g,n}(k_1,\dots,k_n)$ 
whose fibre at a point $(C;\sigma_1,\dots,\sigma_n;L;\varphi)$ 
is the cotangent space $T_{\sigma_i}^*|C|$ 
of the coarse curve $|C|$ at $\sigma_i$. 
Note that, since these classes 
are the pullback of the standard $\psi$ classes 
from $\ol{\mathcal M}_{g,n}$, we can directly integrate
on $\Spin^d_{g,n}(k_1,\dots,k_n)$ 
short-circuiting the pushforard 
$\rm{st}_*$ of \cite[Definition 4.2.1]{FJR1}. 

\subsubsection{Narrow part FJRW invariants} 
\label{subsubsec:narrowFJRW}
The definition, further simplifies in genus zero 
and when the entries are narrow states $\phi_k$ with $k+1\in \Nar$.
Let $\pi \colon \cC \to \Spin^d_{0,n}(k_1,\dots,k_n)$ be the 
universal orbicurve and $\cL\to \cC$ be the universal $d$-spin structure. 
If all the entries are narrow states, 
then the following lemma shows that  
 $H^0(C, L^{\otimes w_j})$ vanishes all $j$ 
(and for every curve $C$); 
hence $-\RR\pi_* \bigoplus_{j=1}^N \cL^{\otimes w_j}
= R^1\pi_* \bigoplus_{j=1}^N \cL^{\otimes w_j}$
is a vector bundle which we refer to as the \emph{obstruction 
bundle}. 
In these cases the virtual homology cycle is simply 
the Poincar\'e dual cycle of the top Chern class 
of the obstruction bundle 
\cite[Theorem 4.1.8, Concavity]{FJR1}. 

\begin{lem} \label{lem:concave}
Suppose that $k_i+1 \in \Nar$ for all $i$.
Then $H^0(C,L^{\otimes w_j})$ vanishes for all $j=1,\dots, N$ for  
$(C;\sigma_1,\dots,\sigma_n;L;\varphi)\in 
\Spin^d_{0,n}(k_1,\dots,k_n)$. 
In particular, the obstruction bundle  
$\bigoplus_{j=1}^N R^1\pi_* (\cL^{\otimes w_j})$ 
over $\Spin^d_{0,n}(k_1,\dots,k_n)$ 
is locally free of rank $2 - N + \sum_{i=1}^n \sum_{j=1}^N \fracp{k_i q_j}$. 
\end{lem} 
\begin{proof} 
Let $p\colon C \to \ol{C}$ denote the map forgetting the 
stack-theoretic structures at all the markings 
(but not at the nodes). 
Then 
\[
p^* p_* (L^{\otimes w_j}) = L^{\otimes w_j} \otimes 
\cO_C\left( - \sum_{i=1}^n 
d \age_{\sigma_i} (\cL^{\otimes w_j}) \sigma_i \right) 
= L^{\otimes w_j}\otimes \cO_C\left(
- \sum_{i=1}^n d \fracp{(k_i+1) q_j} \sigma_i \right). 
\]  
Here we regard $\sigma_i$ as a stacky divisor of degree $1/d$. 
Hence 
\begin{align*} 
L' := (p^* p_* (L^{\otimes w_j}))^{\otimes (q_j^{-1})}  
& = \omega_{\rm \log}\otimes \cO_C\left(
-\sum_{i=1}^n d q_j^{-1} \fracp{(k_i+1) q_j}\sigma_i 
\right). 
\end{align*} 
Notice that $q_j^{-1} = d/w_j$ is an integer by the Gorenstein condition 
and $q_j^{-1} \fracp{(k_i+1) q_j} \ge 1$ since $k_i+1 \in \Nar$. 
Hence $L'$ is a subsheaf of the pull-back of $\omega_{|C|}$. 
Note that $H^0(|C|,\omega_{|C|})=0$ since $|C|$ is of genus zero. 
Hence $H^0(C, L')=0$ 
and thus $H^0(C,p^* p_* (L^{\otimes w_j})) =0$. 
This implies $H^0(C, L^{\otimes w_j}) =0$. 
Finally, by Riemann-Roch \cite{AGV, Kawasaki, Toen}, 
$R^1 \pi_* (\cL^{\otimes w_j})$ is locally free of rank 
\begin{align*} 
-\chi(L^{\otimes w_j}) & = 
-1 - \deg L^{\otimes w_j} + \sum_{i=1}^n \age_{\sigma_i}(L^{\otimes w_j}) 
=  
-1 - (n-2)q_j + \sum_{i=1}^n  \fracp{(k_i+1)q_j} \\ 
& = -1 + 2 q_j + \sum_{i=1}^n \fracp{q_j k_i}. 
\end{align*} 
The conclusion follows. 
\end{proof}
\begin{rem} 
The above lemma usually fails without the Gorenstein condition. 
This is the main reason that we restrict ourselves to 
the Gorenstein case. 
\end{rem} 
We can now specialize \eqref{eq:fullFJRWinv} and get
the narrow part descendant FJRW invariants as 
\begin{equation} 
\label{eq:narrow_FJRWinv} 
\corr{\tau_{b_1} (\phi_{k_1}),\dots, \tau_{b_n} (\phi_{k_n})
}_{0,n}^{\FJRW} 
= \int_{[\Spin^d_{0,n}(k_1,\dots,k_n)]} 
\prod_{i=1}^n \psi_i^{b_i} \cup 
\prod_{j=1}^N c_{\text{top}}
\left
(R^1\pi_* (\cL^{\otimes w_j})
\right),   
\end{equation} 
where $\phi_{k_1},\dots, \phi_{k_n} \in H_{\nar}(W,\bmu_d)$ 
(i.e.\ $k_i+1 \in \Nar$) and $b_1,\dots, b_n\ge 0$. 
We sometimes omit $\tau_0$ from the notation, e.g.\ 
writing $\corr{\phi_{k_1},\dots,\phi_{k_n}}^{\FJRW}_{0,n}$ 
for $\corr{\tau_0(\phi_{k_1}),\dots,\tau_0(\phi_{k_n})}^{\FJRW}_{0,n}$. 
The FJRW invariants satisfy the homogeneity 
(\cite[Dimension Axiom in \S 4.1]{FJR1}) 
\begin{equation} 
\label{eq:homogeneity_FJRW} 
\corr{\tau_{b_1} (\alpha_1),\dots, 
\tau_{b_n}(\alpha_n)}_{0,n}^{\FJRW}=0 
\ \text{ unless } \ 
\sum_{i=1}^n (b_i + \frac{1}{2}\deg \alpha_i) = n + \hc -3,  
\end{equation} 
where $\hc:=N-2$ is the dimension of the Calabi-Yau hypersurface 
$X_W$. 
Again the invariants \eqref{eq:narrow_FJRWinv} differ 
from the original definition\footnote
{In \cite[\S3.1, (15)]{ChiRquintic}, where the case $d=5$ was 
discussed, the invariant
$\corr{\tau_{b_1} (\phi_{k_1}),\dots, \tau_{b_n} (\phi_{k_n})
}_{0,n}^{\FJRW}$ was defined to be 
$5 \left(\prod_{i=1}^n \psi_i^{b_i} \cup c_{\text{top}}
(R^1\pi_* (\cL))^5  \right) 
\cap \left[\cR_5(e^{2\pi \iu (k_1+1)/5},\dots,e^{2\pi \iu (k_n+1)/5})
\right]$; 
see also \cite[\S2.3, (14)]{ChiRquintic}. 
\label{footnote:reltoChiR}} \cite{FJR1, ChiRquintic} by the 
factor of $1/d$. See Appendix \ref{sec:factorofd}. 

\begin{rem} 
Polishchuk and Vaintrob \cite{PV:MF-cohFT} 
recently constructed a purely algebraic 
cohomological field theory of singularities based 
on matrix factorizations. 
They constructed a fundamental matrix factorization 
on the moduli space which plays the 
same role as the virtual fundamental class. 
The role of matrix factorizations in our paper 
is different from theirs, but 
it would be interesting to study the relationships. 
\end{rem}

\subsection{GW theory} 
GW theory for orbifolds has been developed 
by Chen-Ruan \cite{Chen-Ruan:GW} in symplectic category 
and Abramovich-Graber-Vistoli \cite{AGV} in algebraic 
category. We will work in the algebraic category. 
For the details, we again refer the reader to 
these original articles. 

\subsubsection{State space} 
\label{subsubsec:GW-statespace} 
The state space of orbifold GW theory 
is given by the \emph{Chen-Ruan cohomology group} 
of the orbifold. We explain the case of the Calabi-Yau 
hypersurface $X_W \subset \PP(\uw)$. 
Set  
\begin{align*} 
\frF &:= 
\{0\le f<1\, |\, f w_j \in \ZZ \text{ for some $1\le j\le N$}\} \\   
& = \{ 0\le f<1 \, |\, fd \in \ZZ, \ 
fd \notin \Nar\}.  
\end{align*} 
In the second line we used the Gorenstein condition 
(i.e.\ $w_j \mid d$).  
An element $f\in \frF$ gives rise to the stabilizer 
$\exp(2\pi\iu f)$ along the substack 
$\PP(\uw)_f\subset \PP(\uw)$: 
\[
\PP(\uw)_f := \left [ \left((\CC^N)_{fd}\setminus \{0\}\right)   
\Big/ \CC^\times \right] 
\]
where recall that $(\CC^N)_k$ is the subspace of $\CC^N$ 
fixed by $\zeta^k=\exp(2\pi\iu k/d)$. 
The \emph{inertia stacks} $\cI \PP(\uw)$, $\cI X_W$ are defined to be 
\[
\cI \PP(\uw) = \bigsqcup_{f\in \frF} \PP(\uw)_f, \quad 
\cI X_W = \bigsqcup_{f\in \frF} (\PP(\uw)_f \cap X_W). 
\]
The Chen-Ruan cohomology $H_{\CR}(X_W)$ 
is defined to be the cohomology of the inertia stack: 
\[
H_{\CR}(X_W) := H(\cI X_W;\CC) = 
\bigoplus_{f\in\frF} H(\PP(\uw)_f \cap X_W;\CC).  
\]
The degree of $\alpha\in H^{k}(\PP(\uw)_f\cap X_W)$, 
as an element of $H_{\CR}(X_W)$, 
is defined to be 
\begin{equation*} 
\deg \alpha = k + 2 \sum_{j=1}^N \fracp{fw_j}.  
\end{equation*} 
Note that this is an integer since $\sum_{j=1}^N f w_j = f d\in \ZZ$. 
Let $\inv \colon \PP(\uw)_f \cong \PP(\uw)_{\fracp{1-f}}$ 
denote the natural isomorphism.  
For $\alpha_1 \in H(\PP(\uw)_f \cap X_W)$ 
and $\alpha_2 \in H(\PP(\uw)_{\fracp{1-f}} \cap X_W)$, 
we define 
\begin{equation} 
\label{eq:GW-pairing} 
(\alpha_1,\alpha_2) = \int_{\PP(\uw)_f \cap X_W} 
\alpha_1 \cup \inv^* \alpha_2.  
\end{equation} 
We set $(\alpha_1,\alpha_2)=0$ if 
$\alpha_1\in H(\PP(\uw)_{f_1}\cap X_W)$ 
and $\alpha_2 \in H(\PP(\uw)_{f_2}\cap X_W)$ 
and $f_1+f_2\notin \ZZ$. 
Then $(\cdot,\cdot)$ defines a graded symmetric 
non-degenerate pairing on $H_{\CR}(X_W)$. 

The \emph{ambient part} $H_{\amb}(X_W)$ is defined to be 
the image of the restriction map 
\[
H_{\amb}(X_W) := \Image \left(
i^* \colon H_{\CR}(\PP(\uw)) = H(\cI\PP(\uw)) 
\to H(\cI X_W) = H_{\CR}(X_W) \right). 
\]
Let $\unit_f\in H(\PP(\uw)_f \cap X_W)$ 
denote the identity class on $\PP(\uw)_f \cap X_W$ 
and $p=c_1(\cO(1))$ denote the hyperplane class on $\PP(\uw)$. 
The ambient part is spanned, as a $\CC$-vector space, 
by $p^i \unit_f$, $0\le i\le  
\sharp\{1\le j\le N\,|\,fw_j\in \ZZ\}-1$, $f\in\frF$. 
The pairing $(\cdot,\cdot)$ restricts to a non-degenerate 
pairing on $H_{\amb}(X_W)$ and $H_{\amb}(X_W)$ is 
orthogonal to the complementary 
\emph{primitive part} 
$H_{\pri}(X_W) := \Ker(i_* \colon H_{\CR}(X_W) \to H_{\CR}(\PP(\uw)))$.

\begin{rem}[Comparison of state spaces] 
The FJRW and GW state spaces 
can be identified, up to Tate twist, with 
the \emph{relative} Chen-Ruan cohomology 
(Chiodo-Nagel \cite{CN:CICY}): 
\begin{align*} 
H(W,\bmu_d) & = H_{\CR}\left(
[\CC^N/\bmu_d], [W^{-1}(1)/\bmu_d]\right) \\ 
H_{\CR}(X_W) & = H_{\CR}(
\cO_{\PP(\uw)}(-d), \tW^{-1}(1)). 
\end{align*} 
The first identification follows immediately from the definition. 
The second identification follows from the Thom isomorphism.  
Here $\tW\colon \cO_{\PP(\uw)}(-d) \to \CC$ is the 
function in \S \ref{subsec:overview}. 
Note that the pairs $(\cO_{\PP(\uw)}(-d), \tW^{-1}(1))$, 
$([\CC^N/\bmu_d], [W^{-1}(1)/\bmu_d])$ are connected
by a variation of GIT quotients. 
Chiodo-Ruan \cite{ChiR} showed that there 
exists a bigraded isomorphism 
\[
H^{p,q}(W,\bmu_d) \cong H^{p,q}_{\CR}(X_W)
\]
which preserves the pairings on both sides. 
In this paper, we will construct a graded isomorphism 
(preserving the pairing) 
\[
H^{p,p}_{\nar}(W,\bmu_d) \cong 
H^{p,p}_{\amb}(X_W) 
\]
which \emph{depends} on a point of the K\"{a}hler moduli space. 
(Note that the narrow/ambient part has no $(p,q)$-Hodge component 
with $p\neq q$.) See Remark \ref{rem:gradedisom}. 
\end{rem} 

\subsubsection{GW invariants} 
For $n\ge 0$ and $\beta \in H_2(|X_W|,\ZZ)$, 
let $(X_W)_{0,n,\beta}$ denote the moduli stack 
of genus zero, $n$-pointed, degree $\beta$ 
stable maps to $X_W$ (it is 
denoted by $\mathcal{K}_{0,n}(X_W,\beta)$ in \cite{AGV}). 
This carries a \emph{virtual fundamental class} 
$[(X_W)_{0,n,\beta}]_{\rm vir} \in H_*((X_W)_{0,n,\beta};\QQ)$ 
of degree $2(n+\hc -3)$ with $\hc := N- 2 = \dim X_W$. 
We have the evaluation map at the $i$th marked point 
\[
\ev_i \colon (X_W)_{0,n,\beta} \to \ov{\cI}X_W 
\]
where $\ov{\cI}X_W$ denotes the rigidified cyclotomic 
inertia stack (see \cite{AGV}). 
Because the underlying complex analytic spaces of 
$\ov{\cI}X_W$ and $\cI X_W$ are the same, we can define 
the pull-back $\ev_i^* \colon H_{\CR}(X_W) 
\to H((X_W)_{0,n,\beta})$. 
Let $\psi_i$ be the first Chern class of the 
line bundle over $(X_W)_{0,n,\beta}$ whose fibre 
at a stable map is the cotangent space of the 
coarse curve at the $i$th marked point. 
The orbifold GW 
invariant is defined to be 
\begin{equation} 
\label{eq:orbGW-inv} 
\corr{\tau_{b_1}(\alpha_1),\dots,\tau_{b_n}( \alpha_n)
}^{\GW}_{0,n,\beta} 
:= \int_{[(X_W)_{0,n,\beta}]_{\rm vir}} 
\prod_{i=1}^n \ev_i^*(\alpha_i) \psi_i^{b_i}. 
\end{equation} 
Here $\alpha_1,\dots,\alpha_n \in H_{\CR}(X_W)$ and 
$b_1,\dots,b_n\ge 0$. Again we sometimes 
omit $\tau_0$ from the notation. The orbifold GW 
invariants are also homogeneous. The invariant 
\eqref{eq:orbGW-inv} vanishes unless 
$\sum_{i=1}^n (b_i + \frac{1}{2} \deg \alpha_i) 
= n+ \hc -3$. 

\subsection{Quantum cohomology and quantum connection} 
\label{subsec:qcoh_conn} 
We can associate the \emph{quantum cohomology rings} 
to both of FJRW and GW theories. 
The quantum ring of FJRW theory is defined over $\CC$, 
whereas the quantum ring of GW theory is defined over the 
\emph{Novikov ring} $\Lambda := \CC[\![\Eff]\!]$. 
It is the completion of the group ring $\CC[\Eff]$ of 
the semigroup $\Eff\subset H_2(|X_W|,\ZZ)$ 
consisting of classes of effective curves. 
For $\beta\in \Eff$, we denote by $Q^\beta$ the corresponding 
element in $\Lambda$. 
In \S \ref{subsubsec:Q=1} below, 
we see how the divisor equation reduces 
the ground ring to $\CC$ (by setting $Q^\beta =1$) 
for GW theory. 
The construction here is standard and 
can be applied to any (genus zero) 
cohomological field theories with homogeneity. 
See \cite{Manin}.   

In order to describe the quantum rings of both 
theories in a uniform way, we use the following 
notation. Let $K$ denote the ground ring. It is 
$\CC$ for FJRW theory and $\Lambda$ for GW theory. 
Let $H$ denote the state space. It is $H(W,\bmu_d)$ 
or $H_{\CR}(X_W)\otimes \Lambda$. 
Let $\{T_i\}_{i=0}^s$ be a homogeneous basis of $H$. 
We choose $T_0$ to be the identity class, i.e.\ 
$T_0 = \unit_0 \in H_{\CR}(X_W)$ in GW theory 
and $T_0 = \phi_0\in H(W,\bmu_d)_1$ in FJRW theory. 
We set $g_{ij} = (T_i,T_j)$ and let 
$(g^{ij})$ denote the matrix inverse to $(g_{ij})$. 
We write\footnote
{Note that the FJRW correlators are not defined for $n=0,1,2$  
(since the moduli spaces are empty), but the GW 
correlators still exist for these $n$ because 
the degree $\beta$ can be non-zero. We set the correlator 
to be zero when the subscript $(0,n)$ or 
$(0,n,\beta)$ is not in the stable range.} 
\begin{equation} 
\label{eq:unified-correlatornotation} 
\corr{\tau_{b_1}(T_{i_1}),\dots,\tau_{b_n}(T_{i_n})}_{0,n} 
= \begin{cases} 
\corr{\tau_{b_1}(T_{i_1}),\dots,\tau_{b_n}(T_{i_n})
}^{\FJRW}_{0,n} & 
\text{for FJRW theory;} \\
\sum_{\beta\in \Eff} 
\corr{\tau_{b_1}(T_{i_1}),\dots,\tau_{b_n}(T_{i_n})
}^{\GW}_{0,n,\beta}  
Q^\beta
& \text{for GW theory}.  
\end{cases} 
\end{equation} 
Let $t^0,\dots,t^s$ denote the co-ordinates of $H$ 
dual to the basis $T_0,\dots,T_s$ such that 
$t = \sum_{i=0}^s t^i T_i$ denotes a general point\footnote
{We use upper indices for the co-ordinates $t^i$. 
Note that $t^i$ does not mean the $i$-th power of $t$.}
on $H$. 
We regard $H$ as a supermanifold such that 
$t^i$ has the parity $|i| \equiv \deg T_i \mod 2$ 
and odd co-ordinates anticommute $t^i t^j = (-1)^{|i||j|} t^j t^i$. 
Let $K[\![t]\!]$ denote the tensor product of the 
formal power series ring in even variables
and the exterior algebra in odd variables, i.e. 
$K[\![t]\!] = K[\![t^i:\text{even}]\!] \otimes_K 
\bigwedge^\bullet_K(\bigoplus_{|i|:\text{odd}} K t^i)$. 
The \emph{quantum product} $\bullet$ is a 
$K[\![t]\!]$-bilinear product 
on $H\otimes K[\![t]\!]$ defined by 
\begin{equation} 
\label{eq:quantum-prod} 
T_i \bullet T_j = \sum_{k,l=0}^s \sum_{n\ge 0} 
\frac{1}{n!} 
\corr{T_i, T_j,T_k, t,\dots,t}_{0,n+3} g^{kl} T_l. 
\end{equation} 
This is super-commutative and associative by the WDVV 
equation. 
We call $(H\otimes K[\![t]\!], \bullet)$ the 
quantum cohomology ring. 
The identity of the product $\bullet$ 
is given by $T_0$. 
When we set $Q=0$ and $t=0$ in Gromov--Witten theory, 
the product $\bullet$ defines the so-called \emph{Chen-Ruan 
orbifold cup product} on $H_{\CR}(X_W)$.  
This limit $Q=t=0$ is called the \emph{large radius limit}. 
On the other hand, the limit point $t=0$ in FJRW theory 
is called the \emph{Landau-Ginzburg point}. 

The \emph{quantum connection} is  
the set of operators 
$\nabla_i$, $i=0,\dots,s$ on $H\otimes K[\![t]\!][z,z^{-1}]$ 
defined by 
\begin{equation} 
\label{eq:quantum-conn} 
\nabla_i \alpha = \parfrac{\alpha}{t^i} 
+ \frac{1}{z} T_i \bullet \alpha, \quad 
\alpha \in H\otimes K[\![t]\!][z,z^{-1}] 
\end{equation} 
Here $z$ is a formal parameter. 
The associativity of the product $\bullet$ implies that 
these operators supercommute, i.e.\ 
$\nabla_i \nabla_j - (-1)^{|i||j|} \nabla_j \nabla_i =0$. 
We regard the quantum cohomology $H\otimes K[\![t]\!]$ 
as the trivial vector bundle over the formal neighbourhood 
of the origin in $H$ and $\nabla$ as a flat connection on it 
with parameter $z$. 
Moreover, we can extend the connection in the 
$z$-direction because of the homogeneity in 
FJRW/GW theory. 
Define a section $E\in H\otimes K[\![t]\!]$ by 
\[
E := \sum_{i=0}^s \left(1-\frac{1}{2} \deg T_i\right) 
t^i T_i.  
\]
This corresponds to the \emph{Euler vector field}\footnote
{Since we are working in the Calabi-Yau case, 
the term $c_1(X)$ vanishes.} 
$\sum_{i=0}^s \left(1-\frac{1}{2} \deg T_i\right) 
t^i \parfrac{}{t^i}$. By abuse of notation, 
we also denote the vector field by $E$. 
It defines the half of the grading of variables: 
$\deg t^i := 2 - \deg T_i$. 
Let $\grading$ denote the grading operator  
\[
\grading(T_i) := \frac{\deg T_i}{2}\, T_i.  
\]
The connection $\nabla_{z}$ in the $z$-direction 
is defined to be  
\[
\nabla_{z}\alpha  = \parfrac{\alpha}{z} - 
\frac{1}{z^2} E\bullet \alpha
+ \frac{1}{z} \grading(\alpha). 
\]
for $\alpha\in H\otimes K[\![t]\!][z,z^{-1}]$. 
We have $[\nabla_z,\nabla_i]=0$, $i=1,\dots,s$ 
and the pairing $(\cdot,\cdot)$ is $\nabla$-flat. 
The precise meaning of the flatness of $(\cdot,\cdot)$ is as follows. 
Let $(-)^* \colon H \otimes K[\![t]\!][z,z^{-1}] \to 
H \otimes K[\![t]\!][z,z^{-1}]$ denote the map 
$\alpha(z) \mapsto \alpha(-z)$ flipping the sign of $z$. 
By extending the pairing $(\cdot,\cdot)$ on $H$ 
to $H\otimes K[\![t]\!][z,z^{-1}]$ 
bilinearly over $K[\![t]\!][z,z^{-1}]$, we have 
\begin{align}
\label{eq:flatness-pairing}  
\parfrac{}{t^i} \left((-)^*\alpha_1,\alpha_2\right) 
 = 
\left((-)^* \nabla_i \alpha_1, \alpha_2\right) + 
\left((-)^*\alpha_1, \nabla_i \alpha_2\right).  
\end{align} 
for $\alpha_1,\alpha_2 \in  H\otimes K[\![t]\!][z,z^{-1}]$. 
The flatness of the pairing follows from the Frobenius property 
$(\alpha_1 \bullet \alpha_2, \alpha_3) = 
(\alpha_1, \alpha_2 \bullet \alpha_3)$. 

We have a canonical solution of the quantum 
connection. Define $L\in \End(H)\otimes K[\![t]\!][\![z^{-1}]\!]$ 
by 
\begin{equation} 
\label{eq:fundamentalsol} 
L(t,z) \alpha = \alpha + \sum_{i,j=1}^{s} 
\sum_{n\ge 0} \sum_{b\ge 0} \frac{1}{n!(-z)^{b+1}} 
\corr{\tau_{b}(\alpha),t,\dots,t,T_i}_{0,n+2} g^{ij} T_j.  
\end{equation} 
This is an invertible endomorphism satisfying the 
following differential equations:
\begin{pro} 
\label{pro:fundamentalsol} 
For $\alpha \in H$, we have  
\[
\nabla_i(L(t,z) z^{-\grading} \alpha)=0, \quad 
\nabla_z(L(t,z) z^{-\grading} \alpha)=0.  
\] 
For $\alpha_1,\alpha_2\in H$, we have 
\[
\left
(L(t,-z) \alpha_1, L(t,z)\alpha_2
\right) = (\alpha_1,\alpha_2). 
\]
\end{pro}  
\begin{proof} 
These are well-known in GW theory  
(see e.g.\ \cite[Proposition 2.4]{Ir}) and 
can be proved similarly for FJRW theory. 
So we only sketch the outline of the proof 
in the case of FJRW theory. 
The equation $\nabla_i (L(t,z) z^{-\grading}\alpha) =0$ is 
a formal consequence of the Topological Recursion Relation (TRR), 
as shown in \cite[Proposition 2]{Pandharipande:afterGivental} for 
GW theory. 
The TRR in FJRW theory is proved in \cite[Theorem 4.2.8]{FJR1}. 
The equation $\nabla_z (L(t,z) z^{-\grading} \alpha)=0$ 
follows from the homogeneity \eqref{eq:homogeneity_FJRW} 
of FJRW invariants. 
Since $L(t,z) \alpha_i$ is flat in the $t$-direction, 
the pairing $(L(t,-z) \alpha_1, L(t,z) \alpha_2)$ 
does not depend on $t$ (see \eqref{eq:flatness-pairing}). 
Therefore $(L(t,-z)\alpha_1, L(t,z) \alpha_2) = 
(L(0,-z) \alpha_1, L(0,z) \alpha_2) = (\alpha_1,\alpha_2)$. 
\end{proof} 

\subsubsection{Restriction to the narrow/ambient part} 
\label{subsubsec:restr-narrow}

Recall that the state space $H$ is decomposed 
as $H_{\nar}(W,\bmu_d)\oplus H_{\bro}(W,\bmu_d)$ 
for FJRW theory 
and as $H_{\amb}(X_W)\otimes \Lambda 
\oplus H_{\pri}(X_W)\otimes \Lambda$ for 
GW theory. In this section, 
we denote this decomposition as 
\[
H = H' \oplus H'' 
\]
where $H'$ denotes the narrow/ambient part  
and $H''$ denotes the broad/primitive part.  
The decomposition is orthogonal with respect to the 
pairing $(\cdot,\cdot)$. 
Moreover we have the following. 

\begin{pro} 
\label{prop:ambient-narrow-closed} 
For $\alpha_1,\dots,\alpha_n\in H'$ and 
$\gamma \in H''$, we have 
\[
\corr{\tau_{b_1}(\alpha_1),\dots,
\tau_{b_n}(\alpha_n), \tau_{c}( \gamma)}_{0,n+1} 
= 0. 
\]
In particular, $H'$ is closed under 
the quantum product $\bullet$ when  
the parameter $t$ is restricted to lie on $H'$;  
the quantum connection $\nabla$ and 
the fundamental solution $L(t,z)$ preserve 
$H'$ as far as $t\in H'$. 
\end{pro} 
\begin{proof} 
Because of the deformation invariance (\cite[Theorem 1.2.5]{FJR:virt}, 
\cite[Theorem 3.4.2]{Chen-Ruan:GW})  
we can assume that $W$ is of Fermat type 
$W = x_1^{d/w_1} + \cdots + x_N^{d/w_N}$. 
Then the maximal diagonal group of symmetries preserving $W$ 
is given as 
$G_{\max} = 
\{(z_1,\dots,z_N)\,|\, z_i^{d/w_i} =1\} 
\cong \bmu_{d/w_1} \times \cdots \times \bmu_{d/w_N}$. 
The $G_{\max}$-action on $\CC^N$ naturally lifts to 
the state space $H$. 
We claim that $H'$ is the $G_{\max}$-invariant 
part of $H$: 
\begin{equation*} 
\label{eq:G_max-fixed}
H' = H^{G_{\max}} 
\end{equation*}
and that $H''$ is the sum of non-trivial 
irreducible $G_{\max}$-representations. 
The proposition follows from this claim and the 
$G_{\max}$-invariance of the correlator: 
\[
\corr{\tau_{b_1} (g\alpha_1), \dots
\tau_{b_n} (g\alpha_n), \tau_{c} (g\gamma)}_{0,n+1} 
= \corr{\tau_{b_1} (\alpha_1),\dots, 
\tau_{b_n} (\alpha_n), \tau_c(\gamma)}_{0,n+1}, 
\quad g\in G_{\max}. 
\]

First we show the claim for the FJRW state space.  
We use the description of $H(W,\bmu_d)_k$ 
as the Jacobi space $\Omega(W_k)^{\bmu_d}$ 
given in \eqref{eq:FJRWstatespace_Jacobi} 
If $k\in \Nar$, the sector $H(W,\bmu_d)_k$ is 
obviously $G_{\max}$-invariant. 
Assume that $k\notin \Nar$. Then 
each element of $H(W,\bmu_d)_k$ can be represented 
by a sum of monomial $N_k$-forms of the form 
$\prod_{k q_i \in \ZZ} (x_i^{a_i} d x_i)$ with 
$0\le a_i \le d/w_i-2$. 
But each summand spans a non-trivial irreducible 
$G_{\max}$-representations and the claim follows. 

Next we show the claim for the GW state space. 
Each twisted sector has a decomposition 
$H(\PP(\uw)_f \cap X_W) = H_{\amb}(\PP(\uw)_f \cap 
X_W) \oplus H_{\pri}(\PP(\uw)_f \cap X_W)$. 
It is obvious that $H_{\amb}(\PP(\uw)_f \cap X_W)$ 
is $G_{\max}$-invariant. The primitive part 
$H_{\pri}(\PP(\uw)_f\cap X_W)$ is isomorphic 
to $\Gr_{N_k}^{\rsW} H^{N_k-1}(W^{-1}(1))$ for $k=fd$ 
(see \cite[p.216]{Steenbrink}), 
which is $H(W,\bmu_d)_k$ (see Appendix \ref{sec:proof-Jacobi}). 
Now the claim follows for the same reason as 
the previous paragraph, since $k=fd \notin \Nar$. 
\end{proof} 
\begin{rem} 
The fact that the ambient part is closed under the 
quantum product (in GW theory) is 
shown \cite[Corollary 2.5]{Iritani:period} 
in general for complete intersections. 
\end{rem}

\subsubsection{Divisor equation and the specialization 
$Q=1$} 
\label{subsubsec:Q=1} 
In orbifold GW theory, we have the following 
Divisor Equation \cite[Theorem 8.3.1]{AGV}: 
\begin{align*} 
\corr{\tau_{b_1} (\alpha_1),\dots, \tau_{b_n}(\alpha_n), p
}_{0,n+1,\beta}^{\GW} 
& = \pair{p}{\beta} \corr{\tau_{b_1}(\alpha_1), \dots, 
\tau_{b_n}(\alpha_n)}_{0,n,\beta}^{\GW} \\
& + \sum_{i: b_i>0} 
\corr{\tau_{b_1}(\alpha_1), \dots, \tau_{b_i-1}(\alpha_i), 
\dots \tau_{b_n}(\alpha_n)}_{0,n,\beta}^{\GW}, 
\end{align*} 
where $p=c_1(\cO(1))$ is the hyperplane class on $X_W$ 
in the untwisted sector.  
We choose the homogeneous basis $\{T_i\}_{i=0}^s$ 
such that $T_0 = \unit_0$ and also that $T_1 = p$.  
The divisor equation implies that 
\[
T_i \bullet T_j = \sum_{k,l=0}^s \sum_{n\ge 0} 
\sum_{\beta \in \Eff} 
\corr{T_i,T_j,T_k,t',\dots,t'}_{0,n+3,\beta}^{\GW} 
e^{\pair{p}{\beta}t^1} Q^\beta g^{kl} T_l, 
\]
where $t' = \sum_{i\neq 1} t^i T_i$. 
Therefore, the specialization $T_i\bullet T_j|_{Q=1}$ 
makes sense as an element of 
$H_{\CR}(X_W)\otimes \CC[\![e^{t^1/w}, t']\!]$.  
Here $w$ is the least common multiple of $w_1,\dots,w_N$
(so that $\cO_{\PP(\uw)}(w)$ is a pull-back from the coarse moduli 
space $|\PP(\uw)|$). 
Under this specialization, the limit $e^{t^1/w} = t' =0$ plays the 
role of the large radius limit. 
Similarly, the specialization $Q=1$ of the fundamental 
solution $L(t,z)$ gives 
\begin{align*} 
L(t,z) \alpha  \Bigr|_{Q=1} 
& = e^{-t^1 p/z} \alpha \\ 
 + &\sum_{i,j=1}^s \sum_{n\ge 0} \sum_{b\ge 0} 
\sum_{\beta\in \Eff} 
\frac{1}{n! (-z)^{b+1}}  
 \corr{ \tau_b (e^{-t^1 p/z} \alpha), 
t',\dots,t', T_i}_{0,n+2,\beta}^{\GW} e^{\pair{p}{\beta} t^1} 
g^{ij} T_j.  
\end{align*} 
This is an element of $\End(H_{\CR}(X_W))\otimes 
\CC[\![e^{t^1/w},t']\!][t^1][\![z^{-1}]\!]$ 
and gives a fundamental solution to 
the quantum connection $\nabla|_{Q=1}$. 
In this way the divisor equation reduces 
the ground ring from $\Lambda$ to $\CC$.

\subsection{Quantum $D$-modules and integral structure} 
\label{subsec:qdm_int} 
Here we define the narrow/ambient part of the 
quantum $D$-module and introduce a certain 
integral structure on it. 
In this section we entirely work over $\CC$. 
Let $H$ denote the state space over $\CC$ 
and $H'\subset H$ denote the narrow/ambient part: 
\begin{equation} 
\label{eq:narrow-ambient-part} 
H'= \begin{cases} 
H_{\nar}(W,\bmu_d) & \text{for FJRW theory;} \\
H_{\amb}(X_W) & \text{for GW theory.} 
\end{cases}  
\end{equation} 
Recall that $H'$ is closed under the quantum product 
by Proposition \ref{prop:ambient-narrow-closed}. 
Let $\{T_0,T_1,\dots,T_r\}$ ($r\le s$) be a homogeneous basis of 
$H'$ such that $T_0$ is the identity. 
It is, for example, a reordering of 
the basis $\{\phi_{k-1}|\, k\in \Nar\}$ (FJRW theory) 
or the basis $\{p^i \unit_f |\, i\ge 0, f\in \frF\}$ 
(GW theory). 
For GW theory, we choose $T_1$ 
to be the hyperplane class $p=c_1(\cO(1))$. 
Let $\{t^0,\dots,t^r\}$ denote the dual co-ordinates 
and $t = \sum_{i=0}^r t^i T_i$ denote a general point 
on $H'$.  
The parity of these co-ordinates are all even. 
In this section, \emph{the parameter $t$ is restricted 
to lie on $H'$} unless otherwise stated. 
Also in GW theory, 
\emph{we set the Novikov parameter $Q$ to $1$}   
in the quantum product $\bullet$, the 
connection $\nabla$ and the fundamental solutions 
$L(t,z)$ as done in \S \ref{subsubsec:Q=1}. 

\subsubsection{Convergence assumption} 
\label{subsubsec:convergence} 
We assume that the quantum product 
$T_i \bullet T_j$, $0\le i,j\le r$ are all 
convergent power series. This means 
\begin{align*}
& T_i \bullet T_j \in H'\otimes \CC\{t^0,t^1,\dots,t^r\} 
& & \text{for FJRW theory;} \\
& T_i \bullet T_j \in H' \otimes \CC\{t^0,e^{t^1/w},
t^2,\dots,t^r\}
& & \text{for GW theory.} 
\end{align*} 
Let $U\subset H'$ denote the domain of convergence of 
the product $\bullet$.  
For FJRW theory, $U$ is of the form 
\[
\{ |t^i| < \epsilon, \ (\forall i) \} 
\]
for some $\epsilon>0$. 
For GW theory, $U$ is of the form 
\[
\{ \Re(t^1) <-M, \ |t^i|< \epsilon,\  
(\forall i\neq 1)\} 
\]
for some $\epsilon, M>0$. 
In practice, we do not need to assume 
the full convergence of the product. 
One can consider the quantum $D$-module 
over a submanifold of $U$ where the product 
$\bullet$ is convergent. 
In our case, we show in \S \ref{subsec:refined-mirrorthm} 
that the quantum product $\bullet$ is convergent on 
the image of the mirror map. 
When $X_W$ is a manifold, we show the full convergence 
in \S \ref{subsec:reconstruction} for 
both GW and FJRW theories. 

Note that the convergence assumption 
imply that $\nabla$ and $L(t,z)$ are analytic 
in $t\in U$ and $z\in \CC^\times$. 

\subsubsection{Quantum $D$-module} 
\label{subsubsec:QDM} 
Let $U\subset H'$ be as in \S \ref{subsubsec:convergence}. 
The present quantum $D$-module is defined as a meromorphic 
flat connection over $U\times \CC$. 
Let $z$ denote the co-ordinate on the second factor $\CC$ 
and $\pi \colon U\times \CC \to U$ denote the 
projection to the first factor. 
Let $(-) \colon U\times \CC \to U\times \CC$ be 
the map sending $(t,z)$ to $(t,-z)$. 

\begin{defn}
\label{defn:QDM}
Let $F = H' \times (U\times \CC) \to U\times \CC$ 
be the trivial vector bundle with fibre $H'$. 
Let $\nabla$ be the meromorphic 
flat connection (quantum connection) on $F$ 
\[
\nabla = \dd + \frac{1}{z} \sum_{i=0}^r (T_i \bullet) \dd t^i 
+\left( - \frac{1}{z}(E \bullet) + \grading
\right) \frac{\dd z}{z} 
\]
which can be regarded as a map 
\[
\nabla \colon \cO(F) \to \cO(F)(U\times \{0\}) 
\otimes \left(
\pi^* \Omega_U^1 \oplus \cO_{U\times \CC} \frac{\dd z}{z} 
\right). 
\]
Here $\cO(F)(U\times \{0\})$ denotes the sheaf of 
holomorphic sections of $F$ with at most simple poles 
along $\{z=0\} = U\times \{0\}$. 
Let $P$ be an $\cO_{U\times \CC}$-bilinear pairing 
\[
(-)^* \cO(F) \otimes \cO(F) \to z^{\hc} \cO_{U\times \CC} 
\]
defined by 
\[
P((-)^*s_1, s_2) 
:= (2\pi\iu z)^{\hc} (s_1(t,-z), s_2(t,z)). 
\]
Here $\hc = \dim X_W = N-2$ and 
$(\cdot, \cdot)$ in the right-hand side 
denotes the standard pairing on the state 
space defined in \eqref{eq:FJRW-pairing} and 
\eqref{eq:GW-pairing}.   
The pairing satisfies 
\begin{align*} 
(-1)^{\hc} P((-)^* s_1, s_2) & =  (-)^* P((-)^* s_2, s_1)  \\
\dd P((-)^*s_1,s_2) & 
= P((-)^* \nabla s_1, s_2) + P((-)^* s_1,\nabla s_2). 
\end{align*} 
We call the tuple $(F,\nabla,P)$ 
\emph{the narrow part quantum $D$-module} 
$QDM_{\nar}(W,\bmu_d)$ 
(in the case of FJRW theory) and 
\emph{the ambient part quantum $D$-module} 
$QDM_{\amb}(X_W)$ 
(in the case of GW theory). 
\end{defn} 

\begin{rem} 
In \cite{Manin} and \cite[Definition 2.2]{Ir}, 
the quantum connection $\nabla_{z\partial_z}$ in the $z$-direction 
is shifted by $-\hc/2$ so that it makes the ordinary 
pairing $P/(2\pi\iu z)^{\hc}$ flat. 
In this paper, we adopt the convention in 
\cite[Definition 3.1]{Iritani:period} 
because it is more compatible with 
mirror symmetry. 
\end{rem} 

\begin{rem} 
The quantum $D$-module here is a  
(TEP($\hc$)) structure in the sense of 
Hertling \cite[Remark 2.13]{Hertling:ttstar}. 
Moreover, by the given trivialization, 
$F$ is extended over $U\times \PP^1$ as a 
trivial vector bundle and thus gives a (trTLEP($\hc$))-structure 
\cite[Definition 5.5]{Hertling:ttstar}. 
In the context of LG/CY correspondence, 
it is more convenient to consider the connection 
over $U\times \CC$ 
since the extensions across $z=\infty$ do not 
match under analytic continuation. 
\end{rem} 

\begin{rem} 
\label{rem:A-VHS} 
Over the degree two subspace ${H'}^2\subset H'$,  
the narrow/ambient part quantum $D$-module 
gives rise to a variation of Hodge structure, 
the so-called (narrow/ambient) \emph{A-model VHS} 
\cite[\S 6.2]{Iritani:period}. 
This is defined to be the restriction of $(\cF,\nabla)$ 
to the subspace $({H'}^2 \cap U)\times \{z=1\}$ equipped 
with the decreasing Hodge filtration 
\[
\rsF^{\, p}_{\rm A} := 
{H'}^{\le 2(\hc-p)} \otimes \cO_{{H'}^2\cap U}
\] 
and the polarization 
\[
Q_{\rm A}(\alpha,\beta) = (2\pi\iu)^{\hc} 
\left((-1)^{\deg/2} \alpha,\beta \right).
\]
\end{rem} 

\subsubsection{Galois action} 
\label{subsubsec:Galois} 
The quantum $D$-module has an important discrete symmetry 
which we call the Galois action. This symmetry 
is also compatible with mirror symmetry. 
\begin{pro}[Galois action in FJRW theory]  
\label{pro:Galois-FJRW} 
Let $H$ be the FJRW state space and $H'$ be its narrow part. 
Define the linear map $G\colon H \to H$ by 
\[
G|_{H(W,\bmu_d)_k} = e^{-2\pi\iu (k-1)/d} \id_{H(W,\bmu_d)_k}. 
\]
The map $G$ preserves $H'$. 
Without loss of generality, one can assume that 
the convergence domain $U\subset H'$ is preserved by $G$. 
The bundle map $G_F\colon F\to G^* F$ defined by  
\begin{align*}
G_F \colon H' \times (U\times \CC) & \longrightarrow 
H'\times (U\times \CC) \\
(\alpha, (t, z)) & \longmapsto 
(e^{-2\pi\iu/d} G(\alpha), (G(t),z)) 
\end{align*} 
preserves the connection $\nabla$ 
(i.e.~ $\nabla_{\dd G(v)} \circ G_F = G_F \circ \nabla_v$) 
and the pairing $P$. 
We call it the \emph{Galois action} of the narrow 
part quantum $D$-module. 
\end{pro} 
\begin{proof} 
For a $d$-spin structure $L\to C$ on a pointed orbicurve 
$(C,\sigma_1,\dots,\sigma_n)$,  
we have $\deg(L) - \sum_{i=1}^n \age_{\sigma_i}(L) \in \ZZ$ 
by Riemann-Roch for orbicurves \cite{AGV, Kawasaki, Toen}.  
Thus the moduli space $\Spin_{0,n}^d(k_1,\dots,k_n)$ is 
empty unless $2 + \sum_{i=1}^n k_i \equiv 0 \mod d$. 
Therefore we have 
\[
\corr{\tau_{b_1}(G(\phi_{k_1})),\dots,\tau_{b_n}(G(\phi_{k_n}))}
_{0,n}^{\FJRW} = e^{2(2\pi\iu)/d} 
\corr{\tau_{b_1}(\phi_{k_1}),\dots,\tau_{b_n}(\phi_{k_n})}
_{0,n}^{\FJRW} 
\]
for $k_1+1,\dots,k_n+1 \in \Nar$. 
This fact and the formula \eqref{eq:pairing_narrow} 
of the pairing show that 
\[
G(\alpha_1) \bullet_{G(t)} G(\alpha_2) = G(\alpha_1 \bullet_t \alpha_2) 
\]
for $\alpha_1,\alpha_2\in H'$. 
Here the subscripts of $\bullet$ 
denote the parameter of the quantum product. 
The statement follows easily from this.  
\end{proof} 
\begin{rem} 
\label{rem:fullGalois} 
Since the FJRW invariants in our case are (regardless of 
narrow or broad) certain intersection numbers on 
$\Spin_{0,n}^d(k_1,\dots,k_n)$, the same argument shows that 
the Galois action preserves $\nabla$ and $P$ 
defined on the full FJRW state space $H$. 
\end{rem} 

\begin{pro}[Galois action in GW theory: 
{\cite[Proposition 2.3]{Ir}}] 
\label{pro:Galois-GW} 
Let $H$ be the GW state space and $H'$ be its 
ambient part. 
Define $G\colon H\to H$ to be the affine-linear map 
\[
G( \alpha ) = e^{2\pi\iu f} \alpha - 2\pi\iu p, \quad 
\text{\rm for } \alpha \in H(\PP(\uw)_f \cap X_W), \  
f\in \frF.  
\]
The map $G$ preserves $H'$. 
Without loss of generality, one can assume that 
the convergence domain $U\subset H'$ is preserved by $G$. 
The bundle map $G_F \colon F \to G^* F$ defined by 
\begin{align*} 
H' \times (U \times \CC) & \longrightarrow H' \times (U \times \CC) \\ 
(\alpha, (t,z)) & \longmapsto (\dd G(\alpha), (G(t),z)) 
\end{align*} 
preserves the connection $\nabla$ and the pairing $P$.  
Here $\dd G$ is the differential (linear part) of $G$. 
We call it the \emph{Galois action} of the ambient part 
quantum $D$-module. 
\end{pro} 
\begin{proof} 
In \cite[Proposition 2.3]{Ir}, 
the Galois action was defined for each orbifold line bundle. 
The map $G_F$ here arises from $\cO(1)$. 
\end{proof} 

Via the Galois action, 
the quantum $D$-module $(F,\nabla,P)$ over $U\times \CC$ 
descends to a flat connection on the quotient 
space $(U/\langle G \rangle)\times \CC$. 
We denote it by $(F,\nabla,P)/\langle G \rangle$. 

\subsubsection{Integral structure} 
\label{subsubsec:intstr} 
The $\hGamma$-integral structure in the orbifold GW theory 
was introduced in \cite[\S 2.4]{Ir}, \cite[\S 3.1]{KKP}. 
We generalize it to the case of FJRW theory 
for $(\CC^N, W,\bmu_d)$. 

\begin{defn} 
\label{defn:Gamma} 
(1) In FJRW theory, the \emph{Gamma class} 
$\hGamma_{\FJRW}\in \End(H)$ 
is defined to be 
\[
\hGamma_{\FJRW} := \bigoplus_{k=0}^{d-1} 
\prod_{j=1}^N \Gamma\left(1-\fracp{kq_j}\right)
\]
where $q_j = w_j/d$ and the $k$th summand 
acts on $H(W,\bmu_d)_k$ by the scalar multiplication. 
This is similar to the Gamma class \cite[\S 2.4]{Ir} 
of the tangent bundle of the orbifold $[\CC^N/\bmu_d]$. 

(2) In GW theory, the \emph{Gamma class}  
$\hGamma_{\GW} \in \End(H)$ is defined to be 
\[
\hGamma_{\GW} := \bigoplus_{f\in \frF} 
\frac{
\prod_{i=1}^N
\Gamma(1 - \fracp{fw_i}+ w_i p)}
{\Gamma(1+d p)} 
\]
where $p=c_1(\cO(1))$ is the 
hyperplane class and the summand indexed by 
$f\in \frF$ acts on 
$H(\PP(\uw)_f \cap X_W)$ by the cup product. 
This is the Gamma class \cite[\S 2.4]{Ir} of the tangent bundle 
of the Calabi-Yau hypersurface $X_W$. 
\end{defn} 

\begin{rem} 
Libgober \cite{Libgober} observed that the Gamma class 
arises from periods of mirrors of Calabi-Yau hypersurfaces.  
\end{rem} 
We introduce a flat section associated to a
graded matrix factorization of $W$ 
(see \S \ref{subsec:MF}) 
or a vector bundle on $X_W$. 
We use the Chern character map 
\begin{alignat*}{2} 
\ch & \colon \MFhom^{\rm gr}_{\bmu_d}(W) \to \bigoplus_{k=0}^{d-1} 
\Omega(W_k)^{\bmu_d} \cong H(W,\bmu_d)  & \quad
& \text{for FJRW theory;}\\  
\ch & \colon D^b(X_W) \to H_{\CR}(X_W) & 
& \text{for GW theory.}
\end{alignat*} 
The Chern character for a matrix factorization 
(due to Walcher \cite{Walcher} and Polishchuk-Vaintrob \cite{PVmf}) 
will be explained in \S \ref{subsubsec:HRR} 
and we use the isomorphism 
$\bigoplus_{k=0}^{d-1}\Omega(W_k)^{\bmu_d} \cong H(W,\bmu_d)$ 
in Proposition \ref{pro:FJRWstatespace_Jacobi}. 
The Chern character for a orbi-vector bundle 
is the ``stabilizer-equivariant" version 
which appears in the Kawasaki-Riemann-Roch formula 
\cite{Kawasaki, Toen}. 
For $\ch\colon D^b(X_W) \to H_{\CR}(X_W)$, 
see for instance \cite[\S 2.4]{Ir} 
where it is denoted by $\tch$. 

\begin{defn}[$\hGamma$-integral structure] 
\label{defn:intstr} 
Let $\deg_0 \colon H \to H$ be the degree operator 
without the shift (``bare" degree operator): 
\begin{alignat*}{2}
\deg_0 &: = -2 \id_{H(W,\bmu_d)} 
& \quad & \text{for FJRW theory;}   \\ 
\deg_0|_{H^{n}(\PP(\uw)_f\cap X_W)} 
& := n \id_{H^{n}(\PP(\uw)_f \cap X_W)} 
& & \text{for GW theory.} 
\end{alignat*} 
Let $\inv \colon H \to H$ denote the map induced 
from the natural isomorphisms 
\begin{alignat*}{2} 
H(W,\bmu_d)_k & \cong H(W,\bmu_d)_{d-k} 
& \quad &\text{for FJRW theory;} \\ 
H(\PP(\uw)_f \cap X_W) &\cong 
H(\PP(\uw)_{\fracp{1-f}} \cap X_W) 
&  & \text{for GW theory.} 
\end{alignat*} 
Let $\cE$ be an object of $\MFhom^{\rm gr}_{\bmu_d}(W)$ 
(in the case of FJRW theory) or 
an object of $D^b(X_W)$ (in the case of GW theory). 
We define a $\nabla$-flat section $\frs(\cE)$ by 
\begin{equation} 
\label{eq:K-framing} 
\frs(\cE)(t,z) := \frac{1}{(2\pi\iu)^{\hc}} 
L(t,z) z^{-\grading}  \hGamma 
\left((2\pi\iu)^{\frac{\deg_0}{2}}  \inv^* \ch(\cE) \right)    
\end{equation} 
where $\hGamma$, $L(t,z)$ and $\ch(\cE)$ are 
the Gamma class, the fundamental solution \eqref{eq:fundamentalsol} 
and the Chern character in the respective theory. 
It is clear from the definition that $\frs(\cE)$ depends 
only on the numerical class of $\cE$. 
When $\ch(\cE) \in H'$, $\frs(\cE)$ defines a 
flat section of narrow/ambient part quantum $D$-module. 
Define the $\ZZ$-local system over $U\times \CC^\times$ by 
\[
F_{\ZZ} := \{ \frs(\cE) \, |\, \ch(\cE) \in H'\}  
\subset \Gamma(U\times \CC^\times, \cO(F))^{\nabla}
\]
where $\cE$ ranges over objects of $\MFhom^{\rm gr}_{\bmu_d}(W)$ 
or $D^b(X_W)$. (Note that $\ch(\cE)$ lies in $H$, but
 not in $H'$ in general.)
We call this the \emph{$\hGamma$-integral structure} 
of the narrow/ambient part quantum $D$-module.   
\end{defn} 
\begin{rem} 
The degree $\deg_0$ is even on the image of the 
Chern character map. In fact, the Chern character 
takes values in the $(p,p)$-part. 
See Remark \ref{rem:Hodgetype-ch}. 
\end{rem} 

\begin{pro} 
\label{pro:Gamma-intstr-properties}
{\rm (1)} The $\hGamma$-integral structure is preserved under 
the Galois action, i.e.\ 
\begin{alignat*}{2}
e^{-2\pi\iu/d} G \left( \frs(\cE)(G^{-1}(t),z) \right) 
& = \frs(\cE(1))(t,z)  
& & \text{\rm for FJRW theory}; \\ 
\dd G \left( \frs(\cE)(G^{-1}(t), z) \right) & = 
\frs(\cE\otimes \cO(-1))(t,z) \quad 
& & \text{\rm for GW theory},   
\end{alignat*} 
where $\cE(1)$ is the shift of the grading of $\cE$ by $1$. 
In particular, $F_\ZZ$ defines an integral structure 
on the quotient $(F,\nabla,P)/\langle G \rangle$. 

{\rm (2)} We have 
\begin{alignat*}{2} 
P((-)^*\frs(\cE), \frs(\cF)) & = (-1)^{N-1} \chi(\cE,\cF)  
& \quad & \text{\rm for FJRW theory;}\\ 
P((-)^*\frs(\cE),\frs(\cF)) & = \chi(\cE,\cF) 
& & \text{\rm for GW theory}, 
\end{alignat*} 
where $\chi(\cE,\cF) := \sum_{i\in \ZZ} (-1)^i 
\dim \Hom(\cE,\cF[i])$ is the Euler pairing of 
$\MFhom_{\bmu_d}^{\rm gr}(W)$ or of $D^b(X_W)$. 
In particular, $P$ takes values in $\ZZ$ on $F_\ZZ$. 
\end{pro}

\begin{proof} 
The proof relies on the Hirzebruch-Riemann-Roch formula 
in each category. 
We explain the case of FJRW theory. 
See \cite[Proposition 2.10]{Ir} 
(or \cite[Definition 3.6]{Iritani:period})  
for the case of GW theory. 
For Part (1), since the Galois action preserves 
$\nabla$, it suffices to check the equality at $t=0$ 
(see Remark \ref{rem:fullGalois}). 
It follows from $L(0,z) = \id$ 
and 
\[
e^{-2\pi\iu/d} G \left( \inv^* \ch(\cE)\right) 
= \inv^* \ch(\cE(1)).  
\]
Next we show Part (2). 
Setting $
\Psi(\cE) = \hGamma((2\pi\iu)^{\frac{\deg_0}{2}} 
 \inv^* \ch(\cE) )$, 
we have 
\begin{align*} 
P((-)^*\frs(\cE), \frs(\cF)) & = (2\pi\iu)^{-\hc} z^{\hc} 
((-z)^{-\grading} \Psi(\cE), z^{-\grading} \Psi(\cF)) 
\quad \text{by Proposition \ref{pro:fundamentalsol}} \\ 
& = (2\pi\iu)^{-\hc} 
( (-1)^{-\grading} \Psi(\cE), \Psi(\cF) ) \\ 
& = (2\pi\iu)^{-\hc} 
\sum_{k=0}^{d-1} 
\left( \prod_{j=1}^N 
\Gamma(1-\fracp{k q_j}) \Gamma(1- \fracp{(d-k)q_j})\right) 
(2\pi\iu)^{-2} \\ 
& \qquad \qquad \times 
(-1)^{-\frac{N_k}{2} +1 - \sum_{j} \fracp{kq_j}}
\left(
(\inv^* \ch(\cE))_k, (\inv^* \ch(\cF))_{d-k}\right) \\
& =  \sum_{k=0}^{d-1} (2\pi\iu)^{-N_k}
\left( \prod_{\fracp{kq_j}\neq 0} 
\frac{1}{1 - e^{2 \pi\iu \fracp{kq_j}}}
\right) (-1)^{N-N_k} (-1)^{-\frac{N_k}{2}+1}
\left(\ch(\cE)_{d-k}, \ch(\cF)_{k} \right) 
\end{align*} 
where $\ch(\cF)_k$ denotes the component of $\ch(\cF)$ 
in the sector $H(W,\bmu_d)_k$ and we used the 
equality $\Gamma(1-x) \Gamma(x) = \pi/\sin(\pi x)$. 
Note that $\ch(\cF)_k$ vanishes if $N_k$ is odd. 
By Proposition \ref{pro:FJRWstatespace_Jacobi}, we 
can write the last expression in terms of the residue pairing: 
\[
\sum_{k=0}^{d-1} 
\left( \prod_{\fracp{kq_j}\neq 0} 
\frac{1}{1 - e^{2 \pi\iu \fracp{kq_j}}}
\right) (-1)^{N-1} (-1)^{\frac{N_k(N_k-1)}{2}}
\frac{1}{d} 
\Res_{W_k}\left(\ch(\cE)_{d-k}, \ch(\cF)_{k} \right) 
\]
where we used the fact that the degree of $\ch(\cF)_k$ 
\emph{as an element of $\Omega(W_k)^{\bmu_d}$} is 
$(N_k/2) d$ (see Remark \ref{rem:Hodgetype-ch};  
the degree here is different from the degree 
as an element of the FJRW state space). 
This equals $(-1)^{N-1} \chi(\cE,\cF)$ by Hirzebruch-Riemann-Roch 
Theorem \ref{thm:HRR}. 
\end{proof} 
\begin{rem}
In Proposition \ref{pro:Gamma-intstr-properties}, 
we do not need to assume that $\ch(\cE) \in H'$ 
or $t\in H'$.  
\end{rem} 

\subsection{Statements of the main results}  
\label{subsec:main-statement}
In this section we summarize our main results 
in three theorems. Theorems \ref{thm:main}, \ref{thm:big} 
are about analytic continuation of quantum $D$-modules 
(with integral structure) and Theorem \ref{thm:monodromyrep} 
is about the monodromy representation and derived equivalences. 
These theorems are more precise versions of
Theorems \ref{thm:overview_main} and \ref{thm:overview-autoeq} 
in the introduction. 
The proofs take the entire paper 
and are completed in \S \ref{sec:globalDmod} 
(see \S \ref{subsec:flowchart}). 

Let $v$ be an inhomogeneous co-ordinate of $\PP(1,d)$ 
such that $v=\infty$ is the $\bmu_d$-orbifold point. 
Then $u=v^{-1/d}$ gives a uniformizing co-ordiante 
around the orbifold point (LG point). 
Set $\cM:=\PP(1,d)\setminus \{ v =0 ,v= \vc\}$, 
where $v=0$ is the large radius limit point 
and $v=\vc:= d^{-d} \prod_{j=1}^N w_j^{w_j}$ 
is the conifold point. 
We write $(-) \colon \cM\times \CC_z \to \cM\times \CC_z$ 
for the map sending $(v,z)$ to $(v,-z)$. 
The following theorem is a precise version of 
Theorem \ref{thm:overview_main} 
(without the part concerning the Orlov equivalence). 
The proof will be completed in \S \ref{subsec:reduction}. 

\begin{thm} 
\label{thm:main} 
There exists a locally free sheaf $\cF$ over $\cM \times \CC_z$ 
with a meromorphic flat connection $\nabla$ (with simple poles 
along $z=0$) 
\[
\nabla \colon \cF \to 
\cF(\cM\times \{0\}) \otimes \Omega_{\cM\times \CC_z}^1, 
\]
a $\nabla$-flat, symmetric and non-degenerate pairing 
\[
P \colon (-)^*\cF \otimes \cF \to z^{\hc} \cO_{\cM\times \CC_z} 
\]
and a $\ZZ$-local subsystem $F_\ZZ$ of the same rank 
over $\cM\times \CC_z^\times$ 
\[
F_\ZZ \subset (\cF|_{\cM \times \CC_z^\times})^\nabla 
\]
such that the following holds: 

{\rm (i)} For a small open neighbourhood 
$U_{\FJRW} = \{|u|<\epsilon\}\subset \cM$ of the LG point, 
we have a mirror map $\tau_{\sFJRW} \colon U_{\FJRW} \to 
H^2_{\nar}(W,\bmu_d)/\langle G\rangle$ 
such that $\tau_{\sFJRW} = -u \phi_1 + O(u^2)$ and 
\[
(\cF, \nabla, (-1)^{N-1}P, F_\ZZ) \Bigr |_{U_\FJRW \times \CC_z} 
\cong \tau_{\sFJRW}^* 
\left( QDM_{\nar}(W,\bmu_d) /\langle G\rangle \right),   
\]
where in the right-hand side appears 
the narrow part quantum $D$-module (Definition \ref{defn:QDM}) 
of $(\CC^N, W,\bmu_d)$ equipped with the $\hGamma$-integral 
structure (Definition \ref{defn:intstr}) 
and $G$ is the Galois action (\S \ref{subsubsec:Galois}). 

{\rm (ii)} For a small open neighbourhood 
$U_{\GW} = \{|v|<\epsilon\}\subset \cM$ 
of the large radius limit point, 
we have a mirror map 
$\tau_{\sGW} \colon U_\GW \to H^2_{\amb}(X_W)/\langle G\rangle$
such that $\tau_{\sGW}(v) = p \log v + O(v)$ and 
\[
(\cF, \nabla, P, F_\ZZ) \Bigr |_{U_\GW \times \CC_z} 
\cong \tau_{\sGW}^*\left( 
QDM_{\amb}(X_W)/\langle G\rangle \right),   
\] 
where in the right-hand side appears 
the ambient part quantum $D$-module of $X_W$ 
equipped with the $\hGamma$-integral structure
and $G$ is the Galois action.  
\end{thm}

\begin{rem} 
\label{rem:gradedisom} 
Restricting the global $D$-module $(\cF,\nabla,P,F_\ZZ)$ 
to $z=1$, we obtain the analytic continuation between 
the narrow A-model VHS of $(\CC^N,W,\bmu_d)$ 
and the ambient A-model VHS of $X_W$ in 
Remark \ref{rem:A-VHS}. 
The fibre $\cF_{(x,1)}$ at $(x,1)\in \cM\times \CC_z$ has 
a well-defined filtration and a polarization 
\begin{align*} 
\rsF^{\, p} ( \cF_{(x,1)}) & = 
\left\{v\in \cF_{(x,1)} \;|\; s_v(z) 
\text{ has a pole of order $\le \hc-p$ at $z=0$} 
\right\} \\ 
Q(v_1,v_2) & = P(s_{v_1}(-1), s_{v_2}(1)), \quad 
v_1,v_2 \in \cF_{(x,1)}  
\end{align*} 
where $s_v(z) \in H^0(\CC^\times_z, \cF|_{\{t\}\times \CC_z^\times})$ 
is a unique $\nabla$-flat section such that $s_v(1)=v$. 
The filtration and the polarization coincide with those 
of the A-model VHS near the respective limit point. 
By analytic continuation, we have an isomorphism 
of state spaces 
\[
\Theta(x) \colon (H_{\nar}(W,\bmu_d), \rsF^{\, p}_{\rm A}, Q_{\rm A}) 
\cong (H_{\amb}(X_W), \rsF^{\, p}_{\rm A}, Q_{\rm A})  
\]
for a point $x$ on the universal cover $\widetilde{\cM}$. 
Taking the associated graded vector space 
with respect to $\rsF^{\,\bullet}$,  
we can turn this into a graded isomorphism 
(preserving the polarization). 
Note that $\Theta$ does not map the identity to the identity 
because of the factor $F$ in the asymptotics \eqref{eq:z-asympt-Itw}. 
\end{rem} 

In the case where the Calabi-Yau hypersurface 
$X_W$ is a smooth manifold (e.g.\ $\PP(\uw) = \PP^n$ or 
$\PP(1,1,1,1,2)$, $\PP(1,1,1,1,1,1,1,2,3)$, etc), 
we can use the reconstruction theorem to prove 
that the ``big" quantum $D$-modules are analytically 
continued to each other. 
Here the word ``big" means the 
quantum $D$-module over the full narrow/ambient sector $H'$. 
This is used in contrast with the ``small" quantum $D$-module 
which is the restriction of the big one to the 
image of the mirror maps. 
The following theorem will be proved in \S \ref{subsec:reconstruction}. 

\begin{thm} 
\label{thm:big} 
Assume that $X_W$ is a manifold. 

{\rm (i)} The ``big" quantum product of $(\CC^N,W,\bmu_d)$ 
on the narrow part 
and the ``big" quantum product of $X_W$ 
on the ambient part 
are convergent in the sense of \S \ref{subsubsec:convergence}. 

{\rm (ii)} 
The global $D$-module $(\cF,\nabla,P,F_\ZZ)$ 
in Theorem \ref{thm:main} can be extended to 
a $D$-module $(\cF^\ext,\nabla^\ext,P^\ext,F_\ZZ^\ext)$ 
over a base $\cM_\ext\times \CC_z$, 
where $\cM_\ext$ is a complex manifold 
of dimension $\rank \cF = \dim H_{\amb}(X_W)$ which contains 
a Zariski open subset $\cM'$ of $\cM$ as a 
submanifold. 
The extended $D$-module is identified 
with the ``big" narrow/ambient part quantum $D$-module 
of FJRW/GW theory in a 
neighbourhood of $U_{\GW}$ or $U_{\FJRW}$.  

More precisely, there exists a locally free sheaf 
$\cF^\ext$ over $\cM_\ext\times \CC_z $
equipped with a meromorphic flat connection $\nabla^\ext$ 
(with poles of order two along $z=0$)
\[
\nabla^\ext \colon \cF^\ext \to \cF^\ext(\cM_\ext
\times \{0\}) \otimes \left (
\pi^* \Omega_{\cM_\ext}^1 
\oplus \cO_{\cM_\ext \times \CC_z} \frac{\dd z}{z} 
\right), 
\]
where $\pi \colon \cM_\ext \times \CC_z \to \cM_\ext$ 
is the projection, 
a $\nabla^\ext$-flat, symmetric and non-degenerate 
pairing 
\[
P^\ext \colon (-)^* \cF \otimes \cF \to z^{\hc} \cO_{\cM_\ext \times \CC_z} 
\]
and a $\ZZ$-local subsystem $F^\ext_\ZZ$ of the same rank 
over $\cM_\ext \times \CC^\times_z$ 
\[
F^\ext_\ZZ \subset (\cF^\ext|_{\cM_\ext\times \CC^\times_z})^{\nabla^\ext} 
\]
such that the following holds: 
%
%
%
\begin{itemize} 
\item $(\cF^\ext,\nabla^\ext, P^\ext, F_\ZZ^\ext)|_{\cM'} 
= (\cF,\nabla,P,F_\ZZ)|_{\cM'}$; 


\item 
There exist open neighbourhoods  
$U_\heartsuit^\ext$ of $U_{\heartsuit}$ in $\cM_\ext$ 
and open embeddings 
\[
\tau^{\ext}_\heartsuit \colon 
U_\heartsuit^\ext \hookrightarrow H'_\heartsuit/\langle G\rangle, 
\quad \tau^\ext_{\heartsuit}|_{U_{\heartsuit}} = \tau_{\heartsuit}
\]
for $\heartsuit = \GW$ and $\FJRW$ 
such that we have isomorphisms  
\begin{align*} 
(\cF^\ext,\nabla^\ext,(-1)^{N-1} P^\ext, F_\ZZ^\ext) 
|_{U^\ext_{\FJRW}}  
& \cong {\tau^{\ext}_{\sFJRW}}^*
(QDM_{\nar}(W,\bmu_d)/\langle G\rangle) \\ 
(\cF^\ext,\nabla^\ext,P^\ext, F_\ZZ^\ext) |_{U^\ext_{\GW}}  
& \cong {\tau^{\ext}_{\sGW}}^*  
(QDM_{\amb}(X_W) /\langle G \rangle). 
\end{align*} 
\end{itemize} 
\end{thm} 

Finally we give a statement on the monodromy representation 
of the global quantum $D$-module $(\cF,\nabla, P, F_\ZZ)$. 
We choose base points $b_0, b_\infty \in\cM$ 
near the large radius limit point and 
the LG point such that 
$b_0, b_\infty \in \RR_{>0}$ and 
$0<b_0\ll 1 \ll b_\infty$. 
We choose paths $\gamma_{\CY}, \gamma_{\LG}, 
\gamma_0, \gamma_1$ in $\cM$ as in Figure \ref{fig:paths_in_M}. 
We also define (see also Figure \ref{fig:ancont-path}) 
\[
\gamma_l:= \gamma_{\LG}^{l} \circ \gamma_0 \circ \gamma_{\CY}^l, 
\quad 
\gamma_{\con} := \gamma_1^{-1} \circ \gamma_0 
\]
for $l\in \ZZ$. We adopt the convention that 
the composite $\gamma_A \circ \gamma_B$ 
means the concatenation of $\gamma_A$ at 
the end of $\gamma_B$. 

\begin{figure}[htbp]
\begin{center} 
\begin{picture}(300,50)
\put(30,25){\circle*{3}}
\put(30,17){\makebox(0,0){$0$}}

\put(30,25){\circle{40}}
\put(30,45){\vector(-1,0){0}}
\put(0,40){\makebox(0,0){$\gamma_{\CY}$}}

\put(270,25){\circle*{3}}
\put(270,17){\makebox(0,0){$\infty$}}

\put(270,25){\circle{40}}
\put(270,45){\vector(-1,0){0}}
\put(300,35){\makebox(0,0){$\gamma_{\LG}$}}

\blacken
{\path(48,23)(48,27)(52,27)(52,23)(48,23)} 
\put(60,25){\makebox(0,0){$b_0$}}
\blacken
{\path(248,23)(248,27)(252,27)(252,23)(248,23)} 
\put(240,25){\makebox(0,0){$b_\infty$}}

\put(150,25){\circle*{3}}
\put(160,25){\makebox(0,0){$\vc$}}



\put(150,25){\oval(200,30)}
\put(150,40){\vector(1,0){0}}
\put(150,10){\vector(1,0){0}}
\put(130,3){\makebox(0,0){$\gamma_0$}}
\put(130,47){\makebox(0,0){$\gamma_1$}}
\end{picture} 
\end{center} 
\caption{Various paths in $\cM$} 
\label{fig:paths_in_M} 
\end{figure} 

Let $N(X_W)$ denote the numerical $K$-group of $X_W$. 
From the definition of the $\hGamma$-integral structure, 
the fibre at $b_0$ of the global $\ZZ$-local system $F_\ZZ$ 
is identified with 
\begin{equation} 
\label{eq:Nprime} 
N'(X_W) := \{E\in N(X_W) \;|\; \ch(E) \in H_{\amb}(X_W)\}.  
\end{equation} 
Similarly, the fibre at $b_\infty$ of $F_\ZZ$ 
is identified with the group $N'(W,\bmu_d)$ 
of numerical classes of matrix factorizations $\cE$ such that 
$\ch(\cE) \in H_{\nar}(W,\bmu_d)$. 
The following theorem is a detailed version of 
Theorem \ref{thm:overview-autoeq} 
(which also includes the part of Theorem \ref{thm:overview_main} 
concerning the Orlov equivalence). 
The proof will be given in \S \ref{subsec:monodromy}. 

\begin{thm}
\label{thm:monodromyrep} 
The $\ZZ$-local system $F_\ZZ$ of the global $D$-module 
$(\cF,\nabla,P,F_\ZZ)$ in Theorem \ref{thm:main} 
induces the representation of the quiver of Figure \ref{fig:paths_in_M} 
given by the assignment $b_0 \mapsto N'(X_W)$, 
$b_\infty\mapsto N'(W,\bmu_d)$ and   
\begin{alignat*}{2} 
\gamma_{\CY} &  \longmapsto  \ &  \cO(-1) & 
\colon  N'(X_W) \to N'(X_W) \\ 
\gamma_{\LG} &  \longmapsto  \ &(1) & 
\colon  N'(W,\bmu_d) \to N'(W,\bmu_d) \\ 
\gamma_l^{-1} &  \longmapsto  \ & \Phi_l & 
\colon  N'(W,\bmu_d) \to N'(X_W) \\ 
\gamma_{\con}^{-1} & \longmapsto \ & T_\cO & 
\colon N'(X_W) \to N'(X_W) 
\end{alignat*} 
where $\cO(-1)$ denotes the tensor product by $\cO(-1)$, 
$(1)$ denotes the shift of the grading by $1$,  
$\Phi_l$ denotes the Orlov equivalence (\S \ref{subsec:Orlovequiv}) 
defined for $l\in \ZZ$ and 
$T_\cO$ denotes the Seidel-Thomas spherical twist 
(\S \ref{subsec:monodromy}) by the structure sheaf. 
Moreover, the monodromy representation 
\[
\rho \colon \pi_1(\cM, b_0) \to \Aut(N'(X_W),\chi)
\]
can be lifted to a group homomorphism 
\[
\hrho \colon \pi_1(\cM,b_0)
\to \Auteq(D^b(X_W))/[2], 
\] 
where $[2]$ is the $2$-shift functor. 
\end{thm}

\subsection{Flowchart of the proof} 
\label{subsec:flowchart} 
Here we show by a picture 
the organization of the rest of the paper. 
Several key results are highlighted in the chart. 
\[
\xymatrix{
& 
& 
*+[F]\txt{\S\S \ref{subsec:MF}--\ref{subsec:Orlovequiv}: 
Matrix factorization \\ and Orlov equivalence }
\ar[d] \\ 
& 
*+[F]\txt{\S \ref{sect:computing}:  
Computation of the \\ 
twisted FJRW theory: \\ 
Mirror Theorem \\
(Theorem \ref{thm:FJRW-mirrorthm})} 
 \ar[d] 
 \ar@{.>}[r]
& 
*+[F]\txt{\S\S 
\ref{subsect:Ifunct-and-MB}--\ref{subsec:noneqlimit-Orlov} 
Analytic \\ continuation 
of the \\ 
$I$-functions matches  \\ 
the Orlov equivalence \\ 
(Proposition \ref{pro:Utw},  \\ 
Theorem \ref{thm:UisOrlov})}
\ar[d] \\ 
*+[F]\txt{\S \ref{subsec:mGKZ}: 
Construction \\ 
of the global  \\ 
$D$-module $\cF^\tw$ for \\
 the twisted theory} 
\ar[r] 
&
*+[F]\txt{\S 
\ref{subsec:refined-mirrorthm}: 
Refined mirror \\ theorem  
(Theorem \ref{thm:refined-mirrorthm}) \\
$+$ analytic continuation \\ 
of the twisted theories \\ 
(Corollary \ref{cor:anconti-twisted}) 
}
\ar[r] 
\ar[d] 
& 
*+[F]\txt{\S \ref{subsec:ancont-revisited}: 
Analytic \\ continuation revisited} 
\ar[dd] 
\ar[dl] 
\\
&  
*+[F]\txt{\S \ref{subsec:reduction}: 
Construction of the \\ 
global $D$-module $\cF$. \\
Proof of Theorem \ref{thm:main}}
\ar[d] 
\ar[dr]
& 
\\
& 
*+[F]\txt{\S \ref{subsec:reconstruction}: 
Reconstruction of \\ 
big quantum $D$-modules. \\ 
Proof of Theorem \ref{thm:big} 
}
&
*+[F]\txt{\S \ref{subsec:monodromy}: 
Global monodromy \\ 
and autoequivalences. \\  
Proof of Theorem \ref{thm:monodromyrep}}
}
\]

\bigskip 
\noindent
The essence of the paper lies in the computation in 
\S \ref{sec:Orlov=MB} 
where the analytic continuation map $\UU$ 
of the two $I$-functions (or more precisely 
the two $\Hf$-functions) is matched up with the Orlov 
derived equivalence.

\section{Computing FJRW theory}
\label{sect:computing}
We compute FJRW invariants attached to 
narrow state space entries. 
In \S\ref{subsubsect:extending}, 
we provide an extension 
of the definition of the invariants to a larger state space. 
The new invariants are zero on 
the extended part, but arise as the non-equivariant limit 
of the $e_{T}$-twisted invariants.   
In \S\S \ref{subsect:Giventalspace}--\ref{subsec:twist-equivEuler},  
we calculate the twisted invariants (or more precisely 
the $I$-function) using Chiodo-Zvonkine's results \cite{CZ} 
and Givental's symplectic formalism \cite{Givental}. 

\subsection{Extending FJRW theory}
\label{subsubsect:extending}
Define the \emph{extended narrow state space} 
(or simply the \emph{extended state space}) 
to be 
\begin{equation} 
\label{eq:extendedstate}
H_{\ext}=\bigoplus_{k=1}^{d} 
\phi_{k-1}\CC=H_{\nar}(W,\bmu_d)\oplus 
\bigoplus_{k\not \in \Nar} \phi_{k-1}\CC.
\end{equation} 
This modified state space may be regarded as the result of the
replacement of each term of  the broad sector 
$H(W,\bmu_d)_k$, $k\not\in \Nar$, 
with a one-dimensional term $\phi_{k-1}\CC$. As we 
show straight away 
in Proposition \ref{pro:placeholder} 
these new states play the role of placeholders in the theory: they 
only yield vanishing invariants and they allow 
to simplify the computation of the invariants with narrow entries.

We need to extend the grading of $H_{\nar}(W,\bmu_d)$ to the extended 
state space. We set  (cf.\ \eqref{eq:deg-FJRW}) 
\begin{equation} 
\label{eq:deg-extendedstate}
\deg \phi_{k-1} = 2 \sum_{j=1}^N \fracp{(k-1)q_j} = 
2 N_{k} + 
2 \sum_{j=1}^N \fracp{k q_j} -2 
\end{equation} 
with $q_j := w_j/d$. 
The relevant moduli stack is $\Spin^d_{0,n} (k_1,\dots,k_n)$ 
defined as in \eqref{eq:Spin}, 
but for $k_1,\dots,k_n\in \{0,\dots, d-1\}$. 
Its universal curve $\pi\colon \Ccal\to \Spin^d_{0,n}(k_1,\dots,k_n)$ 
is equipped with a universal $d$-spin structure 
$\cL$ and a line bundle 
$\Mcal_i = \cO(\cD_i)$, where $\cD_i\subset \cC$ 
denotes the divisor of the $i$th marking. 
The \emph{extended obstruction $K$-class}  
is defined to be 
\[
\textstyle -\RR \pi_*\left(\bigoplus_{j=1}^N 
\wt\cL^{\otimes w_j}\right)\qquad 
\text{for} \quad  
\wt \cL=\cL\otimes \cM^\vee 
\]
where $\cM = \bigotimes_{i=1}^n \cM_i$. 
Let $p\colon \cC \to \ov\cC$ denote the 
morphism forgetting the stack-theoretic structure 
along all the markings $\cD_1,\dots,\cD_n$ 
(but not along the nodes). 
Then we have 
\[
\age_{\cD_i} (\wt \cL) = \frac{k_i}{d}, \qquad 
\wt \cL^{\otimes d} \cong p^* \ov\omega
\]
for the relative dualizing sheaf $\ov\omega$ of 
$\ov\pi \colon \ov\cC \to \Spin_{0,n}^d(k_1,\dots,k_n)$.  
\begin{pro}
For any fibre $C$ of $\Ccal$, 
we have  
$H^0(C,\wt \cL^{\otimes w_j}\rest{C})=0$, $j=1,\dots,N$.  
As a consequence, 
$R^{1}\pi_* (\wt \cL^{\otimes w_j})$ is locally free
and the extended obstruction $K$-class is represented by a vector bundle.
\end{pro}
\begin{proof}
Because $w_j$ divides $d$, $\wt \cL^{\otimes w_j}$ is a 
root of $p^* \ov\omega$. 
On the other hand, 
we have $H^0(C,p^* \ov\omega) = H^0(\ov{C},\ov\omega) = 0$ 
because the genus of $\ov{C} = p(C)$ is zero. 
Hence $\wt \cL^{\otimes w_j}|_C$ does not have 
nonzero global sections either. 
\end{proof}
We define the \emph{extended FJRW invariants} to be  
\begin{equation}
\label{eq:extendedtheory}
\langle \tau_{b_1}(\phi_{k_1}),\dots,\tau_{b_n}(\phi_{k_n})
\rangle_{0,n}^{{\FJRW,\ext}}:=
\int_{[\Spin_{0,n}^d({k_1},\dots,{k_n})]} 
\left(\prod_{i=1}^n \psi_i^{b_i}\right)
\cup \ctop\left(\bigoplus_{j=1}^\nn R^1\pi_* \wt\cL^{\otimes w_j}
\right),    
\end{equation} 
for $\phi_{k_1},\dots,\phi_{k_n}$ 
lying within the \emph{extended state space} $H_{\ext}$.
\begin{pro}
\label{pro:placeholder}
The above invariants vanish if one of 
the entries $\phi_{k_1},\dots, \phi_{k_n}$ does not belong to 
the narrow state space $H_{\nar}(W,\bmu_d)$. 
Otherwise 
$\corr{\tau_{b_1}(\phi_{k_1}),\dots,\tau_{b_n}(\phi_{k_n})
}_{0,n}^{{\FJRW,\ext}}$ equals
$\corr{\tau_{b_1}(\phi_{k_1}),\dots,\tau_{b_n}(\phi_{k_n})
}_{0,n}^{\FJRW}$. 
\end{pro}
\begin{proof}
The proof parallels the argument of 
\cite[Lemma 4.1.1]{ChiRquintic}.  
Let us compare $\wt \cL^{\otimes w_j}$ and 
$\cL^{\otimes w_j}$ after push-forward 
via the morphism $p\colon \cC \to \ov\cC$ 
forgetting the stack-theoretic structure at the markings. 
We get\footnote
{To see this, use $p^* p_* \cE = 
\cE \otimes \cO(-\sum_{i=1}^n d\age_{\cD_i}(\cE)\cD_i)$ 
for an invertible sheaf $\cE$ on $\cC$.} 
\begin{equation}
\label{eq:allyouwanttoknow}
p_*(\wt \cL^{\otimes w_j})=
p_* (\cL^{\otimes w_j}) \otimes
\Ocal\left(-\textsum_{i:(k_i+1){w_j}\in d\ZZ} \ov\cD_i\right), 
\end{equation}
where $\ov\cD_i \subset \ov\cC$ is the divisor supported 
on the $i$th coarse marking.  
Therefore, if $(k_i+1){w_j}\not \in d\ZZ$ for all $i$, we have
$R^{1}\pi_* (\wt \cL^{\otimes w_j})=R^{1}\pi_* (\cL^{\otimes w_j})$. 
This shows the second claim above. 

%
The vanishing condition in the statement holds 
when $(k_i + 1) w_j \in d \ZZ$ for some $1\le i\le n$ and 
some $1\le j\le N$. 
This simply means that $\cL^{\otimes w_j}|_{\cD_i}$ is 
pulled back from the coarse divisor $\ov\cD_i$. 
On the other hand $\cL^{\otimes w_j}$ is a root of 
$\omega_{\log}$ and $\omega_{\log}|_{\cD_i} \cong \cO_{\cD_i}$ 
via the residue map. 
Therefore, $c_1(\cL^{\otimes w_j}|_{\cD_i})=0$ and 
hence $c_1(p_* (\cL^{\otimes w_j})|_{\ov\cD_i})=0$ 
in the rational cohomology group. 

Set $\Tcal = p_*(\wt \cL^{\otimes w_j})$. 
From \eqref{eq:allyouwanttoknow}, 
$\Tcal(\ov\cD_i)|_{\ov\cD_i} 
= p_*(\cL^{\otimes w_j})|_{\ov\cD_i}$ has 
vanishing first Chern class. 
Write the exact sequence 
\[
\begin{CD}
0@>>> \Tcal @>>> \Tcal (\ov\cD_{i}) @>>> \Tcal (\ov\cD_{i})|_{\ov\cD_{i}}
@>>> 0 
\end{CD}
\]
and the induced exact sequence of vector bundles 
\begin{equation}
\begin{CD} 
\label{eq:longexact}
0 @>>> \ov\pi_* \left( \Tcal (\ov\cD_{i})\rest{\ov\cD_{i}} \right) 
@>>> R^1\ov\pi_* \Tcal 
@>>> R^1\ov\pi_* \Tcal (\ov\cD_{i}) @>>> 0.   
\end{CD}
\end{equation} 
The vanishing $\ctop (R^1 \pi_* (\wt \cL^{\otimes w_j})) 
=\ctop (R^1\ov\pi_* \Tcal)= 0$ 
follows from $c_1(\Tcal (\ov\cD_{i})\rest{\ov\cD_{i}})=0$. 
Note that, in order to get \eqref{eq:longexact}, 
we need to show that $\Tcal(\ov\cD_i)$ has 
only trivial sections on each fibre $\ov{C}$ of $\ov\cC$. 
For $a = d/w_j$, we find that 
$\Tcal(\ov\cD_i)^{\otimes a} \cong 
\ov\omega(\ov\cD_i -\sum_{l\neq i} a \fracp{k_l/a} \ov\cD_l)$ 
which is a subsheaf of $\ov\omega(\ov\cD_i)$. 
It is easy to see that 
$H^0(\ov{C},\ov\omega(\ov\cD_i)|_{\ov{C}})=0$ 
for a genus zero curve $\ov{C}$ by induction on 
the components (see \cite{ChiRquintic}).  
Therefore $H^0(\ov{C}, \Tcal(\ov\cD_i)|_{\ov{C}})=0$. 
\end{proof}
\subsection{Twisted FJRW theory and Givental's formalism}
\label{subsect:Giventalspace} 
Let $K = \CC[\![\bs]\!]$ denote the completion of the 
polynomial ring 
$\CC[s^{(j)}_k \,|\, 1\le j\le N, k\ge 0]$ 
with respect to the 
additive valuation 
\[
v(s_k^{(j)} ) = k+1. 
\]
We define the ring $K\{z,z^{-1}\}$ of 
adically convergent power series in $z$ by 
\[
K\{z,z^{-1}\} = 
\left\{ \sum_{n\in \ZZ} a_n z^n \,\Big |\, a_n\in K, \ 
v(a_n) \to \infty \text{ as } 
|n| \to \infty \right \}.  
\]
Define $K\{z\}$ (resp.\ $K\{z^{-1}\}$) to be the subring of 
$K\{z,z^{-1}\}$ consisting of non-negative (resp.\ non-positive) 
power series in $z$. 
We introduce a symmetric non-degenerate pairing 
$(\cdot,\cdot)_{\bs}$ 
on $H_\ext\otimes K$ taking values in $K$: 
\begin{equation} 
\label{eq:twisted-pairing}
(\phi_h,\phi_k)_{\bs} = \frac{1}{d} 
\left( \prod_{j : \langle (h+1)q_j  \rangle{{=0}} }
\exp\left(-s^{(j)}_0 \right) \right)\delta_{h+k,d-2},
\end{equation} 
where $\delta_{h+k,d-2}$ is $1$ if $h+k\equiv d-2\ (d)$ and $0$ 
otherwise. 
Through the entire section we adopt the convention 
that the index is reduced 
modulo $d$ to the suitable range $\{0,\dots, d-1\}$. 

\begin{defn}[Givental's symplectic space]
Givental's symplectic space is the space 
\[
\cH := H_\ext \otimes K\{z,z^{-1}\} 
\]
equipped with the symplectic form 
\[
\Omega^{\bs}(f_1,f_2)= 
\Res_{z=0} (f_1(-z),f_2(z))_{\bs} dz. 
\]
The space $\cH$ has a standard polarization  
$\cH = \cH_+ \oplus \cH_-$, 
where $\cH_{+} = H_\ext \otimes K\{z\}$, 
$\cH_- = H_\ext \otimes z^{-1} K \{z^{-1}\}$ 
are isotropic subspaces of $\cH$. 
This polarization allows us to identify 
$\cH$ with the total space of 
the cotangent bundle of $\cH_+$. 
\end{defn}

For the basis 
$\{\phi_k\}_{k=0}^{d-1}$ of $H_{\ext}$, 
we write $g_{hk}^\bs$ for $(\phi_h,\phi_k)_{\bs}$ 
and $g^{hk}_\bs$ for 
the coefficients of the inverse matrix.
A general point of $\cH$ can be written as $\bq+ \bp$ 
with 
\begin{equation} 
\label{eq:Darboux}
\bq = \sum_{b\ge 0} \sum _{k=0}^{d-1} q^k_{b} \phi_k z^b 
\in \cH_+, \quad 
\bp = \sum _{b\ge 0} \sum_{h,k=0}^{d-1} 
p_{b,h} g_{\bs}^{hk} \frac{\phi_k}{(-z)^{1+b}} \in \cH_-. 
\end{equation} 
Here $\{q^k_{b}, p_{b,k} \;|\; b\ge 0, 0\le k\le d-1 \}$ 
can be regarded as Darboux co-ordinates on $\cH$. 
Following established practice, we denote the coordinates in 
Givental formalism by $q_b^k$, 
with the label $b$ corresponding to 
gravitational descendants and with the 
label $k$ corresponding to state space entries. 
We indicate explicitly where the superscript ``$k$'' index 
is meant instead as an exponent of a power, indeed 
this turns out to be useful in some very special cases. 

\begin{defn}[Twisted FJRW theory cf.\ \cite{Coates-Givental}]  
\label{defn:univcharclass} 
Consider the universal characteristic class 
of the extended obstruction $K$-class: 
\begin{equation} 
\label{eq:univ-charclass}
e(\bs)=\exp\left(\sum_{1\le j\le N} \sum_{l=0}^\infty s_l^{(j)} 
\ch_l (\RR \pi_* \wt \cL^{\otimes w_j})\right)
\in H^*(\Spin^d_{0,n}(k_1,\dots,k_n);\CC)\otimes K 
\end{equation}
and define the \emph{twisted FJRW invariants} as  
\begin{equation}
\label{eq:twisted-FJRW-inv} 
\corr{\tau_{b_1}(\phi_{k_1}),\dots,\tau_{b_n}(\phi_{k_n}) 
}_{0,n}^{\bs}=\int_{[\Spin^d_{0,n}(k_1,\dots,k_n)]} 
\left(\prod_{i=1}^n \psi_i^{b_i}\right)
\cup e(\bs). 
\end{equation} 
The twisted FJRW invariants are encoded in the generating function 
\begin{equation}
\label{eq:twisted-potential}
\bF_0^\bs = 
\sum_{\substack{b_1,\dots, b_n\ge 0 \\ 
0\le k_1, \dots, k_n\le s}}
\corr{\tau_{b_1}(\phi_{k_1}),\dots,\tau_{b_n}(\phi_{k_n})
}_{0,n}^{\bs}\frac{t_{b_1}^{k_1}\cdots t_{b_n}^{k_n}}{n!}.
\end{equation} 
This is a formal power series in infinitely many 
variables $\{t_b^k\,|\, 0\le k\le d-1, b\ge 0\}$. 
Again, the superscript $k$ of $t_b^k$ means an index, not an exponent 
of a power. 
\end{defn} 

The twisted FJRW invariants here 
are a generalization of the extended 
invariants \eqref{eq:extendedtheory}.

\begin{defn}[Givental's Lagrangian submanifold] 
We relate the variables $\{t_b^k\}$ of $\bF^\bs_0$ 
and the co-ordinates $\{q_b^k\}$ on $\cH_+$ 
by the following \emph{dilaton shift}: 
\[
q_b^k = - \delta_b^1 \delta^k_0 + t_b^k. 
\]
Then $\bF_0^\bs$ can be regarded as a function defined on 
a formal neighbourhood of $-z \phi_0 \in \cH_+$. 
The graph of $\dd\bF_0^{\bs}$ defines a Lagrangian 
submanifold of $(\cH,\Omega^\bs)$: 
\begin{equation} 
\label{eq:Lag-cone}
\cL^{\bs}:=\left\{ \bq +\bp \in \cH \,\Big |\, 
p_{b,k} = \parfrac{\bF_0^{\bs}}{q_b^k}, \ b\ge 0, 0\le k\le d-1 
\right \}.
\end{equation} 
The submanifold $\cL^\bs$ 
can be defined as a formal scheme over $K$. 
See \cite[Appendix B]{CCIT:computing}.  
\end{defn} 


\subsubsection{The untwisted theory} 
\label{subsubsect:untwisted} 
Consider the case where $s_l^{(j)}=0$ for all $1\le j\le N$ and 
$l\ge 0$. Then $e(\bs)=1$ and the associated correlators 
give the so called \emph{untwisted} invariants    
\begin{align}
\label{eq:untwisted-invariants}
\begin{split} 
\langle \tau_{b_1}(\phi_{k_1}),\dots,\tau_{b_n}(\phi_{k_n})
\rangle^{\un}_{0,n} & =\int_{[\Spin^d_{0,n}({k_1},\dots,{k_n})]} 
\prod_{i=1}^n\psi_i^{b_i}\\
& =\begin{cases} 
\displaystyle 
\frac{1}{d} \frac{(\sum_{i=1}^n b_i)!}{b_1!\cdots b_n!} & 
\text{if } n-3= \sum_{i=1}^n b_i \text{ and } 
2+\sum_{i=1}^n k_i\in d\ZZ,\\ 
0 & \text{otherwise.}
\end{cases}
\end{split} 
\end{align} 
One can show this by using the String Equation (since 
one of $b_i$ has to be zero). 
The same Hodge integral over $\ol{\cM}_{0,n}$ 
is given in \cite[Eqn (2)]{Faber-Pandharipande};  
the new factor $1/d$ here comes from the fact that 
$\Spin_{0,n}^d(k_1,\dots,k_n)$ has $\bmu_d$ as 
the generic stabilizer and that 
$[\Spin_{0,n}^d(k_1,\dots,k_n)] = 
\frac{1}{d} [\ov{\cM}_{0,n}]$. 
The generating function $\bF^\un_0$ of untwisted 
invariants are defined similarly to \eqref{eq:twisted-potential}. 
The pairing $(\cdot,\cdot)_{\bs}$ 
and the symplectic form $\Omega^\bs$ specialize to 
\[			
(\phi_k,\phi_h)_\un = \frac{1}{d} \delta_{d-2, k+h}, 
\quad 
\Omega^\un(f_1,f_2) = \Res_{z=0} (f_1(-z),f_2(z))_\un dz. 
\]
The Lagrangian submanifold $\cL^\un \subset (\cH,\Omega^\un)$ 
is defined as the graph of $\dd \bF^\un_0$ as in \eqref{eq:Lag-cone}. 
(Here one should use as Darboux co-ordinates those  
given by the \emph{untwisted} pairing $g_{kh}^\un 
= (\phi_k,\phi_h)_{\un}$ instead of $g_{kh}^\bs$, cf.\ 
\eqref{eq:Darboux}.) 

Since the untwisted invariants are 
the usual intersection numbers on $\ol {\mathcal M}_{0,n}$, 
the generating function 
$\bF^\un_0$ satisfy the well known tautological equations:  
String Equation (SE), Dilaton Equation (DE) and 
Topological Recursion Relations (TRR). 
Givental \cite{Givental} showed that 
these three equations for a genus zero 
potential $\bF_0$ are equivalent to the following 
geometric properties for 
the graph $\cL$ of the differential $\dd\bF_0$: 
\begin{itemize} 
\item $\cL$ is a cone in $\cH$ with 
vertex at the origin $\bp=\bq=0$ 
(with the dilaton shift $q^k_a=t^k_a- \delta_a^1 \delta^k_0$ understood); 

\item 
The tangent space $T$ to $\cL$ at any point on $\cL$ 
satisfies $zT=\cL \cap T$; Moreover 
the tangent space to $\cL$ at any point in $zT\subset \cL$ equals $T$. 
\end{itemize} 
We refer to these properties as \emph{Givental's geometric 
properties} for $\cL$. 
In particular, $\cL^\un$ satisfies Givental's geometric properties. 

\subsubsection{The twisted theory} 
\label{subsubsec:twistedtheory}
The Lagrangian submanifold $\cL^\bs$ was 
determined by Chiodo-Zvonkine \cite{CZ}. 
Define a linear symplectic transformation 
$\Delta \colon (\cH,\Omega^\un) \to (\cH,\Omega^\bs)$ by 
\begin{equation}
\label{eq:transf}
\Delta=\bigoplus_{i=0}^{d-1} \exp \left(\sum_{j=1}^\nn\sum_{l\ge 0}
s^{(j)}_l \frac{B_{l+1}\left(\langle iq_j\rangle+q_j \right)}
{(l+1)!}z^l \right)
\end{equation} 
where $B_n(x)$ is the Bernoulli polynomial defined 
by $\sum_{n=0}^\infty B_n(x) z^n/n! = z e^{zx}/(e^z-1)$. 
\begin{thm}[Chiodo-Zvonkine \cite{CZ}]  
\label{thm:symplectictrans}
We have $\cL^\bs = \Delta (\cL^\un)$. 
\end{thm} 
Because Givental's geometric properties are preserved 
by a linear symplectic transformation, 
\emph{the generating function $\bF^\bs_0$ of twisted FJRW invariants 
satisfies SE, DE and TRR}. 

The adaptation of \cite{CZ} to our context was  
explained in \cite[Proposition 4.1.5]{ChiRquintic}; 
we omit the details. 

\subsection{Family of elements on the Lagrangian cone}
\label{subsec:familyofelements} 
The \emph{twisted $J$-function} is a family of elements 
lying on $\cL^\bs$ parametrized by 
$t=\sum_{k=0}^{d-1} t^k \phi_k\in H_\ext \otimes K$: 
\begin{align}
\label{eq:twisted-J-funct} 
J^{\bs}(t,-z) & 
= -z \phi_0 + t + \sum_{n=2}^\infty 
\sum_{b=0}^\infty \sum_{0\le k,h\le d-1} 
\frac{1}{n!} 
\corr{ t,\dots, t, \tau_b(\phi_k)}
_{0,n+1}^{\bs} g^{kh}_{\bs} \frac{\phi_h}{(-z)^{b+1}}. 
\end{align}
Here $J^\bs(t,-z)\in \cH$ is characterized as 
a unique point lying on $\cL^\bs$ with the property: 
\begin{equation}
\label{eq:charactJ}
J^{\bs}(t,-z)=-\phi_0 z+t+O(z^{-1}). 
\end{equation} 
It is known \cite{Givental} that the 
$J$-function reconstructs 
the cone $\cL^\bs$ itself via 
Givental's geometric properties. 
Here we will find another explicit 
family of elements ($I$-function) on $\cL^{\bs}$. 

The $J$-function $J^\un\in \cL^\un$ of the 
untwisted theory (\S \ref{subsubsect:untwisted}) 
is the specialization of \eqref{eq:twisted-J-funct} 
at $\bs =0$. 
Using \eqref{eq:untwisted-invariants}, we calculate  
\begin{align*}
& J^\un (t,-z )  =
\sum_{\bk=(k_0,\dots, k_{d-1}) \in \ZZ^d_{\ge 0}}
J_{\bk}^{\un}(t,-z), \\  
& \text{where} \quad 
J_{\bk}^{\un}(t,-z) = \frac{1}{(-z)^{\abs{\bk}-1}} 
\frac{(t^0)^{k_0} \dots (t^{d-1})^{k_{d-1}}}{k_0!\dots k_{d-1}!}
{\phi}_{h(\bk)}, 
\end{align*}
with $|\bk| = \sum_{i=0}^{d-1}k_i$ 
and $h(\bk) = \sum_{i=0}^{d-1} i k_i$. 
Here $(t^i)^{k_i}$ means the $k_i$-th power of the variable $t^i$.  
Introduce the \emph{modification factor} $M_{\bk}(z)$ by 
\[
M_{\bk}(z)=\prod_{j=1}^\nn 
\exp\left(-\sum_{0\le  m<\floor{q_jh(\bk)}} 
\bs^{(j)} 
\left(- (q_j +\fracp{q_j h(\bk)}+m) z \right)\right),
\]
where $\bs^{(j)}(x)=\sum_{n\ge 0} s^{(j)}_n x^n/n!$ 
and define the \emph{twisted $I$-function} by  
\begin{equation} 
\label{eq:twisted-I-funct} 
I^\bs (t,z)=\sum_{k_0,\dots,k_{d-1}\ge 0} M_{\bk}(z)J^\un_{\bk}(t,z).
\end{equation} 
Using Theorem \ref{thm:symplectictrans}, 
we get the following statement. 
\begin{thm}
\label{thm:I_is_on_the_cone}
The family $t\mapsto I^\bs(t,-z)$ of elements of $\cH$ 
lies on $\cL^\bs$.  
\end{thm}
\begin{proof} 
The discussion here is parallel to 
\cite[Theorem 4.8]{CCIT:computing} 
and \cite[Theorem 4.1.6]{ChiRquintic}. 
We give a sketch of the proof 
and leave the details to the reader. 
Introduce a function 
\begin{eqnarray*}
G^{(j)}_y(x,z)=\sum_{m,l\ge 0}
s^{(j)}_{l+m-1}\frac{B_m(y)}{m!}\frac{x^l}{l!}z^{m-1}, \quad 
j=1,\dots, N 
\end{eqnarray*}
with $s^{(j)}_{-1}=0$.
Set $D := \sum_{k=0}^{d-1} k t^k (\partial/\partial t^k)$. 
Givental's geometric properties for the cone $\cL^\un$ 
yield the following fact (see 
\cite[Eqn (14)]{CCIT:computing} and 
\cite[Lemma 4.1.10]{ChiRquintic}): 
\begin{lem} 
\label{lem:diffop-applied} 
The family $t\mapsto \exp(-\sum_{j=1}^N 
G_{0}^{(j)}(zq_j D+zq_j,z))J^\un(t,-z)$ lies on $\cL^{\un}$.\qed
\end{lem}
The conclusion of Theorem \ref{thm:I_is_on_the_cone} follows from 
Theorem \ref{thm:symplectictrans}: we apply the symplectic 
transformation $\Delta \colon (\cH,\Omega^\un) \to (\cH,\Omega^\bs)$ 
in \eqref{eq:transf} to the family in Lemma \ref{lem:diffop-applied}. 
Note that we have 
\[
\Delta = \bigoplus_{i=0}^{d-1} 
\exp \left( \sum_{j=1}^N G_{\fracp{iq_j}+q_j}^{(j)}(0,z) 
\right) 
= \bigoplus_{i=0}^{d-1} 
\exp \left( \sum_{j=1}^N G_0^{(j)}\left( (\fracp{iq_j} +q_j)z, z 
\right) \right) 
\]
where we used 
$G^{(j)}_y(x,z) = G^{(j)}_0(x + yz, z)$ 
in the second equality. 
Using the identity 
\[
G^{(j)}_0(x+z,z)=G^{(j)}_0(x,z)+\bs^{(j)}(x) 
\]
we can easily check that 
\[
I^\bs(t,-z) = \Delta \exp\left(-\sum_{j=1}^N 
G_0^{(j)}(zq_j D+zq_j,z)\right)J^\un(t,-z). 
\]
Theorem \ref{thm:symplectictrans} and Lemma \ref{lem:diffop-applied} 
show that $I^\bs(t,-z)$ is on the cone $\cL^{\bs}$. 
%
%
\end{proof}

\subsection{The twist by the equivariant Euler class} 
\label{subsec:twist-equivEuler}
Let $T= (\CC^\times)^N$ act on 
the extended obstruction bundle 
$\bigoplus_{j=1}^N R^1\pi_*(\wt \cL^{\otimes w_j})$ 
diagonally by scaling the fibres 
and trivially on the base $\Spin^d_{0,n}(k_1,\dots,k_n)$. 
Let $\lambda_1,\dots,\lambda_N\in H^2_T(\pt)$ denote the 
equivariant parameters. 
Then the $T$-equivariant Euler class $e_T$ of the extended obstruction 
bundle is given by 
\[
e_T\left(
\bigoplus_{j=1}^N R^1 \pi_*(\wt \cL^{\otimes w_j}) \right)
= \prod_{j=1}^N \sum_{l=0}^{r_j} 
\lambda_j^{r_j - l} 
c_l( R^1 \pi_*(\wt \cL^{\otimes w_j}))   
\]
with $r_j = \rank(R^1 \pi_*(\wt \cL^{\otimes w_j}))$.  
This class can be obtained from the universal 
class $e(\bs)$ \eqref{eq:univ-charclass} 
by the substitution:  
\begin{equation*} 
s_l^{(j)}= 
\begin{cases} 
- \log \lambda_j  & l=0; \\ 
(l-1)! (-\lambda_j)^{-l} &  l>0.   
\end{cases} 
\end{equation*} 
Note that this specialization yields $
\exp(-\bs^{(j)}(x)) = x + \lambda_j$, 
where $\bs^{(j)}(x) = \sum_{n=0}^\infty s^{(j)}_n x^n/n!$ 
as before. 
With this choice of parameters, we obtain 
\begin{itemize} 
\item The \emph{$e_T$-twisted pairing} as 
the specialization of \eqref{eq:twisted-pairing}: 
\[
(\phi_h,\phi_k)_{\tw} := \frac{1}{d} 
\left( \prod_{j:\fracp{q_j(h+1)}=0} 
\lambda_j  \right) \delta_{d-2, k+h}; 
\]

\item The \emph{$e_T$-twisted FJRW invariants} 
$\corr{\tau_{b_1}(\phi_{k_1}),\dots,\tau_{b_n}(\phi_{k_n})
}_{0,n}^{\tw}$ as the specializations 
of \eqref{eq:twisted-FJRW-inv}; 

\item The \emph{$e_T$-twisted $J$-function} 
$J^\tw(t,-z;\lambda)$ 
as the specialization of \eqref{eq:twisted-J-funct};  

\item The \emph{$e_T$-twisted $I$-function} 
$I^\tw(u,z;\lambda)$ 
as the specialization of $u I^\bs(-u\phi_1,z)$ 
(see \eqref{eq:twisted-I-funct}):   
\begin{equation} 
\label{eq:FJRW-Ifunction}
I^\tw(u,z;\lambda) : = 
z \sum_{k=1}^{\infty} u^{k} 
\frac{\prod_{j=1}^N \prod_{0<b<q_j k, \fracp{b} = \fracp{q_jk}} 
(\lambda_j-bz)} 
{\prod_{0<b<k} (-bz)} 
\phi_{k-1}. 
\end{equation} 
\end{itemize} 
Notice that the denominator $\prod_{0<b<k}(-bz)$ here is 
nothing but $(k-1)!(-z)^{k-1}$, but we prefer this expression in view of 
our parallel treatment of $I$-functions in \S\ref{subsubsect:Ifunct}. 
Notice also that the non-equivariant limit $\lambda_j \to 0$ 
of the $e_T$-twisted FJRW invariants yield the 
extended invariants \eqref{eq:extendedtheory}. 


\begin{rem} 
\label{rem:loglambda} 
Note that the specialization $s_0^{(j)} = -\log \lambda_j$ 
does not make sense for every element in the ground ring $K$. 
For this reason, we do not try to define 
the $e_T$-twisted Lagrangian cone. 
The specializations of the $I$- and $J$-functions, 
however, still make sense as elements of 
$H_\ext \otimes \CC[z,z^{-1}]
[\lambda_1^\pm,\dots,\lambda_N^\pm]
[\![t^0,\dots,t^{d-1}]\!]$. 
\end{rem} 
 
The $e_T$-twisted $I$-function 
has the following $z$-asymptotics 
\begin{equation} 
\label{eq:FJRW-I-asympt}
I^\tw(u,z;\lambda) = z F(u) \phi_0 + G(u;\lambda) + O(z^{-1}) 
\end{equation} 
where $F\in \CC[\![u]\!]$ is a scalar-valued and 
$G\in H_\ext^{\le 2}\otimes \CC[\lambda_1,\dots,\lambda_N][\![u]\!]$ 
is an $H_\ext^{\le 2}$-valued power series 
(where $H_\ext^{\le 2}$ denotes the degree 
$\le 2$ part of $H_\ext$): 
\begin{align}
\label{eq:FandG}  
\begin{split}  
F(u) & =  
\sum_{k \ge 1: k \equiv 1\, (d)} 
u^k 
\frac{
\prod_{j=1}^N 
(kq_j-1)_{\ceil{kq_j}-1}} 
{(k-1)!}, 
\\ 
G(u;\lambda) &= - 
\sum_{k\ge 2: \sum_{j=1}^N \fracp{q_j(k-1)} = 1} 
u^k 
\frac{
\prod_{j=1}^N (kq_j-1)_{\ceil{k q_j}-1}} 
{(k-1)!} 
\phi_{k-1} \\ 
& \quad - \sum_{i=1}^N \lambda_i 
\sum_{k\ge d+1: k\equiv 1\, (d)} 
u^k \left( 
\sum_{0<a<k q_i, \fracp{a} = \fracp{kq_i}}
\frac{1}{a}\right) 
\frac{
 \prod_{j=1}^N (kq_j-1)_{\ceil{kq_j}-1}} 
{(k-1)!} 
\phi_0.   
\end{split} 
\end{align} 
where $(a)_n = a(a-1) \cdots (a-n+1) = \Gamma(a+1)/\Gamma(a-n+1)$ 
denotes the falling factorial. 
We define the \emph{FJRW mirror map} to be 
the $H_\ext^{\le 2}$-valued function: 
\[
\varsigma(u;\lambda) = \frac{G(u;\lambda)}{F(u)} = - u \phi_1 + O(u^2) 
\in H_\ext^{\le 2} \otimes \CC[\lambda_1,\dots,\lambda_N][\![u]\!]. 
\]
We now state the mirror theorem for the $e_T$-twisted FJRW theory. 
\begin{thm} 
\label{thm:FJRW-mirrorthm}
We have $J^\tw(\varsigma(u;\lambda),z;\lambda) = I^\tw(u,z;\lambda)/F(u)$ 
for the function $F(u)$ in \eqref{eq:FandG}. 
\end{thm} 
\begin{proof} 
Because of the problem we discussed in Remark \ref{rem:loglambda}, 
we use another specialization 
$s^{(j)}_l=\ov{s}^{(j)}_l$, where 
\[
\ov{s}^{(j)}_l := 
\begin{cases} 
0 & l =0; \\ 
(l-1)! (-\lambda_j)^{-l} & l\ge 0.
\end{cases} 
\]
This specialization yields $\exp(-\ov{\bs}^{(j)}(x)) = 1 + x/\lambda_j$. 
It defines a well-defined homomorphism 
$K \to \CC[\![\lambda^{-1}]\!] := 
\CC[\![\lambda_1^{-1},\dots,\lambda_N^{-1}]\!]$ 
and the characteristic class: 
\[ 
e(\ov{\bs}) = 
\left(\prod_{j=1}^N \lambda_j^{-r_j} \right)   
e_T\left(\bigoplus_{j=1}^N 
R^1\pi_*(\wt \cL^{\otimes w_j})\right). 
\]
The Lagrangian cone $\cL^{\ov{\bs}}$ 
can be defined as a formal scheme over 
$\CC[\![\lambda^{-1}]\!]$ 
and $I^{\ov\bs}(t,-z)$ is lying on $\cL^{\ov{\bs}}$ 
by Theorem \ref{thm:I_is_on_the_cone}. 
After some computation, we find that 
$J^\tw$ and $I^\tw$ are related to 
$J^{\ov\bs}$ and $I^{\ov\bs}$ as 
\begin{align}
\label{eq:tw-sbar}
\begin{split}  
J^{\tw}(t,z;\lambda) & = 
R(\lambda) 
J^{\ov\bs}(R(\lambda)^{-1}t,z;\lambda) \\
I^{\tw}(u,z;\lambda) & =  R(\lambda) 
u I^{\ov{\bs}}\left(- u R(\lambda)^{-1}\phi_1, z \right)
\end{split} 
\end{align} 
where $R(\lambda) \colon H_\ext \to H_\ext$ is 
defined by $R(\lambda)\phi_h = 
(\prod_{j=1}^N \lambda_j^{-\fracp{hq_j}}) \phi_h$.  
It is also easy to check that $I^{\ov\bs}(-u\phi_1,z)$ 
has the asymptotics 
\[
I^{\ov\bs}(-u\phi_1, z) = z \ov{F}(u;\lambda) \phi_0 + \ov{G}(u;\lambda) + O(z^{-1}) 
\]
for a scalar valued function $\ov{F} = 1 + O(u)$ 
and an $H_\ext^{\le 2}$-valued function $\ov{G}$. 
Here the functions $F$, $G$ appearing in 
\eqref{eq:FandG} are related to $\ov{F}$, 
$\ov{G}$ as 
\begin{equation} 
\label{eq:F-G-ovF-ovG}
F(u) = u \ov{F}(\lambda^q u;\lambda), \quad 
G(u;\lambda) = u R(\lambda) \ov{G}(\lambda^{q} u;\lambda) 
\end{equation} 
with $\lambda^q = \prod_{j=1}^N \lambda_j^{q_j}$. 
Because $\cL^{\ov\bs}$ is a cone and $I^{\ov\bs}(-u\phi_1,-z)$ 
is on $\cL^{\ov\bs}$, we have 
(we regard $I^{\ov\bs}(-u\phi_1,-z)$ as a 
$\CC[\![\lambda^{-1}]\!][\![u]\!]$-valued 
point on $\cL^{\ov\bs}$ and apply 
\cite[Proposition B2]{CCIT:computing}): 
\[
\frac{1}{\ov{F}(u;\lambda)} 
I^{\ov\bs}(-u\phi_1,-z) 
= -z \phi_0 + \frac{\ov{G}(u;\lambda)}{\ov{F}(u;\lambda)}   
+ O (z^{-1}) \in \cL^{\ov\bs}. 
\]
By the characterization \eqref{eq:charactJ} of the $J$-function, 
this coincides with $J^{\ov\bs}(\ov{G}(u;\lambda)/\ov{F}(u;\lambda),-z)$. 
The conclusion follows from this and  
the relations \eqref{eq:tw-sbar}, \eqref{eq:F-G-ovF-ovG}.  
\end{proof}

\subsection{The $e_{\CC^\times}$-twisted quantum connection}
\label{subsec:twisted-q-connection} 

Here we discuss the $e_{\CC^\times}$-twisted quantum cohomology 
for both of FJRW and GW theory. 
We show that the non-equivariant limit $\lambda \to 0$ of the 
$e_{\CC^\times}$-twisted quantum product exists 
and reduces to the original one \eqref{eq:quantum-prod} 
restricted to the narrow/ambient part. 
In the rest of the paper, we only consider the 
$e_{\CC^\times}$-twisted theory as a twisted theory. 
Thus the word ``$e_{\CC^\times}$-twisted" is sometimes 
abbreviated as ``twisted". 

\subsubsection{A brief summary of the  
$e_{\CC^\times}$-twisted theory} 
\label{subsubsec:summary} 

We mean by the \emph{$e_{\CC^\times}$-twisted FJRW theory} 
the $e_T$-twisted FJRW theory (\S \ref{subsec:twist-equivEuler}) 
with the equivariant parameters $\lambda_1,\dots,\lambda_N$ 
specialized to the following values: 
\[
\lambda_i =  -q_i \lambda, 
\quad i=1,\dots, N.  
\]
The $e_{\CC^\times}$-twisted pairing $(\phi_h,\phi_k)_\tw$ and 
FJRW invariants 
$\corr{\tau_{b_1}(\phi_{k_1}),\dots,
\tau_{b_n}(\phi_{k_n})}_{0,n}^{\FJRW,\tw}$ 
take values in $\CC[\lambda]$. 
(We have put the superscript ``FJRW" 
to distinguish them from GW invariants.)


The \emph{$e_{\CC^\times}$-twisted GW theory}  
\cite{Coates-Givental, Tseng, CCIT:computing} 
for $\PP(\uw)$ is defined on the state space 
$H_{\CR}(\PP(\uw)) = H(\cI\PP(\uw))$. 
We consider the twist by the line bundle $\cO(d)$.  
Let $\PP(\uw)_{0,n,\beta}$ denote the moduli stack of 
$n$-pointed stable maps to $\PP(\uw)$ 
of genus $0$ and degree $\beta$. 
Let $\pi\colon \cC_{0,n,\beta} \to \PP(\uw)_{0,n,\beta}$ be 
the universal curve 
and let $f \colon \cC_{0,n,\beta} \to \PP(\uw)$ 
be the universal stable map: 
\[
\begin{CD} 
\cC_{0,n,\beta} @>{f}>> \PP(\uw) \\ 
@V{\pi} VV @. \\ 
\PP(\uw)_{0,n,\beta} 
\end{CD} 
\]
It can be shown that $\RR\pi_* f^* \cO(d)$ is represented 
by a vector bundle (see \cite{6authors}; 
this is because $\cO(d)$ is ample and $\cO(d)$ is pulled back 
from the coarse moduli $|\PP(\uw)|$). 
The $e_{\CC^\times}$-twisted GW invariants  
of $\PP(\uw)$ are defined to be 
\[
\corr{\tau_{b_1}(\alpha_1),\cdots,\tau_{b_n}(\alpha_n)
}^{\GW, \rm tw}_{0,n,\beta} 
= \int_{[\PP(\uw)_{0,n,d}]_{\rm vir}} 
\prod_{i=1}^n \ev_i^*(\alpha_i) \psi_i^{b_i} 
\cup 
e_{\CC^\times} (\RR\pi_* f^* \cO(d)), 
\]
where $\alpha_1,\dots,\alpha_n\in H_{\CR}(\PP(\uw))$ 
and $\CC^\times$ acts on $R\pi_* f^* \cO(d)$ 
by scaling the fibre. 
This is an element of $H_{\CC^\times}(\pt) = \CC[\lambda]$. 
We endow $H_{\CR}(\PP(\uw))$ with the following twisted pairing 
\[
(\alpha_1,\alpha_2)_{\tw} = \int_{\cI \PP(\uw)} 
\alpha_1 \cup \alpha_2 \cup e_{\CC^\times}(\pr^*\cO(d)) 
\]
where $\pr \colon \cI\PP(\uw) \to \PP(\uw)$ is the natural 
projection. 

\subsubsection{Twisted quantum product}
\label{subsubsec:tw-quantum}

We denote by $\ov{H}$ the state space of the twisted theory:  
\begin{align*} 
\ov{H} := \begin{cases} 
H_{\rm ext}  & \text{for FJRW theory;} \\ 
H_{\CR}(\PP(\uw)) &\text{for GW theory.} 
\end{cases} 
\end{align*}
Both state spaces are of dimension $d$. 
The same procedure as \S \ref{subsec:qcoh_conn} 
defines the twisted quantum cohomology.  
The $e_{\CC^\times}$-twisted quantum product $\bullet^{\tw}$ 
on $\ov{H}$ is defined by the formula \eqref{eq:quantum-prod} 
with the correlators $\corr{\cdots}_{0,n}$ replaced by 
the $e_{\CC^\times}$-twisted invariants 
and $(g_{ij})$ replaced by the $e_{\CC^\times}$-twisted pairing. 
Because the divisor equation holds also for the twisted 
GW theory, we can consider the specialization 
$Q=1$ for $\bullet^{\tw}$ (see \S \ref{subsubsec:Q=1}).  
In GW theory, we shall denote by $\bullet^\tw$ 
the twisted quantum product with $Q$ already specialized to $1$. 

Let $H'$ denote the narrow/ambient part 
\eqref{eq:narrow-ambient-part} of the state space. 
Let $\pr$ denote the natural projection 
\[
\pr \colon \ov{H} \longrightarrow  H'. 
\]
Let $T_0,\dots,T_{d-1}$ be a homogeneous basis of $\ov{H}$ 
such that $T_0$ is the identity ($T_0 =\phi_0$ in FJRW theory\footnote
{The element $\phi_0$ is the identity in the twisted FJRW theory 
because of the string equation  
(see \S \ref{subsubsec:twistedtheory}).} 
and $T_0 = \unit_0$ in GW theory). 
In the case of GW theory, 
we take $T_1 = p = c_1(\cO(1))$. 
Let $t^0,\dots,t^{d-1}$ denote the linear co-ordinate on $\ov{H}$ 
dual to the basis $T_0,\dots, T_{d-1}$. 
\begin{pro} 
\label{pro:noneq-reduction} 
The $e_{\CC^\times}$-twisted quantum products 
are regular at $\lambda =0$, i.e. 
\[
T_i \bullet^{\tw} T_j \in 
\begin{cases} 
\ov{H} \otimes \CC[\lambda][\![t^0,\dots,t^{d-1}]\!] 
& \text{for FJRW theory;} \\ 
\ov{H} \otimes \CC[\lambda][\![t^0, e^{t^1/w}, t^2,\dots,t^{d-1}]\!]
& \text{for Gromow-Witten theory}.
\end{cases} 
\]
Here $w$ is the least common multiple of $w_1$, \dots,$w_n$ 
(see \S \ref{subsubsec:Q=1}). 
Moreover we have 
\begin{equation} 
\label{eq:pr-hom} 
\lim_{\lambda \to 0}
\pr\left( T_i \bullet^{\tw}_t T_j \right) 
= \pr(T_i) \bullet_{\pr(t)} \pr(T_j) 
\end{equation} 
where the product in the right-hand side is the 
ordinary quantum product on the narrow/ambient part 
in \S \ref{subsec:qdm_int} and the subscripts 
$t\in \ov{H}$, $\pt(t) \in H'$ denote the parameter of the product. 
\end{pro} 
\begin{proof} 
This was proved in \cite[Corollary 2.5]{Iritani:period} 
for GW theory, so we only discuss the case of the FJRW theory. 
The $e_{\CC^\times}$-twisted FJRW quantum product can be written as 
\begin{equation*} 
\phi_{i} \bullet^{\tw} \phi_{j} = \sum_{k=0}^{d-1} \sum_{n\ge 0} 
\frac{1}{n!} 
\corr{\phi_{i}, \phi_{j},\phi_k, t,\dots,t}_{0,n+3} ^{\FJRW, \tw}
\left(d \prod_{j : \langle (k+1)q_j \rangle{{=0}} }\la_j^{-1}\right) 
\phi_{d-2-k} 
\end{equation*} 
with $\lambda_j = -q_j \lambda$. 
To see that this expression is regular at $\lambda= 0$, 
it suffices to show that 
\begin{equation*} 
\label{eq:divisible}
\corr{\tau_{b_1}(\phi_{k_1}), \dots, \tau_{b_n}(\phi_{k_n}) 
}_{0,n}^{\FJRW, \tw} \in 
\left(\prod_{j:\langle (k_i+1)q_j \rangle{{=0}}} 
\la_j\right)\CC[\la_1,\dots,\la_{\nn}], \quad 
1\le \forall i\le n. 
\end{equation*} 
This happens because $e_T(R^1\pi_*\wt \cL^{\otimes w_j})$ 
is divisible by $\lambda_j$ as soon as $d$ divides $(k_i+1)w_j$.  
This follows from the fact that $R^1\pi_*\wt \cL^{\otimes w_j}$ 
contains the sub-line bundle $\ov\pi_*(\Tcal(\ov{\cD}_i)|_{\ov\cD_i})$ 
whose equivariant 1st Chern class is $\lambda_j$. 
See \eqref{eq:longexact}. 

Using Proposition \ref{pro:placeholder} and the fact that 
the $e_{\CC^\times}$-twisted invariant 
equals the extended invariant \eqref{eq:extendedtheory} 
in the non-equivariant limit $\lambda\to 0$, 
we have 
\begin{equation} 
\label{eq:limit-cor-pr} 
\lim_{\lambda \to 0} \corr{\tau_{b_1}(\phi_{k_1}), 
\dots, \tau_{b_n}(\phi_{k_n})}_{0,n}^{\FJRW,\tw}  
= \corr{\tau_{b_1}(\pr (\phi_{k_1})), \dots, 
\tau_{b_n}(\pr (\phi_{k_n}))}_{0,n}^{\FJRW}.  
\end{equation} 
The equality \eqref{eq:pr-hom} follows easily 
from this. 
\end{proof}  


\subsubsection{Twisted quantum connection and 
the fundamental solution} 
\label{subsubsec:tw-quantum-fund} 
The \emph{$e_{\CC^\times}$-twisted quantum connection}  
is defined similarly to \eqref{eq:quantum-conn}:  
\[
\nabla^\tw_i = \parfrac{}{t^i} + \frac{1}{z} T_i \bullet^\tw. 
\]
By Proposition \ref{pro:noneq-reduction}, 
$e_{\CC^\times}$-twisted quantum connection 
is regular at $\lambda=0$. 
The non-equivariant limit is called 
the \emph{$e$-twisted quantum connection}. 
In contrast with the untwisted case, the connection $\nabla^\tw$ 
cannot be extended in the $z$-direction since 
the variable $\lambda$ has a degree. 

Let $L^\tw(t,z;\lambda)$ denote the canonical fundamental solution 
of the connection $\nabla^\tw$ defined by the same formula 
\eqref{eq:fundamentalsol} with $\corr{\cdots}_{0,n}$, $g_{ij}$ 
there replaced with the twisted counterparts. 
This satisfies (part of) the properties in 
Proposition \ref{pro:fundamentalsol}: 
\begin{pro} 
\label{pro:twisted-fundsol} 
For $\alpha, \alpha_1,\alpha_2 \in \ov{H}$, we have 
\[
\nabla^{\tw}_i L^\tw(t,z;\lambda) \alpha = 0, 
\quad 
(L^\tw(t,-z;\lambda) \alpha_1, L^\tw(t,z;\lambda)\alpha_2)_\tw 
= (\alpha_1,\alpha_2)_\tw.  
\]
In particular, the connection $\nabla^\tw$ is flat, 
i.e.~$[\nabla^{\tw}_i, \nabla^{\tw}_j] =0$  
and that the pairing $(\cdot,\cdot)_{\tw}$ is 
$\nabla^{\tw}$-flat (see \eqref{eq:flatness-pairing} 
for a precise meaning of the flatness of the pairing). 
\end{pro} 
\begin{proof} 
The outline of the proof is the same as 
Proposition \ref{pro:fundamentalsol}. 
It suffices to show that the twisted theory satisfies TRR, 
but this follows from Givental's geometric properties 
(see \S \ref{subsubsec:twistedtheory}). 
The same discussion as in \cite[Proposition 2.1]{Iritani:period} 
shows the statement for the pairing. 
\end{proof} 

\begin{rem} 
The fact that $\nabla^{\rm tw}$ 
is flat implies that $\bullet^\tw$ is associative. 
(The commutativity of $\bullet^{\tw}$ is clear from 
the definition.)
\end{rem} 

Because $L^{\tw}$ satisfies the differential equation 
regular at $\lambda=0$, it follows that 
$L^{\tw}$ is also regular at $\lambda=0$ 
(see also \cite[Proposition 2.4]{Iritani:period}). 
\[
L^{\tw}(t,z;\lambda) 
\in 
\begin{cases} 
\End(\ov{H}) \otimes \CC[\lambda][\![t^0,\dots,t^{d-1}]\!][\![z^{-1}]\!] 
& \text{for FJRW theory;} \\
\End(\ov{H}) \otimes \CC[\lambda]
[\![t^0, e^{t^1/w}, t^2, \dots, t^{d-1}]\!][t^1][\![z^{-1}]\!]  
& \text{for GW theory}. 
\end{cases} 
\]
Let $L(t,z)$ denote the fundamental solution \eqref{eq:fundamentalsol}
in the original FJRW/GW theories. (For GW theory, 
we specialize $Q$ to $1$.) 
We have the following: 
\begin{pro}
\label{pro:Ltw-L}  
\[
\lim_{\lambda \to 0} \pr \left( L^\tw(t,z;\lambda) \alpha\right)  
= L\left(\pr(t), z\right) \pr(\alpha). 
\] 
\end{pro} 
\begin{proof} 
It was shown in \cite[Proposition 3.24]{Iritani:period} 
for GW theory. 
For FJRW theory, the equality follows easily 
from \eqref{eq:limit-cor-pr}. 
\end{proof}

\subsubsection{Twisted $J$-function} 

Recall from \S \ref{subsec:familyofelements} 
that the $J$-function \eqref{eq:twisted-J-funct}
is a special family of elements lying on 
the Givental Lagrangian cone. 
The $e_{\CC^\times}$-twisted $J$-function is defined 
similarly: 
\begin{equation} 
\label{eq:twistedJ} 
J^\tw (t,z) = z T_0 + t + \sum_{n\ge 0} \sum_{b\ge 0} 
\sum_{i,j=0}^s  
\frac{1}{n! z^{b+1}} 
\corr{t,\dots, t, \tau_b(T_i)}_{0,n+1}^{\tw} g^{ij}_{\tw} T_j,   
\end{equation} 
where $(g^{ij}_{\tw})$ denotes the inverse of the twisted 
pairing matrix $g_{ij}^\tw =(T_i,T_j)_{\tw}$. 
(In the case of GW theory, as in 
\eqref{eq:unified-correlatornotation}, 
we also take the summation over curve classes $\beta$ 
(see \cite[Eqn (8)]{CCIT:computing}). Then we 
specialize it to $Q=1$ using the divisor equation.) 
The following relation of $L^{\tw}$ and $J^{\tw}$ 
is a key to understand the role of the $J$-function 
in the quantum $D$-module. 
\begin{pro} 
\label{pro:Ltw-Jtw} 
$L^\tw(t,z;\lambda) J^{\tw}(t,z;\lambda) = z T_0$. 
\end{pro} 
\begin{proof} 
By Proposition \ref{pro:twisted-fundsol}, we have 
$L^\tw(t,z;\lambda)^{-1} = L^{\tw}(t,-z;\lambda)^*$ where $*$ 
denotes the adjoint with respect to the twisted 
pairing. Thus we have 
\begin{align*} 
(T_i, L^\tw(t,z;\lambda)^{-1} zT_0)_{\tw} 
&= (T_i, L^\tw(t,-z;\lambda)^* z T_0)_{\tw} \\
& = (L^\tw(t,-z;\lambda) T_i, T_0)_{\tw} = 
(T_i, J^\tw(t,z;\lambda))_{\tw},  
\end{align*} 
where the last equality follows directly from the definitions 
\eqref{eq:fundamentalsol}, \eqref{eq:twistedJ} of $L^\tw$ and $J^\tw$ 
and the string equation. 
The conclusion follows.  
\end{proof} 

\section{Orlov equivalence matches Mellin-Barnes analytic contination}
\label{sec:Orlov=MB}

\subsection{Matrix factorizations} 
\label{subsec:MF} 
Matrix factorizations were originally introduced 
by Eisenbud \cite{Eis} for the study of 
maximal Cohen-Macaulay modules. 
Recently Kontsevich proposed that they form  
the category of B-branes in the Landau-Ginzburg model. 
References are made to 
\cite{Orlov:cat-sing, Orlov:equivalence, Walcher, 
Hori-Walcher:F-term, HHP, Dyckerhoff, PVmf}. 
The paper \cite{Dyckerhoff} contains 
a nice introduction to the subject. 

We introduce the differential graded (dg) category 
of graded matrix factorizations of a degree-$d$ 
weighted homogeneous polynomial $W \in \CC[x_1,\dots,x_N]$ 
from the introduction \S \ref{subsec:overview}.  
Set $R := \CC[x_1,\dots,x_N]$. 
Notice that $R$ is a $\ZZ_{\ge 0}$-graded ring 
by $\deg x_i = w_i$. 

\begin{defn}[Graded matrix factorization 
{\cite{Walcher, Hori-Walcher:F-term}, 
\cite[\S 3.1]{Orlov:equivalence}}]
\label{defn:grMFjsigma}
A \emph{graded matrix factorization} of $W$ 
is a collection $(E^i,\delta_i)_{i\in \ZZ}$ 
of finitely generated graded free $R$-modules $E^i$ 
and degree-zero homomorphisms 
$\delta_i\in \Hom_{{\gr}R}(E^i,E^{i+1})$ 
\[
\begin{CD} 
\cdots @>{\delta_{-1}}>> E^0 @>{\delta_0}>> E^1 
@>{\delta_1}>> E^2 
@>{\delta_2}>> E^3 @>{\delta_3}>> \cdots 
\end{CD}    
\]
such that it is 2-periodic up to the shift of grading 
\[ 
E^{i+2}=E^i(d), \quad \delta_{i+2}=\delta_i(d)
\]
and that 
$\delta_{i+1}\circ \delta_{i} = 
W \cdot \id_{E^i}\colon E^i\to E^{i+2} = E^i(d)$ for all $i$.   
This is equivalent to the data $E^0$, $E^1$, 
$\delta_0\in \Hom_{{\gr}R}(E^0,E^1)$, 
$\delta_1\in\Hom_{{\gr}R}(E^1,E^0(d))$ 
such that $\delta_1\circ\delta_0=W\cdot \id_{E^0}$ 
and $\delta_0(d) \circ \delta_1=W\cdot \id_{E^1}$. 
These data are denoted also by $(E,\delta_E)$, 
where 
\[
E:= E^0 \oplus E^1, \quad 
\delta_E := 
\begin{pmatrix}
0 & \delta_1 \\ 
\delta_0 & 0 
\end{pmatrix} \colon E \to E \quad 
\text{satisfying} \quad \delta_E^2 = W \cdot \id_E. 
\]

These objects form a dg category as follows.
Consider the graded matrix factorizations 
$\ov{E}=(E^i,\delta_{i})_{i\in \ZZ}$ and 
$\ov{F}=(F^i,\delta'_{i})_{i\in \ZZ}$; 
the space of homomorphisms is defined to be 
the $\ZZ$-graded vector space 
\[
\sfHom^\bullet(\ov{E},\ov{F}) 
= \left\{ (f_n)_{n\in \ZZ} \,\Big|\, f_n\in 
\Hom_{{\gr}R}(E^n,F^{n+\bullet}), \, 
f_{n+2} = f_n(d) \right\} 
\] 
equipped with the differential 
\[
(\dd f)_n = \delta'_{n+\bullet} \circ f_n - (-1)^{\bullet}
f_{n+1} \circ \delta_{n}, \quad 
f\in \sfHom^\bullet(\ov{E},\ov{F}). 
\]
The homotopy category of the above 
dg category is denoted by $\MFhom^{\rm gr}_{\bmu_d}(W)$. 
It is a triangulated category.  
\end{defn}

\begin{rem}[{\cite[\S 4.4]{PVmf}}]
The lower index in the notation $\MFhom^{\rm gr}_{\bmu_d}(W)$ 
emphasizes the fact that a graded matrix factorization  
is automatically $\bmu_d$-equivariant. 
The $\bmu_d$-action on $R$ is defined by 
$\zeta \cdot x_i = \zeta^{-w_i} x_i$, where 
$\zeta = \exp(2\pi\iu/d)\in \bmu_d$. 
For a graded matrix factorization 
$\ov{E} = (E^i,\delta_{i})_{i\in \ZZ}$,  
we define the $\bmu_d$-action on $E^i$ by 
$\zeta \cdot e = \zeta^{-n} e$ for $e\in (E^i)_n$. 
Then the $R$-module $E^i$ is $\bmu_d$-linearized 
and $\delta_{i}$ is $\bmu_d$-equivariant. 
\end{rem}

We introduce a graded Koszul matrix factorization 
(see \cite[\S 2]{BGS:CM-hypersurfaceII} 
for the ungraded case). 

\begin{defn}[Graded Koszul matrix factorization] 
\label{defn:gradedKoszul-MF}
Suppose that $W$ is of the form 
\begin{equation}
\label{eq:decomposition}
W=\sum_{i=1}^{N} a_ib_i 
\end{equation}
for homogeneous elements $a_i,b_i\in R$ such 
that $\deg(a_i) + \deg (b_i) = d$. 
Let $V$ be the graded vector space 
$\bigoplus_{i=1}^{N}  \CC e_i$ with 
$\deg(e_i)=-\deg(a_i)$. 
For $q\in \ZZ$, the \emph{graded Koszul matrix factorization} 
$\{\ua,\ub\}_q$ is defined by the data 
\[
E^i= \bigoplus_{\substack{{k=0,\dots, N}\\ {k\equiv i\, (2)}}}
R\otimes \Bigl(\bigwedge^{k} V\Bigr) 
\left(\tfrac{d(i-k)}{2}+q\right), \quad 
\delta_i = \delta'_{\ua} + \delta''_{\ub} \colon 
E^i \to E^{i+1}, \quad i\in \ZZ, 
\]
where $\delta'_{\ua}$, $\delta''_{\ub}$ 
are the Koszul differentials: 
\[
\delta'_{\ua} = \sum_{j=1}^N a_j e_j\wedge, \quad 
\delta''_{\ub} = \sum_{j=1}^N b_j \iota(e^*_j).  
\]
Observe that the grading is shifted so that 
the map $\delta_i$ preserves the degree. 
Note also that $\{\ua,\ub\}_q = \{\ua,\ub\}_0(q)$. 
\end{defn}

\subsubsection{Hirzebruch-Riemann-Roch Theorem}
\label{subsubsec:HRR} 
For graded matrix factorizations $\ov{E}$, $\ov{F}$ of $W$,  
we write $\chi(\ov{E}, \ov{F})$ for the 
Euler characteristic 
\[
\sum_{k\in \ZZ} (-1)^k \dim 
H^k\left(\sfHom^\bullet(\ol E,\ol F), \dd
\right).  
\]
This can be computed via Hirzebruch-Riemann-Roch (HRR) 
due to Walcher \cite{Walcher} and Polishchuk-Vaintrob \cite{PVmf} 
for $G$-equivariant matrix factorizations. 
To this effect we need to introduce 
the Chern character taking values in the 
orbifold Jacobi space $\bigoplus_{k=0}^{d-1} 
\Omega(W_k)^{\bmu_d}$, which is identified 
with the FJRW state space by Proposition 
\ref{pro:FJRWstatespace_Jacobi}. 
Let $x_{j_1},\dots,x_{j_{N_k}}$ denote 
the co-ordinates of the $\zeta^k$-fixed part $(\CC^N)_k$ 
where $\zeta = e^{2\pi\iu/d}$. 
For a graded matrix factorization 
$\ol E=(E, \delta_E)$,
we define \cite[Theorem 3.3.3]{PVmf}
\[
\ch(\ov{E}):=\bigoplus_{k=0}^{d-1} 
\left [
\str_R\left(
\partial_{j_1}\delta_E\circ 
\partial_{j_2}\delta_E \circ 
\cdots \circ 
\partial_{j_{N_k}}\delta_E 
\circ \zeta^k\right) \Big|_{(\CC^N)_k}
\dd x_{j_1} \wedge \cdots \wedge \dd x_{j_{N_k}} 
\right]. 
\]
Here we take a free basis of $E = E^0 \oplus E^1$ over 
$R=\CC[x_1,\dots,x_N]$ and regard $\delta_E$ 
as a matrix with entries in $R$; 
the supertrace $\str_R(f)$ of an operator 
$f\in \End_R(E)$ is defined to be 
$\tr(f_{0,0}) - \tr(f_{1,1})$, where 
$f_{\sigma,\sigma} \colon E^\sigma \to E^\sigma$, $\sigma=0,1$ are 
the components of $f$.  
The right hand side are meant to be the class 
in $\bigoplus_{k=0}^{d-1}\Omega(W_k)$ 
and lies in the $\bmu_d$-invariant part.
This is independent of the choice of a 
co-ordinate ordering or the choice of a 
basis of $E$. 

\begin{rem}
\label{rem:Hodgetype-ch} 
Let $\ch(\ov{E})_k$ denote the $\Omega(W_k)^{\bmu_d}$ 
component of $\ch(\ov{E})$. 
For a graded matrix factorization $\ov{E}$, 
one can see that $\ch(\ov{E})_k$ vanishes if 
$N_k$ is odd and $\ch(\ov{E})_k$ is of degree $(N_k /2)d$. 
In terms of the Hodge decomposition, 
the component 
$\ch(\ov{E})_k$ has Hodge type $(N_k/2,N_k/2)$. 
\end{rem}

\begin{exa}
\label{exa:ch-Koszul} 
For a general weighted homogeneous polynomial $W$, 
we can write $W = \sum_{j=1}^N a_j b_j$ 
with $a_j = q_j \partial_j W$, $b_j = x_j$ ($q_j := w_j/d$). 
The Chern character of the graded Koszul matrix factorization 
$\{\ua,\ub\}_q$ of $W$ is supported on the narrow sector. 
In fact, by a direct calculation, 
we obtain  
\begin{equation}
\label{eq:chernbrane}
\ch(\{\ua, \ub\}_q ) = 
\bigoplus_{k\in \Nar} 
\zeta^{qk}\left( 
\prod_{j=1}^{N}(1-\zeta^{-w_j k})
\right)  \phi_{k-1}. 
\end{equation}
See \cite[Proposition 4.3.4]{PVmf} where 
$\{\ua,\ub\}_0$ is denoted by $k^{\rm st}$. 
(The case $q\neq 0$ follows from the case with $q=0$ since 
$q$ is just a shift the grading.) 
These Chern characters span the narrow part $H_{\nar}(W,\bmu_d)$. 
\end{exa} 

\begin{thm}[Walcher {\cite[\S 5]{Walcher}}, 
Polishchuk-Vaintrob {\cite[Theorem 4.2.1]{PVmf}}] 
\label{thm:HRR} 
Let 
$\ov{E},\ov{F}$ be graded matrix factorizations. 
The Euler characteristic 
$\chi( \ov{E},\ov{F})$ is given by 
the formula: 
\begin{equation}
\label{eq:HRR}
\sum_{k=0}^{d-1} 
\left( 
\prod_{kw_j\not \in d\ZZ} \frac{1}{1-\zeta^{kw_j}} 
\right) 
(-1)^{\frac{N_k(N_k-1)}{2}}
\frac{1}{d} 
\Res_{W_k}\left(\ch(\ol E)_k,\ch(\ol F)_{d-k}
\right). 
\end{equation} 
Here $\ch(\ov{E})_k$ denotes the $\Omega(W_k)^{\bmu_d}$-component 
of $\ch(\ov{E})$. 
\end{thm}
\begin{proof} 
Because Polishchuk-Vaintrob considered the $G$-equivariant 
(ungraded) matrix factorizations over the ring of formal power series,   
we need check that the Euler characteristic does 
not change under the base change from the polynomial 
ring to the formal power series ring for \emph{graded} 
matrix factorizations. 
Set $\hR =\CC[\![x_1,\dots,x_N]\!]$. 
Let  $(E,\delta_E), (F,\delta_F)$ be graded matrix factorizations.  
Let $(\hE,\hdelta_E) = (E,\delta_E) \otimes_R \hR$, 
$(\hF,\hdelta_F) = (F,\delta_F)\otimes_R \hR$ 
be $\bmu_d$-equivariant matrix factorizations over $\hR$ 
without the $\ZZ$-grading. 
We have the identification as $\ZZ/2$-graded 
complexes: 
\[
\sfHom^\sigma \left((\hE,\hdelta_E),(\hF,\hdelta_F) \right) = 
\mathbin{\widehat\bigoplus}_{j \equiv \sigma\, (2)} 
\sfHom^j((E,\delta_E),(F,\delta_F)), \quad 
\sigma \in \ZZ/2, 
\]
where the completed direct sum consists of 
arbitrary sequences of homomorphisms  
bounded in the negative direction. 
Hence the cohomology is again the completed 
direct sum of the cohomology $H^j(\sfHom^\bullet((E,\delta_E),(F,\delta_F)))$. 
The HRR for the left-hand side 
implies the finite-dimensionality and the boundedness 
of the cohomology of the right-hand side, and 
the HRR for the right-hand side as well. 
\end{proof}

\begin{rem}
Dyckerhoff \cite{Dyckerhoff} identified 
the Hochschild homology of the category of matrix factorizations 
over a formal power series ring 
with the Jacobi space of the potential.  
Polishchuk-Vaintrob \cite{PVmf} observed that the 
Hochschild homology can be identified with the 
FJRW state space in the $G$-equivariant case. 
The Chern character naturally takes values 
in the Hochschild homology and the Riemann-Roch 
formula was derived in the categorical framework 
in \cite{PVmf}. 
\end{rem} 

\subsection{Orlov equivalence}
\label{subsec:Orlovequiv} 
Under the Calabi-Yau condition $d=\sum_{j=1}^N w_j$, 
Orlov \cite[Theorem 2.5]{Orlov:equivalence} 
constructed the equivalence of triangulated categories  
\begin{equation}
\Phi_l \colon \MFhom^{\rm gr}_{\bmu_d}(W)
\longrightarrow D^b(X_W)
\end{equation} 
parametrized by $l\in \ZZ$. 
Consider a graded matrix factorization 
$\ov{E}=(E^i,\delta_i)_{i\in \ZZ}$ of $W$ 
\[
\dots\to E^0\xrightarrow{\delta_0} E^1\xrightarrow{\delta_1} 
E^2=E^0(d)\xrightarrow{\ \delta_2=\delta_0(d)\ } 
E^3=E^1(d)\xrightarrow{\delta_3=\delta_1(d)\ }\dots
\]
and set 
\[
S=R/(W) = \CC[x_1,\dots,x_N]/(W). 
\]
By tensoring the above data $(E^i, \delta_i)_{i\in \ZZ}$ 
with $S$ over $R$, we obtain an acyclic 
(see Eisenbud \cite{Eis} and Buchweitz \cite{Buch}) 
complex 
\begin{equation*}
\cdots\xrightarrow{\ \delta_{-1}\otimes S\ } 
\Ccal^0\xrightarrow{\ \delta_0\otimes S\ } 
\Ccal^1\xrightarrow{\ \delta_1\otimes S\ }  
\Ccal^2=\Ccal^0(d)\xrightarrow{\ \delta_2\otimes S\ } 
\Ccal^3=\Ccal^1(d) \xrightarrow{\ \delta_3\otimes S\ }  
\cdots.                                    
\end{equation*} 
By construction we can extract from $\Ccal^\bullet$ 
a positively graded and left semiinfinite 
complex $L_0^\bullet$. 
To this effect, after expressing each $\Ccal^i$ 
as a direct sum of $S$-modules of the form $S(k)$ 
for some $k\in \ZZ$, 
we mod out the $S$-modules of the form $S(-e)$ with $e\le 0$. 
More precisely we may notice that 
$E^0$ and $E^1$ have the same dimension and 
can be expresses as  
\[
E^0=\bigoplus_{1\le h\le r} R(-j_h),\qquad 
E^1=\bigoplus_{r+1\le h\le 2r} R(-j_h).
\]
In this way we have 
$\Ccal^0=\bigoplus_{1\le h\le r} S(-j_h)$, 
and $\Ccal^1=\bigoplus_{r+1\le h\le 2r} S(-j_h)$ and 
\[
\Ccal^i=\bigoplus_{\substack{
{1\le h-2r\langle {i}/2\rangle\le r }}} 
S(d\lfloor i/2\rfloor -j_h)
\]
(note that $2r\langle i/2\rangle$ 
equals $0$ or $r$ according to the parity of $i$).
Then, the definition of $L_0^\bullet$ reads  
\begin{equation*}
L_0^i=\bigoplus_{\substack{{1\le h-2r\langle {i}/2\rangle\le r } 
\\{d\lfloor i/2\rfloor <j_{h}}}} 
S\left(d\lfloor i/2\rfloor -j_h)\right). 
\end{equation*}
Since $\Ccal^\bullet$ is acyclic, 
$L_0^\bullet$ is represented by a bounded complex of 
coherent sheaves. 
For simplicity, we stated the definition 
of the positively graded complex $L_0^\bullet$. 
For any $l\in \ZZ$, 
we can define $L_l^\bullet$ as 
\begin{equation}
\label{eq:La}
L_l^i=\bigoplus_{\substack{{1\le h-2r\langle {i}/2\rangle\le r } 
\\{d\lfloor i/2\rfloor -j_h<l}}} 
S\left(d\lfloor i/2\rfloor -j_h)\right).
\end{equation}
This amounts to extracting from each $\Ccal^i$, 
only the $S$-modules of the form $S(-e)$ with $e>l$. 
We have the following statement.  
(We stress that the equivalence of categories 
holds only under the CY condition 
$d=\sum_{j=1}^N w_j$, which we assume throughout the 
paper.)  
\begin{pro}[Herbst-Hori-Page {\cite[\S 10.6, (10.56--58)]{HHP}}] 
\label{pro:Orlovexplicit} 
The Orlov equivalence  
\[
\Phi_l \colon \MFhom^{\rm gr}_{\bmu_d}(W) 
\longrightarrow D^b(X_W)
\] 
for $l\in \ZZ$ assigns to 
$(E,\delta_E)\in \MFhom^{\rm gr}_{\bmu_d}(W)$ 
the left semiinfinite complex \eqref{eq:La}
\[
\Phi_l(E,\delta_E)=L_l^\bullet\in D^b(X_W). 
\] 
Here the graded module $S(k)$ in $L_l^\bullet$ 
is identified with the sheaf $\Ocal(k)$ on $X_W$. 
\end{pro}

\begin{rem}\label{rem:rightsemiinf}
We point out that there are two presentations of 
$\Phi_l(E,\delta_E)$ in the derived category. 
Because the complex $\cC^\bullet$ is acyclic, 
the left semiinfinite complex $L_l^\bullet$ can be 
equivalently represented by 
the (complementary) right semiinfinite complex 
$(L_l^c)^{\bullet}[1]$, where  
\begin{equation}
\label{eq:Lac}
(L_l^c)^i=\bigoplus_{\substack{{1\le h-2r\langle {i}/2\rangle\le r } 
\\{d\lfloor i/2\rfloor -j_h\ge l}}} 
S\left(d\lfloor i/2\rfloor -j_h)\right).
\end{equation}
\end{rem}

\begin{rem}[Herbst-Hori-Page brane transportation] 
\label{rem:branetrans} 
Although we will not use this in the rest of the paper, 
we should mention that 
Orlov functors $\Phi_l$ can be constructed, 
for $\wt R=R[p]$ and $\tW=pW$, 
by lifting the $\bmu_d$-action to a $\CC^\times$-action and 
by obtaining in this way 
a graded and $\CC^\times$-equivariant matrix factorization in 
$\MFhom^{\rm gr}_{\CC^\times}(\tW)$.
Clearly $\bmu_d$-actions are not uniquely lifted to 
$\CC^\times$-actions; we need an extra datum of  
an integer parameter $l$. 
This point of view due to Herbst, Hori, and Page
explains the presence of several Orlov functors 
$\Phi_l$ for $l\in \ZZ$. 
From $\MFhom^{\rm gr}_{\CC^\times}(\tW)$
a natural functor leads to $D^b(X_W)$, 
see \cite{HHP}.
\end{rem}

We apply Orlov's functor $\Phi_l$ to 
the graded Koszul matrix factorization $\{\ua,\ub\}_q$  
from Example \ref{exa:ch-Koszul}  
(see also Definition \ref{defn:gradedKoszul-MF}). 
\begin{pro}
\label{pro:OrlovKoszul} 
The image via $\Phi_l$ of the graded matrix factorization 
$\{\ua,\ub\}_q$ in Example \ref{exa:ch-Koszul} 
is represented by the complex on $X_W$ 
\begin{equation*}
\bigoplus_{\substack{{j_1<j_2<\dots<j_r}\\ 
\sum_{a=1}^r w_{j_a}\le m}} 
\cO\left(l+ m - \textstyle
\sum_{a=1}^r w_{j_a}\right) e_{j_1}\wedge \cdots \wedge e_{j_r} 
\left[r + 1 +  2 t \right]  
\end{equation*}  
equipped with the Koszul differential 
$\delta''_{\ub} = \sum_{j=1}^N x_j \iota(e_i^*)$.  
Here $t$ and $m$ denote the integer quotient 
and remainder of $q-l$ divided by $d$. 
\end{pro}
\begin{proof}
Write $(E^i, \delta_i)_{i\in \ZZ} = \{\ua,\ub\}_q$. 
Let us consider $E^i\otimes_R S$
\begin{equation}
\label{eq:parity}
\bigoplus_{r\in \ZZ\mid r\equiv i\, (2)} 
S\otimes \Bigl(\bigwedge^{r} V \Bigr) 
\left(\textstyle \frac{d(i-r)}2+q\right). 
\end{equation}
where $V = \bigoplus_{j=1}^N \CC e_j$ with 
$\deg (e_j) = -\deg (a_j) = w_j -d$.  
Each summand is of the form 
\begin{equation}
\label{eq:terms}
\bigoplus_{{j_1<\dots<j_r}}
S\left( \textstyle - \sum_{a=1}^r \deg(e_{j_a}) + \frac{d(i-r)}2 +q \right) 
= \bigoplus_{{j_1<\dots<j_r}}
S\left(\textstyle -\sum_{a=1}^r w_{j_a}+\frac{d(i+r)}2 +q \right). 
\end{equation}
By Proposition \ref{pro:Orlovexplicit} and 
Remark \ref{rem:rightsemiinf} 
we can regard the image via $\Phi_l$ of 
the Koszul matrix factorization $\{\ua,\ub\}_q$ 
as the (complementary) right semiinfinite complex 
$(L_l^c)^{\bullet} [1]$. 
The terms $S(h)$ appearing in the above formula 
contribute to $L_l^c$ if and only if $h\ge l$; 
therefore we consider the inequality
\begin{equation*}
h=-\sum_{a=1}^r w_{j_a}+ \frac{i+r}{2}d + q \ge l, 
\end{equation*} 
which can be rewritten as (using $q- l = td +m$) 
\[
m +t d +  \frac{i+r}{2} d
\ge \sum_{a=1}^r w_{j_a}.
\]
Since $\sum_{a=1}^r w_{j_a}$ lies in $\{0,\dots,d\}$ 
by the CY condition $d=\sum_{j=1}^N w_j$,  
we deduce that $h\ge l$ 
if and only if either we have 
$(i+r)/2 > - t$ 
(note $i+r$ is even by \eqref{eq:parity}) 
or we have $m \ge \sum_{a=1}^r w_{j_a}$ 
alongside with $({i+r})/2 =-t$. 

Let us consider all terms of \eqref{eq:terms} 
for which $(i+r)/2> -t$. 
Then, the summand of \eqref{eq:terms} 
attached to $j_1<\dots<j_r$ is of the form 
$S(l+n-\sum_{a=1}^r w_{j_a})$ with $n\ge d$. 
Such summands with fixed $n$ 
form an exact sequence $\cE_n^\bullet$ on $X_W$
\begin{multline}
\label{eq:exactseq}
\cE_n^\bullet: \quad 
\Ocal(l+n)\overset{\delta_{\ub}''}
\longleftarrow \bigoplus_j \Ocal(l+n-w_j)
\overset{\delta''_{\ub}}\longleftarrow 
\bigoplus_{j_1<j_2}\Ocal(l+n-w_{j_1}-w_{j_2})
\overset{\delta''_{\ub}}\longleftarrow  \cdots \\ 
\overset{\delta''_{\ub}}{\longleftarrow} 
\bigoplus_j \Ocal\left( l + n -\textstyle\sum_{j'\neq j} w_{j'}\right)
\overset{\delta''_{\ub}}\longleftarrow  \Ocal(l +n-d)
\end{multline} 
where we wrote $\Ocal(h)$ for $S(h)$ 
following Proposition \ref{pro:Orlovexplicit}. 
Therefore, all together, the sum  
\[ 
\bigoplus_{\substack{r\equiv i \, (2) \\ 
(i+r)/2 > -t }} 
S\otimes \Bigl( \bigwedge^r V\Bigr) \left(  
\textstyle \frac{d(i-r)}{2} +q \right) 
\] 
gives an acyclic subcomplex of $(L^{c}_l)^\bullet$. 
It is acyclic because it can be written 
as a successive extension by the 
acyclic complexes $\cE_n^\bullet$ 
of a complex supported on arbitrarily 
high homological degrees.  
The quotient of $(L^c_l)^\bullet$ 
by this acyclic subcomplex consists of 
terms of \eqref{eq:terms} 
with $(i+r)/2 =-t$ 
and $\sum_{a=1}^r w_{j_a} \le m$. 
The conclusion follows. (Recall that 
we need to take the shift $(L^c_l)^\bullet[1]$ by 1.)  
\end{proof}

\subsection{Twisted $I$-functions and 
Mellin-Barnes continuation} 
\label{subsect:Ifunct-and-MB}
We provide two parallel discussions of 
the twisted $I$-functions for GW and FJRW theories. 
We show that the two $I$-functions satisfy 
the same Picard-Fuchs equation under a co-ordiante change. 
We compute the connection matrix between the two $I$-functions 
(or more precisely the $\Hf$-functions) 
using the Mellin-Barnes method of analytic continuation. 

On both sides we systematically work 
with the $e_{\CC^\times}$-twisted theories. 
On the Landau-Ginzburg side we already discussed 
the $e_{T}$-twisted FJRW theory \S\ref{subsec:twist-equivEuler}
over the extended state space; 
its non-equivariant limit, followed by 
projection to the narrow state space, 
encodes the genus zero correlator of FJRW theory. 
The counterpart on the Calabi-Yau side 
is the $e_{\CC^\times}$-twisted theory of 
$\PP(\uw)$, twisted by $\cO(d)$. 
It is treated and computed in genus zero in \cite{CCLT};
again, the non-equivariant limit, followed 
by the projection to the ambient part $H_{\amb}(X_W)$ of 
the state space yields the genus-zero 
correlators in GW theory.  
(See \S \ref{subsec:twisted-q-connection} 
for a review.)

\subsubsection{The $e_{\CC^\times}$-twisted $I$-functions} 
\label{subsubsect:Ifunct}
Recall the $e_T$-twisted $I$-function  \eqref{eq:FJRW-Ifunction}
with $N$ equivariant parameters $\lambda_1,\dots,\lambda_N$.  
Here, without loss of information from the point of view of 
non-equivariant theory,  
we can impose the conditions $\lambda_j= -q_j \lambda$ for all $j$ 
with a single equivariant parameter $\lambda$ 
(as in \S \ref{subsubsec:summary}). 
The $e_{\CC^\times}$-twisted $I$-function in FJRW theory 
is given by: 
\[
I^\tw_{\FJRW}(u,z;\lambda)=
z\sum_{k\in \ZZ_{\ge 1}} u^k 
\frac{\prod_{j=1}^N\prod_{ 
0< b<kq_j, \fracp{b} =\fracp{kq_j}} 
(-q_j \lambda-bz)}
{\prod_{ 0< b< k, \fracp{b} =0}(-bz)} \phi_{k-1}.
\]
Here the index $k-1$ of $\phi_{k-1}$ is reduced modulo $d$ 
within the range $\{0,\dots,d-1\}$. 
This takes values in the extended state space $H_{\ext}$ 
\eqref{eq:extendedstate}. 

In GW theory, the $e_{\CC^\times}$-twisted
$I$-function was computed in \cite{CCLT}.  
It is given by: 
\[
I^\tw_{\GW}(v,z;\lambda)= 
ze^{p\log v/z} 
\sum_{\substack{n\in \QQ_{\ge 0} \\ 
\exists j, \, nw_j\in \ZZ}}
v^{n} 
\frac{\prod_{0< b \le dn, \fracp{b} = 0} (dp+\lambda+bz)}
{\prod_{j=1}^N 
\prod_{0< b\le w_j n, \fracp{b} = \fracp{w_j n} } 
(w_jp+bz)}
\unit_{\langle -n\rangle}.
\]
This encodes the $e_{\CC^\times}$-twisted GW 
invariants of $\PP(\uw)$, twisted by the line bundle $\cO(d)$,  
and takes values in $H_{\CR}(\PP(\uw))$. 

These twisted $I$-functions $I^{\tw}_{\FJRW}(u,z)$ and $I^{\tw}_{\GW}(v,z)$ 
are convergent respectively on the regions $\{|u|<\vc^{-1/d}\}$ 
and $\{|v|<\vc\}$, where $\vc:= d^{-d}\prod_{j=1}^N w_j^{w_j}$, 
see Lemma \ref{lem:convergence}.  

%
%

\subsubsection{Picard-Fuchs equations} 
The $I$-function $I^\tw_{\FJRW}$ is a solution of 
the Picard-Fuchs equation 
\begin{equation}
\label{eq:PFforRW}
\left[
u^d \prod_{j=1}^N 
\prod_{c=0}^{w_j-1} \left( -q_j zD_u -q_j \lambda - c z\right)-
\prod_{c=1}^d\left(-zD_u + cz\right)
\right]
I=0,
\end{equation}
for $D_u=u (\partial/\partial u)$.
The $I$-function $I^\tw_{\GW}$ is a solution of 
the Picard-Fuchs equation 
\begin{equation}
\label{eq:PFforGW}
\left[\prod_{j=1}^N \prod_{c=0}^{w_j-1}
\left({w_j} zD_v - cz\right)-
v\prod_{c=1}^d\left(dzD_v+\la+ c z \right)\right]I=0
\end{equation}
for $D_v=v (\partial/\partial v)$.

Under the change of variable $u= v^{-1/d}$ and 
conjugation with the operator $u^{-\la/z}=v^{\la/dz}$ 
the two equations coincide. 
This happens because we have $dD_v=-D_u$ and 
$v^{-\la/dz} \circ (dzD_v) \circ v^{\la/dz}= dzD_v+\la$.
In particular the limits for 
$\la\to 0$ match under $v=u^{-d}$. 
(We remedy the discrepancy of the equivariant 
Picard-Fuchs equations by introducing the unit 
co-ordinate $t^0$ (or $s^0$) later 
in \S \ref{subsec:refined-mirrorthm}.) 
The components of each of the $I$-functions give a basis 
of solutions to the Picard-Fuchs equation for generic $\lambda$ 
(cf.\ Proposition \ref{pro:basisofsol-cFtw}, 
Lemma \ref{lem:almostgenerate} and \eqref{eq:GKZrelation_of_Delta}).

\subsubsection{The $\Hf$-functions} 
\label{subsubsec:Hfunct}
We introduce a constant linear transform of the $I$-function, 
the $\Hf$-function, 
which is more compatible with the $\hGamma$-integral structure
in \S \ref{subsubsec:intstr}. 
The relevance of such hypergeometric series in 
homological mirror symmetry was observed 
by Horja \cite{Ho}, Hosono \cite{Hosono:centralcharge} 
and Borisov-Horja \cite{Borisov-Horja:FM}. 
The $\Hf$-function 
is defined by the relation\footnote
{See \S \ref{subsec:ancont-revisited}, \eqref{eq:Hfcn-meaning} 
for a precise relationship between the $\Hf$-function and 
the $\hGamma$-integral structure.}  
(cf.\ \eqref{eq:K-framing}): 
\begin{equation}
\label{eq:Ifunct-Hfunct}
I^\tw(x,z;\lambda) = z^{-\grading} \hGamma^\tw 
\left( 
(2\pi\iu)^{\frac{\deg_0}{2}}
\Hf^\tw(x,z;\lambda)\right).      
\end{equation} 
Here the operators $\hGamma^\tw$, $\grading$, $\deg_0$ 
in the respective theory are defined as follows:  
In the twisted FJRW theory, the \emph{twisted Gamma class} 
$\hGamma_{\FJRW}^\tw$ operating on 
the extended state space $H_{\ext}$ is defined to be 
\[
\hGamma^\tw_{\FJRW} := \bigoplus_{k=0}^{d-1} 
\prod_{i=1}^N \Gamma\left(1-\fracp{kq_j}-q_j\xi \right), 
\quad \xi = \lambda/z. 
\]
In the twisted GW theory, 
the \emph{twisted Gamma class}  
$\hGamma_{\GW}^\tw$ operating on $H_{\CR}(\PP(\uw))$ 
is defined to be 
\[
\hGamma_{\GW}^\tw:= \bigoplus_{f\in \frF} 
\frac{
\prod_{i=1}^N
\Gamma(1 - \fracp{fw_i}+ w_i p)}
{\Gamma(1+\xi+d p)}, \quad 
\xi = \lambda/z. 
\]
The non-equivariant limits $\lambda \to 0$ are well-defined and 
induce $\hGamma_{\FJRW}$ and $\hGamma_{\FJRW}$ 
in Definition \ref{defn:Gamma}  
under the projection to the original state spaces. 
The grading operator $\grading$ on $H_{\ext}$ or on 
$H_{\CR}(X_W)$ is given by 
\[
\grading (T_i ) = \frac{\deg T_i}{2} T_i 
\]
where ``$\deg$" denotes the degree defined in 
\eqref{eq:deg-extendedstate} for FJRW theory 
and the age-shifted degree of orbifold 
cohomology classes of $\PP(\uw)$ for GW theory. 
The ``bare" degree operator $\deg_0$ on $H_{\ext}$ 
or on $H_{\CR}(\PP(\uw))$ is defined by 
(cf.\ Definition \ref{defn:intstr}) 
\begin{alignat*}{2} 
\deg_0(\phi_k) &= -2 \phi_k  
& & \text{for twisted FJRW theory;} \\
\deg_0(p^n \unit_f) & = 2n (p^n \unit_f) \quad 
& & \text{for twisted GW theory.}
\end{alignat*}
 
On the Landau-Ginzburg side, we have 
\begin{align*}
I_{\FJRW}^\tw (u,z;\lambda)
&=z^{-\grading} z 
\sum_{k\in \ZZ_{\ge 1}} u^k\frac{(-1)^{k-1}}{\Gamma(k)}  
\prod_{j=1}^N
\frac{\Gamma(\fracp{-q_j k}-q_j\xi)}
{\Gamma(1-q_j(k+\xi))} 
\phi_{k-1} \\
&= z^{-\grading} z 
\sum_{k\in \ZZ_{\ge 1}} u^k\frac{(-1)^{k-1}}{\Gamma(k)}  
 \frac{1}{\prod_{j: k q_j \in \ZZ}(-q_j \xi)} 
\prod_{j=1}^N\frac{\Gamma(1-\fracp{q_j k}-q_j\xi)}{ 
\Gamma(1-q_j(k+\xi)) } 
\phi_{k-1} \\
&= z^{-\grading} \hGamma^\tw_{\FJRW} 
\left( (2\pi\iu)^{\frac{\deg_0}{2}} 
\Hf_{\FJRW}^\tw (u,z;\lambda) \right),
\end{align*}
where $\xi := \lambda/z$ and 
\begin{equation}
\label{eq:HfunctRW}
\Hf_{\FJRW}^{\tw}(u,z;\lambda)=z 
\sum_{k\in \ZZ_{\ge 1}} u^k 
\frac{(-1)^{k-1} (2 \pi\iu)}
{\Gamma(k)
\prod_{j: k q_j \in \ZZ} (-q_j \xi ) 
\prod_{j=1}^N 
\Gamma(1-q_j(k+\xi))}  
\phi_{k-1}.
\end{equation}


On the Calabi-Yau side, we have 
\begin{align*}
I_{\GW}^\tw(v,z;\lambda)&=
z^{-\grading} z e^{p\log v} 
\sum_{\substack{n\in \QQ_{\ge 0}\\ 
\exists j,\, nw_j\in \ZZ}}
v^{n}
\frac{\Gamma(1+dp+\xi+dn)}{\Gamma(1+dp+\xi)}
\prod_{j=1}^N 
\frac{\Gamma(1+w_jp-\fracp{-w_j n})}{\Gamma(1+w_j p+w_j n)}
\unit_{\langle -n\rangle}\\
 &= z^{-\grading} \hGamma^\tw_{\GW} 
\left( (2\pi\iu)^{\frac{\deg_0}{2}} 
\Hf_{\GW}^\tw(v,z;\lambda)\right) ,
\end{align*}
where $\xi := \lambda/z$ and 
\begin{equation}
\label{eq:HfunctGW} 
\Hf_{\GW}^\tw(v,z;\lambda) = z e^{\frac{p}{2\pi\iu} \log v} 
\sum_{\substack{n\in \QQ_{\ge 0}\\\exists j, nw_j\in \ZZ}}
v^{n}
\frac{\Gamma(1+d\frac{p}{2\pi\iu}+\xi+dn)}
{\prod_{j=1}^N\Gamma(1+w_j\frac{p}{2\pi\iu}+w_jn)}
\unit_{\langle -n\rangle}.
\end{equation}

\subsubsection{Mellin-Barnes analytic continuation}
\label{subsubsec:MB} 
The function $\Hf_{\GW}^\tw(v,z;\lambda)$ is convergent and analytic 
on the region $\Re(\log v) < \log \vc$, 
where $\vc := d^{-d} \prod_{j=1}^N w_i^{w_i}$ is the singularity 
of the Picard-Fuchs equation \eqref{eq:PFforGW}. 
Similarly $\Hf_{\FJRW}^\tw(u,z;\lambda)$ is convergent and analytic 
on the region $\Re(\log u) < -(\log \vc)/d$. 
Let $\tcMo$ denote the $(\log v)$-plane 
minus the singularities of the Picard-Fuchs equation: 
\begin{equation} 
\label{eq:cycliccover-of-M}
\tcMo = \CC_{\log v}  \setminus 
\left\{ \log v_c + 2l \pi \iu\ |\ l\in \ZZ \right\}. 
\end{equation} 
Under the identification $\log v = -d \log u$, we regard 
$\Hf^\tw_\GW$ as a single-valued function 
in the left-half of $\tcMo$ 
and $\Hf^\tw_\FJRW$ as a  single-valued function 
on the right-half of $\tcMo$. 
Let $\gamma_l\subset \tcMo$ be a path 
from the large radius limit 
($\Im(\log v) =0, \Re(\log v) \ll 0$) to the LG limit 
($\Im(\log v) =0, \Re (\log v) \gg 0$) 
which passes through the ``window" 
$[\log \vc + 2(l-1)\pi\iu, \log \vc + 2l \pi\iu]$. 
See Figure \ref{fig:ancont-path}. 
We consider analytic continuation along the path $\gamma_l$. 

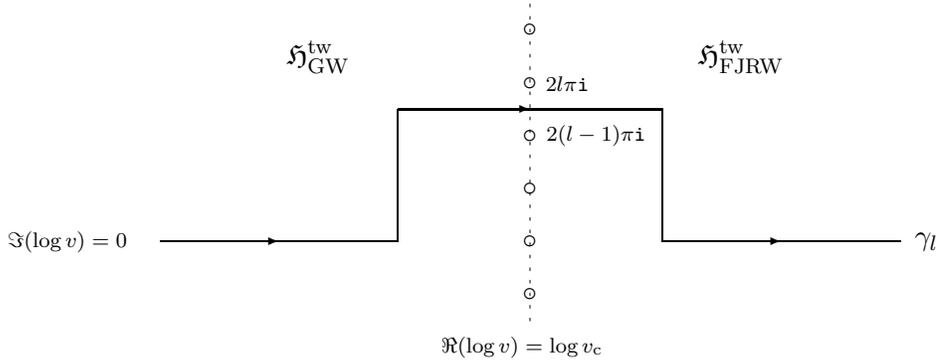
\begin{figure}[htbp] 
\begin{center} 
\begin{picture}(300,140)  
\put(150,20){\makebox(0,0){$\circ$}}
\put(150,40){\makebox(0,0){$\circ$}}  
\put(150,60){\makebox(0,0){$\circ$}}  
\put(150,80){\makebox(0,0){$\circ$}}  
\put(150,100){\makebox(0,0){$\circ$}}  
\put(150,120){\makebox(0,0){$\circ$}}  
\put(150,10){\dashline{1.5}(0,0)(0,120)}

\put(10,40){\vector(1,0){45}}
\put(55,40){\line(1,0){45}} 
\put(100,40){\line(0,1){50}} 
\put(100,90){\vector(1,0){50}}
\put(150,90){\line(1,0){50}}
\put(200,90){\line(0,-1){50}}
\put(200,40){\vector(1,0){45}}
\put(245,40){\line(1,0){45}} 


\put(175,80){\makebox(0,0){\scriptsize $2(l-1)\pi \iu$}}
\put(164,100){\makebox(0,0){\scriptsize $2l \pi \iu$}}
\put(147,0){\makebox(0,0){\scriptsize $\Re(\log v) = \log \vc$}}

\put(70,110){\makebox(0,0){$\Hf^\tw_{\GW}$}}
\put(230,110){\makebox(0,0){$\Hf^\tw_{\FJRW}$}} 

\put(-25,40){\makebox(0,0){\scriptsize $\Im(\log v)=0$}}
\put(300,40){\makebox(0,0){$\gamma_l$}}
\end{picture} 
\end{center} 
\caption{The analytic continuation path $\gamma_l$ 
on the $(\log v)$-plane.}
\label{fig:ancont-path} 
\end{figure}

We rewrite $\Hf_{\GW}^\tw$ by expressing the running index 
$n$ as an element of $\frF+\ZZ_{\ge 0}$. 
For $f\in \frF$, we adopt the notation $\ol f=\fracp{1-f}$. 
We get 
\[
\Hf_{\GW}^\tw (v,z;\lambda)
=z \sum_{f\in \frF} \sum_{k\in \ZZ_{\ge 0}} 
\frac{\Gamma(1+\xi +d\frac{p}{2\pi\iu} +d  \ov{f} + dk )}
{\prod_{j=1}^N \Gamma(1+w_j\frac{p}{2\pi\iu} +w_j\ol f+w_jk)}
v^{\frac{p}{2\pi\iu} +\ol f+k} 
\unit_f.
\]
During the analytic continuation, 
we regard $p$ as a small complex number  
and think of the $\Hf$-function as a scalar valued function. 
At the end of the calculation, 
we take the Taylor expansion in $p$ 
and replace $p$ with the hyperplane class. 
In this way we get analytic continuation of 
a cohomology-valued function. 
We write the sum over $\ZZ_{\ge 0}$ 
as a sum of residues:
\begin{multline*}
z \sum_{f\in \frF} \unit_f 
\sum_{k\in \ZZ_{\ge 0}} \Res_{s=k} ds 
\left(\Gamma(s)\Gamma(1-s)
\frac{\Gamma(1+\xi + d(\frac{p}{2\pi\iu} +\ol f+s))}
{\prod_{j=1}^\nn \Gamma(1+w_j(\frac{p}{2\pi\iu} +\ol f+s))} 
e^{-(2l-1)\pi\iu s} e^{(\frac{p}{2\pi\iu} + \ov{f}+s) \log v}  \right). 
\end{multline*}
Here $l\in \ZZ$ is the index of the path $\gamma_l$. 
Consider the contour integrals along the path of 
Figure \ref{fig:contour} 
of each $1$-form in the above expression.
\begin{figure}[htbp]
\begin{picture}(500,150)(-200,-10)
  \qbezier[5](-100,0)(0,0)(100,0)
  \qbezier[5](-100,40)(0,40)(100,40)
  \qbezier[5](-100,80)(0,80)(100,80)
  \qbezier[5](-100,120)(0,120)(100,120)
  \put(60,40){\circle*{3}}
  \put(65,35){\footnotesize{0}}
  \put(100,40){\circle*{3}}
  \put(105,35){\footnotesize{1}}
  \put(20,40){\circle*{3}}
  \put(25,35){\footnotesize{-1}}
  \put(-20,40){\circle*{3}}
  \put(-15,35){\footnotesize{-2}}
  \put(-60,40){\circle*{3}}
  \put(-55,35){\footnotesize{-3}}
  \put(-100,40){\circle*{3}}
  \put(-95,35){\footnotesize{-4}}
  \put(65,95){\circle*{3}}
  \put(57,95){\circle*{3}}
  \put(49,95){\circle*{3}}
  \put(41,95){\circle*{3}}
  \put(33,95){\circle*{3}}
  \put(25,95){\circle*{3}}
  \put(17,95){\circle*{3}}
  \put(9,95){\circle*{3}}
  \put(1,95){\circle*{3}}
  \put(-7,95){\circle*{3}}
  \put(-15,95){\circle*{3}}
  \put(-23,95){\circle*{3}}
  \put(-31,95){\circle*{3}}
  \put(-39,95){\circle*{3}}
  \put(-47,95){\circle*{3}}
  \put(-55,95){\circle*{3}}
  \put(-63,95){\circle*{3}}
  \put(-71,95){\circle*{3}}
  \put(-79,95){\circle*{3}}
  \put(-87,95){\circle*{3}}
  \put(-95,95){\circle*{3}}
  \put(-103,95){\circle*{3}}
  \put(-111,95){\circle*{3}}
%
  \put(40,90){\vector(0,-1){95}}
  \put(40,90){\line(1,0){30}}
  \put(70,90){\line(0,1){10}}
  \put(70,100){\line(-1,0){30}}
  \put(40,100){\line(0,1){30}}
\end{picture}
\caption{the contour of integration on the $s$-plane}
\label{fig:contour}
\end{figure}
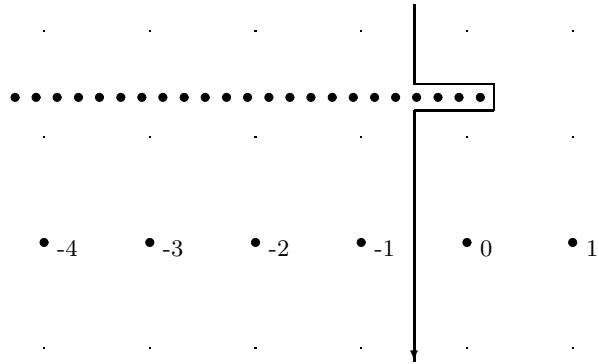
The integrals are absolutely convergent 
(and define analytic functions of $v$)  
if $\abs{\Im (\log v) - (2l-1)\pi}<\pi$ 
(see e.g.\ \cite[Lemma 3.3]{Ho}). 
This condition is satisfied when $\log v$ 
is along (the middle part of) the path $\gamma_l$. 
When $|v|<\vc$, we can close the contour 
to the right and obtain the above sum of residues.
On the other hand, if $|v|>\vc$, 
we can close the contour to the left and  
obtain the sum of residues at
$s=-m\ (m\in \ZZ_{\ge 1})$
plus the sum of residues at
\begin{equation*}
s=-\left(\frac{1+\xi +k}d+\frac{p}{2\pi\iu} +\ol f\right) 
\qquad \text{for $k\in \ZZ_{\ge 0}$}.
\end{equation*}
The sum of these residues gives  
\begin{align}
\label{eq:residuesum-left}  
\begin{split} 
& - z \sum_{f\in \frF} \unit_f 
\sum_{m =1}^\infty  
\frac{\Gamma(1+\xi + d(\frac{p}{2\pi\iu} +\ov{f} - m ))}
{\prod_{j=1}^\nn \Gamma(1+w_j(\frac{p}{2\pi\iu} +\ov{f} -m))} 
e^{(\frac{p}{2\pi\iu} + \ov{f}-m) \log v}  \\ 
& - z \sum_{f\in \frF} \unit_f 
\sum_{k=0}^\infty \frac{\pi}
{\sin \left(-\left( \frac{1+\xi +k}{d} + 
\frac{p}{2\pi\iu} +\ov{f} \right) \pi \right)} 
\frac{(-1)^k}{d \cdot k!}
\frac{e^{(2l-1)\pi\iu 
(\frac{p}{2\pi\iu} + \ov{f} + \frac{1+\xi+k}{d})}}
{\prod_{j=1}^N \Gamma(1 -q_j(1+\xi+k))}
u^{1+\xi+k}. 
\end{split} 
\end{align} 
Here the overall minus sign appears because 
the contour closed to the left encloses each pole clockwise. 
We also used the co-ordinate change $\log u = -\log v/d$. 

We now regard $p$ as the hyperplane class on $\PP(\uw)$. 
The first term of \eqref{eq:residuesum-left} 
vanishes in cohomology because the class 
\[
\prod_{j: w_j\ov{f}\in\ZZ} 
\frac1{\Gamma(1+w_j\frac{p}{2\pi\iu} +w_j(\ol f -m))}
=O(p^{\sharp\{j \,|\, w_j\ov{f} \in\ZZ\}})
\]
is zero on the sector $\PP(\uw)_f$. 
(Note that $\PP(\uw)_f$ is of dimension 
$\sharp\{j\mid w_j\ov{f}\in \ZZ\}-1$.)  
By shifting the index $k$ by $1$ and using 
$\sin(x) = (e^{\iu x}-e^{-\iu x})/2\iu$, 
we can rewrite the second term of \eqref{eq:residuesum-left} as  
\[
z \sum_{f\in \frF} \unit_f \sum_{k=1}^{\infty} 
\frac{1}{d} 
\frac{\left(\zeta^k e^{p+ 2\pi\iu(\ov{f} + \frac{\xi}{d})}\right)^l}
{\zeta^k e^{p+ 2\pi\iu (\ov{f} +\frac{\xi}{d})}-1} 
\cdot 
\frac{2\pi\iu (-1)^{k-1}u^{\xi+ k}}
{(k-1)!\prod_{j=1}^N \Gamma(1-q_j (k+ \xi)) }.   
\]
This expression is regular at $p=0$ and can be regarded 
as an $H_{\CR}(\PP(\uw))$-valued function. 
This is the analytic continuation of $\Hf^\tw_\GW$ 
along the path $\gamma_l$. 
Comparing this with $\Hf^\tw_{\FJRW}$ \eqref{eq:HfunctRW}, 
we have the following proposition: 
\begin{pro} 
\label{pro:Utw}
Define a linear transformation $\UU^\tw_l 
\colon H_{\ext} \to H_{\CR}(\PP(\uw))$  
depending on $l\in \ZZ$ and the parameter $\xi = \lambda/z$ by 
\begin{equation} 
\label{eq:U} 
\UU^\tw_l (\phi_{k-1}) = 
\frac{1}{d}  
\sum_{f\in \frF} \unit_f 
\frac{\left(
\zeta^k e^{p+ 2\pi\iu(\ov{f} + \frac{\xi}{d})}
\right)^l}
{\zeta^k e^{p+ 2\pi\iu (\ov{f} +\frac{\xi}{d})}-1}
\prod_{j:kq_j\in \ZZ} (-q_j \xi), 
\quad k=1,\dots,d. 
\end{equation} 
Then we have 
\[
u^{-\xi} (\Hf^\tw_\GW)_{\rm continued} = 
\UU^\tw_l \left(\Hf^\tw_{\FJRW}\right).    
\]
where $(\Hf^\tw_\GW)_{\rm continued}$ is the analytic 
continuation of $\Hf^\tw_\GW$ along the path $\gamma_l$. 
\end{pro} 
\begin{rem} 
By Proposition \ref{pro:Utw} and \eqref{eq:Ifunct-Hfunct}, 
we can find the connection matrix of the twisted $I$-functions. 
We have 
$u^{-\xi} (I^\tw_{\GW})_{\rm continued} = 
\tUU^{\tw}_l ( I^\tw_{\FJRW})$ for the transformation 
\[
\tUU^{\tw}_l = 
z^{-\grading} \circ \hGamma^\tw_\GW \circ (2\pi\iu)^{\frac{\deg_0}{2}} \circ 
\UU^\tw_l \circ (2\pi\iu)^{-\frac{\deg_0}{2}} \circ (\hGamma^\tw_{\FJRW})^{-1} 
\circ z^{\grading}. 
\]
The non-equivariant limit of this induces a linear transformation 
between the Givental symplectic vector spaces of  FJRW theory 
and GW theory. 
This is the symplectic transformation 
computed in \cite{ChiRquintic} for a quintic.  
\end{rem} 

\subsection{The non-equivariant limit and Orlov equivalence}
\label{subsec:noneqlimit-Orlov}
Here we show that the non-equivariant limit of $\UU^\tw_l$ 
exists and descends to a linear transformation between 
the narrow and the ambient part state spaces. 
We show that it matches with the numerical 
Orlov equivalence. 

\subsubsection{The narrow-to-ambient linear transformation} 
\begin{pro} 
\label{pro:noneqlimit-U} 
The non-equivariant limit $\lambda\to 0$ of $\UU_l^\tw$ exists. 
We have 
\[
\lim_{\la\to 0} (\UU^\tw_l(\phi_{k-1})) =  
\begin{cases} 
\displaystyle 
\frac{1}{d} \sum_{f \in \frF}   
\frac{\left(\zeta^{k}e^{p+ 2\pi\iu\ov{f}}\right)^l}
{{\zeta^k e^{p+ 2\pi\iu \ov{f}}-1}} \unit_f
& \text{for $k\in \Nar$;} \\
\displaystyle 
- p^{N_k -1 } 
\unit_{\langle \frac{k}{d} \rangle} 
\frac{\zeta^{kl}}{d}
\prod_{j:kq_j\in \ZZ} \frac{w_j}{2\pi\iu}  
& 
\text{for $k\not\in \Nar$.}
\end{cases} 
\] 
where $k=1,\dots,d$ and 
$N_k := 1 + \dim \PP(\uw)_{\fracp{k/d}}
= \sharp\{j\,|\,k q_j \in \ZZ\}$.  
\end{pro} 
\begin{proof} 
We take the Taylor expansion of the expression \eqref{eq:U} 
in $p$ first and check if the expansion are regular at $\xi=0$ 
when evaluated in $H_{\CR}(\PP(\uw))$.  

If $k\in \Nar$, or equivalently $\fracp{k/d} \notin \frF$, 
there exists no $f\in \frF$ such that $\zeta^k e^{2\pi\iu \ov{f}}=1$. 
Therefore \eqref{eq:U} is regular 
at $(p,\xi)=(0,0)$ 
and the conclusion follows. 

If $k\notin \Nar$, 
\eqref{eq:U} is not regular at $(p,\xi)=(0,0)$. 
The only non-regular term in \eqref{eq:U} is the one with 
$f = \fracp{k/d}$ (in this case $\zeta^k e^{2\pi\iu \ov{f}} =1$). 
We compute the Taylor expansion in $p$ 
of such term. 
By an elementary computation, we have 
\begin{align*} 
\frac{1}{e^{p+\frac{2\pi\iu\xi}{d}} -1} 
& = \sum_{n=0}^\infty \beta_n(\xi)  p^n, \quad 
\beta_n(\xi) = (-1)^n 
\left(\frac{2\pi\iu\xi}{d}\right)^{-n-1} + O(\xi^{-n}).  
\end{align*} 
When evaluated in the cohomology group $H(\PP(\uw)_f)$, 
this Taylor series is truncated at $n=\dim \PP(\uw)_f 
= N_k-1$ 
(where we used $f = \fracp{k/d}$). 
Therefore the factor $\prod_{j:kq_j \in \ZZ} (-q_j \xi)$ 
cancels all the negative powers of $\xi$ in $\beta_n$. 
Hence $\UU^\tw_l(\phi_{k-1})$ is regular at $\xi=0$ 
and the conclusion follows. 
\end{proof} 

We have natural projections  
$H_{\ext} \to H_{\nar}(W,\bmu_d)$, 
$H_\CR(\PP(\uw)) \to H_{\amb}(X_W)$ 
from the state spaces of the twisted theory 
to the narrow/ambient part of the state spaces. 
We denote this projection by $\pr$. 
By Proposition \ref{pro:noneqlimit-U}, 
$\lim_{\lambda \to 0}\UU^\tw$ descends to these 
projections. 
\begin{cor} 
\label{cor:Udescends}
Define a linear transformation $\UU_l \colon H_{\nar}(W,\bmu_d) 
\to H_{\amb}(X_W)$ by 
\begin{equation} 
\label{eq:Unarrow} 
\UU_l (\phi_{k-1}) = \frac{1}{d} \sum_{f \in \frF}   
\frac{\left(\zeta^{k}e^{p+2\pi\iu\ov{f}}\right)^l}
{{\zeta^k e^{p+2\pi\iu \ov{f}}-1}} \unit_f
\end{equation} 
Then we have the commutative diagram: 
\begin{equation*} 
\begin{CD} 
H_\ext @>{\lim_{\lambda\to 0} \UU^\tw_l}>> H_\CR(\PP(\uw)) \\ 
@V{\pr} VV @VV{\pr}V \\ 
H_{\nar}(W,\bmu_d) @>{\UU_l}>> H_{\amb}(X_W) 
\end{CD} 
\end{equation*} 
The operator $\UU_l$ gives a connection between the non-equivariant 
limit of $\Hf$-functions, i.e.\ 
$\Hf_{\GW} = \UU_l(\Hf_{\FJRW})$ for 
$\Hf_{\heartsuit} := \pr (\lim_{\lambda \to 0} \Hf^\tw_{\heartsuit})$. 
\end{cor}

\subsubsection{The analytic continuation matches Orlov equivalences}
Via the Chern character, the linear transformations 
$\UU_l$ match the Orlov equivalences $\Phi_l$. 
To show this we use the explicit expression 
for Orlov's equivalence 
for Koszul matrix factorizations 
(Proposition \ref{pro:OrlovKoszul})  
and the equation \eqref{eq:chernbrane} 
for the Chern character. 

\begin{lem} 
\label{lem:yformula} 
We have 
\begin{equation*}
\frac{1}{d} \sum_{k=0}^{d-1}
\frac{\zeta^{kj}}{\zeta^ky-1}= 
\frac{y^{d\langle -j/d\rangle}}{y^d-1},   
\end{equation*} 
where $d\langle -j/d\rangle$ is simply $-j$ 
reduced modulo $d$ within $\{0,1,\dots,d-1\}$. 
\end{lem} 
\begin{proof} 
Note that $(1/d)\sum_{k=0}^{d-1} \zeta^{qk}$ 
equals $1$ if $q\in d\ZZ$ and $0$ otherwise. 
Thus we have 
\[
\frac{1}{d} \sum_{k=0}^{d-1}
\frac{\zeta^{kj}}{\zeta^ky-1}= 
-\frac{1}{d} 
\sum_{k=0}^{d-1} 
\sum_{n=0}^\infty 
(\zeta^{k})^{j+n} y^n 
=- 
\sum_{n\ge 0: j+n\in d\ZZ} y^n 
\]
The lemma follows. 
\end{proof} 

\begin{thm} 
\label{thm:UisOrlov} 
Let $\UU_l\colon H_{\nar}(W,\bmu_d) \to H_{\amb}(X_W)$ denote the 
map in Corollary \ref{cor:Udescends}. 
For a graded matrix factorization $E\in \MFhom^{\rm gr}_{\bmu_d}(W)$ 
such that $\ch(E) \in H_{\nar}(W,\bmu_d)$, we have 
\[
\UU_l \left( \inv^* \ch(E)\right)  
= \inv^* \ch \left( \Phi_l (E) \right).  
\]
\end{thm}
\begin{proof}
Because Chern characters of the form $\ch(\{\ua,\ub\}_q)$ 
in Example \ref{exa:ch-Koszul} span the narrow part, 
it suffices to show that 
\[
\UU_l(\inv^*\ch(\{\ua,\ub\}_q)) =
\inv^*\ch(\Phi_l(\{\ua,\ub\}_q)) 
\]
for $q\in \ZZ$ and $\ua,\ub$ in Example \ref{exa:ch-Koszul}. 
Using \eqref{eq:Unarrow} and \eqref{eq:chernbrane}, 
we get
\begin{align*}
\UU_l(\inv^*\ch(\{\ua,\ub\}_q))
&= \UU_l 
\left(\sum_{k\in \Nar} 
\zeta^{-qk}{(1-\zeta^{w_1k})
\cdots (1-\zeta^{w_{\nn}k})}\phi_{k-1}\right)  \\
&= \frac{1}{d} \sum_{f\in \frF}\sum_{k=0}^{d-1} 
\zeta^{-qk}\frac{(1-\zeta^{w_1k})
\cdots(1-\zeta^{w_{\nn}k})}{\zeta^ke^{p+2\pi\iu \ol f}-1}
\left(\zeta^ke^{p+2\pi\iu \ol f}\right)^l{\unit_f} \\
& = \sum_{f\in \frF} \unit_f 
\sum_{j_1<\dots<j_r} y^l (-1)^r  
 \frac{1}{d} \sum_{k=0}^{d-1} 
\frac{(\zeta^k)^{-q+l+w_{j_1} + \cdots+ w_{j_r}}}
{(\zeta^k)y-1}, 
\end{align*}
where we set $y:=e^{p+2\pi\iu \ol f}$. 
Using Lemma \ref{lem:yformula}, we can write 
the coefficient of $\unit_f$ as 
\begin{equation}
\label{eq:termwithy}
\frac{y^l}{1-y^d}
\sum_{j_1<\dots<j_r} (-1)^{r+1}  
y^{d\fracp{ \frac{q-l}d-\frac{1}d 
\sum_{a=1}^r w_{j_a}}}. 
\end{equation} 
Let $m$ be the remainder of $q-l$ divided by $d$. 
The sum \eqref{eq:termwithy} can be decomposed as  
\begin{align*} 
\frac{y^l}{1-y^d} & 
\left ( 
\sum_{\substack{j_1< \dots < j_r  \\ 
\sum_{a=1}^r w_{j_a} \le m }}
(-1)^{r+1} 
y^{m - \sum_{a=1}^r w_{j_a}} 
+ 
\sum_{\substack{j_1<\dots < j_r \\ 
\sum_{a=1}^r w_{j_a} > m}} 
(-1)^{r+1} 
y^{m - \sum_{a=1}^r w_{j_a} + d} 
\right).  
\end{align*} 
This can be further rewritten as 
\begin{equation}
\label{eq:summands}
\frac{y^l}{1-y^d} 
\left( 
(1-y^d) \sum_{\substack{{j_1<\dots<j_r}\\ 
\sum_{a=1}^r w_{j_a} \le m}}  
(-1)^{r+1} y^{m-\sum_{a=1}^r w_{j_a}} + 
\sum_{\substack{{j_1<\dots<j_r}}} 
(-1)^{r+1}y^{m- \sum_{a=1}^r w_{j_a}+d} 
\right).     
\end{equation} 
The second summand equals 
\[
- \frac{y^{d+l+m}
(1-y^{w_1})\dots(1-y^{w_N})}{1-y^d}. 
\]
This is divisible by $p^{\sharp\{j|w_j f \in \ZZ\}-1}$ 
and vanishes 
in $H(\PP(\uw)_f \cap X_W)$ for the dimensional reason  
(note that $\dim (\PP(\uw)_f \cap X_W) = \sharp\{j| w_j f \in \ZZ\} -2$). 
Finally, the first summand of \eqref{eq:summands} 
equals the coefficient of $\unit_f$ of 
$\inv^*\ch(\Phi_l(\{\ua,\ub\}_q))$ 
by Proposition \ref{pro:OrlovKoszul}. 
\end{proof}

\section{Construction of global $D$-module}
\label{sec:globalDmod}

This section is devoted to the proof of 
the main theorems in \S \ref{subsec:overview} and 
\S \ref{subsec:main-statement}.   
We construct a global $D$-module 
over the base $\cM = \PP(1,d)\setminus \{\text{2 points}\}$  
as an explicit GKZ-type differential system 
and show that the $D$-module is isomorphic to 
the quantum $D$-module of GW theory 
near $v= 0$ 
and to the quantum $D$-module of FJRW theory 
near $v=\infty$. 
We use the mirror theorem in 
\S \ref{sect:computing} and that of Coates-Corti-Lee-Tseng 
\cite{CCLT} (and its refinement in \cite{Iritani:period}).  

\subsection{Multi-GKZ system} 
\label{subsec:mGKZ}
Let $v\mapsto [1,v]$ denote 
the inhomogeneous co-ordinate on $\PP(1,d)$ 
where $v=\infty$ is the $\bmu_d$-orbifold point (LG point). 
Using the co-ordinate $v$, we set  
\[
\cM := \PP(1,d) \setminus \{0,\vc \}, 
\quad 
\cMo := \PP(1,d) \setminus \{0, \vc, \infty\}  
\]
where $\vc := d^{-d} \prod_{j=1}^N w_j^{w_j}$ is the 
conifold point.  
Let $u := v^{-1/d}$ denote the uniformizing co-ordinate 
centered at the LG point. 
In this section we introduce a GKZ-type 
(Gelf'and-Kapranov-Zelevinskii \cite{GKZ:hypergeom}) 
hypergeometric $D$-module over the base $\cMo$. 
The $D$-module here involves the parameter $z$ 
which appears in the quantum $D$-module (see \S \ref{subsubsec:QDM}),  
and the equivariant parameter $\lambda$ which appears in 
the twisted theory (see \S \ref{sect:computing}). 
Therefore it is defined as a sheaf over $\cMo\times \CC_z \times \CC_\lambda$. 
Let $\cR^\tw$ denote the sheaf of algebras over 
$\cMo\times \CC_z\times \CC_\lambda$ given by 
the non-commutative ring of differential operators 
\[
\CC \left\langle 
z, \lambda, v^\pm, (v-\vc)^{-1}, z D_v 
\right \rangle 
\]
where $D_v = v (\partial/\partial v)$. 
We also set 
\[ 
B:= \left\{
(\nu_0,\dots,\nu_N)\in \ZZ^{N+1}\,\big |\;  
\nu_i + q_i \nu_0 \ge 0,\, i= 1,\dots, N 
\right \}, \quad 
q_i = w_i/d. 
\] 
\begin{defn} 
The sheaf $\cF^\tw$ over $\cMo\times \CC_z 
\times \CC_\lambda$ is defined to be the $\cR^\tw$-module 
generated by the symbols $\triangle_{\nu}$ with $\nu\in B$ 
subject to the relations: 
\begin{align} 
\label{eq:relations-cFtw}
\begin{split} 
\left(d z D_v +\lambda + (\nu_0+1) z\right) 
\triangle_{\nu} & = \triangle_{\nu + e_0}, \\ 
\left(w_i z D_v -\nu_i z \right) \triangle_\nu 
& = \triangle_{\nu+ e_i}, \quad i\in \{1,\dots, N\}, \\ 
v \cdot \triangle_{\nu} & = \triangle_{\nu + (-d,w_1,\dots,w_N)}. 
\end{split} 
\end{align} 
Here $\nu\in B$ and 
$e_i = (0,\dots, 0, \overset{i}{1},0,\dots, 0)$, $0\le i\le N$. 
This defines a GKZ-type hypergeometric 
differential system. 
In fact, it is easy to see that  
each generator $\triangle_\nu$ satisfies the relation 
\begin{equation} 
\label{eq:GKZrelation_of_Delta} 
\left[ 
v \prod_{k=1}^d (d z D_v + \lambda + (\nu_0 + k) z) - 
\prod_{i=1}^N \prod_{k=0}^{w_i-1} 
(w_i z D_v - (\nu_i+k) z) 
\right] \triangle_\nu = 0. 
\end{equation} 
\end{defn} 

\begin{rem} 
A multi-generated hypergeometric system similar to $\cF^\tw$ 
above appeared in the recent work of Borisov-Horja 
\cite{Borisov-Horja:BB} 
(also will appear in Coates-Corti-Iritani-Tseng \cite{CCIT:toric}). 
The $\cR^\tw$-submodule $\cR^\tw \Delta_0$ of 
$\cF^\tw$ generated by $\triangle_0$ coincides with 
$\cF^\tw$ at the generic point (Lemma \ref{lem:almostgenerate}),  
but not everywhere (for instance along $z=\lambda=0$). 
A closely related multi-generation phenomena 
of quantum cohomology was observed by 
Guest-Sakai \cite{Guest-Sakai} for a Fano hypersurface in $\PP(\uw)$. 
It was shown in \cite{Iritani:period} that 
the quantum $D$-module of a toric Calabi-Yau hypersurface 
can be described by a multi-GKZ system. 
\end{rem} 

\begin{rem} 
\label{rem:equiv-mirror-int} 
Givental's mirror \cite{Givental:mirrorthm-toric} 
(adapted to a Calabi-Yau hypersurface $X_W$ 
in the weighted projective space $\PP(\uw)$) 
gives a solution to the above differential system. 
Let $\sfx_0,\dots,\sfx_N$ be mirror $\CC^\times$-variables 
subject to the relation 
\[
\sfx_0^{-d} \sfx_1^{w_1} \cdots \sfx_N^{w_N} = v. 
\]
The mirror potential $\mW_\lambda$ is defined by 
\[
\mW_\lambda = \sfx_1 + \cdots + \sfx_N - \sfx_0 + 
\lambda\log \sfx_0. 
\]
Then the integrals 
\[
\cI_\nu(v) = \int \sfx_0^{\nu_0} \sfx_1^{\nu_1} 
\cdots \sfx_N^{\nu_N} 
e^{\mW_\lambda/z} \frac{\dd \sfx_0 \wedge 
\dd \log \sfx_1 \wedge 
\cdots \wedge \dd\log \sfx_N}{\dd \log v}, 
\quad 
\nu\in B 
\]
satisfy the same differential relations 
as $\triangle_\nu$'s do. 
The integration cycle is contained in the torus 
$\{(\sfx_0,\dots,\sfx_N)\in (\CC^\times)^{N+1} 
\,|\, \sfx_0^{-d} \sfx_1\cdots\sfx_N = v\}$ and 
possibly noncompact, 
but here we do not try to justify the integral itself.   
The differential relations among $\cI_\nu(v)$ 
follow from a formal computation of integration by parts. 
\end{rem} 

\begin{lem} 
\label{lem:generation-over-Rtw} 
Set $\nu(l) := (l,-\floor{q_1 l},\dots,-\floor{q_N l}) 
\in B$. 
The sheaf $\cF^\tw$ is generated by 
$\triangle_{\nu(l)}$, $l=0,\dots,d-1$ as an $\cR^\tw$-module.  
\end{lem} 
\begin{proof} 
For $\nu =(\nu_0,\dots,\nu_N)\in B$, set $l = d \fracp{\nu_0/d}$. 
Observe that 
\[
\nu = \nu(l) + 
\sum_{i=1}^N (\nu_i + \floor{q_i \nu_0}) e_i 
-\floor{\frac{\nu_0}{d}} (-d, w_1,\dots,w_N), 
\quad 
\nu_i + \floor{q_i \nu_0} \ge 0.  
\]
The conclusion follows from this and the defining relations 
\eqref{eq:relations-cFtw} of $\cF^\tw$. 
\end{proof} 
%

The sheaf $\cF^\tw$ is a $2\ZZ_{\ge 0}$-graded 
$\cR^\tw$-module with respect to the grading 
\[
\deg v =0, \quad 
\deg z = \deg \lambda = \deg (zD_v) =2, \quad 
\deg \triangle_\nu = 2(\nu_0 + \cdots + \nu_N). 
\]
(Strictly speaking, the module of global sections of $\cF^\tw$ 
is graded, but we abuse the language since we are working over 
the affine base.) 


\begin{lem} 
\label{lem:degDelta} 
Set $\delta(l) := \frac{1}{2} \deg \triangle_{\nu(l)}$. 
We have 

{\rm (i)} $\delta(l+1) \le \delta(l) +1$, $\delta(l+d) = \delta(l)$. 

{\rm (ii)} $ 0\le \delta(l) \le N-1$. 
We have 
$\delta(l)=N-1$ if and only if $l\equiv -1 \mod d$. 
\end{lem} 
\begin{proof} 
We have $\delta(l) = l - \sum_{i=1}^N \floor{q_i l}$. 
Part (i) follows from this formula. 
Part (ii) follows from $\delta(l) =  
\sum_{i=1}^N \fracp{q_i l} \le \sum_{i=1}^N (1-q_i) = N-1$. 
The equality holds iff $l \equiv -1 \mod d/w_i$ for all $i$, 
i.e.\ $l\equiv -1 \mod d$. 
\end{proof}

\begin{lem} 
\label{lem:mult-by-zDv} 
The following relations hold in $\cF^\tw$: 

{\rm (i)} For $0\le l <d-1$, 
$m = \min\{ l\le l'\le d-1 \,|\, \delta(l') = \delta(l)+1\}$ exists and we have 
\[
zD_v \cdot \triangle_{\nu(l)} \in d^{l-m} 
\left(\prod_{i=1}^N w_i^{\floor{q_i m}-\floor{q_i l}} \right) 
\triangle_{\nu(m)} + (z,\lambda) \cF^\tw. 
\]

{\rm (ii)}  $zD_v\cdot \triangle_{\nu(d-1)} \in (z,\lambda) \cF^\tw$. 
\end{lem} 
\begin{proof} 
The existence of $m$ follows from Lemma \ref{lem:degDelta}. 
We have by \eqref{eq:relations-cFtw}
\[
(d zD_v + \lambda+(l+1)z) \triangle_{\nu(l)} 
= \prod_{i=1}^N \prod_{\floor{q_il} < k \le \floor{q_i (l+1)}}
(w_i zD_v + kz) 
\triangle_{\nu(l+1)}.  
\]
Hence 
\[
zD_v \cdot \triangle_{\nu(l)} \in  d^{-1} \left(
\prod_{i=1}^N w_i^{\floor{q_i(l+1)}-\floor{q_il}}\right) 
(zD_v)^{\delta(l)-\delta(l+1)+1} \triangle_{\nu(l+1)} 
+ (z,\lambda) \cF^\tw.  
\]
If $\delta(l+1) \le \delta(l)$, we can apply 
this formula recursively for $zD_v \cdot \triangle_{\nu(l+1)}$ 
in the right-hand side. 
In general, if $\delta(l') \le \delta(l)$ for all 
$l'$ with $l<l'<m'$, we have 
\[
zD_v\cdot \triangle_{\nu(l)} \in d^{-(m'-l)} 
\left(\prod_{i=1}^N 
w_i^{\floor{q_im'} - \floor{q_i l}} \right) 
(zD_v)^{\delta(l) - \delta(m') +1} \triangle_{\nu(m')} 
+ (z,\lambda)\cF^\tw.  
\] 
Taking $m'$ to be $m$, we have (i). 
When $l= d-1$, we can take $m'$ to be $l + d = 2d-1$. 
Then we have 
\[
z D_v \cdot \triangle_{\nu(d-1)} - \vc  zD_v \cdot 
\triangle_{\nu(2d-1)} 
\in (z,\lambda)\cF^\tw. 
\]
Part (ii) follows because $\triangle_{\nu(2d -1)} = v^{-1} \triangle_{\nu(d-1)}$ 
and $(1- \vc/v)$ is invertible. 
\end{proof} 

\begin{thm} 
\label{thm:cFtw-free} 
The sheaf 
$\cF^\tw$ is a free 
$\cO_{\cMo\times \CC_z \times \CC_\lambda}$-module of rank $d$
with the basis $\triangle_{\nu(l)}$, $l=0,\dots,d-1$.  
\end{thm} 

\begin{proof} 
Let $\cF^\tw{}'$ be the 
$\cO_{\cMo\times \CC_z\times \CC_\lambda}$-submodule 
of $\cF^\tw$ 
generated by $\triangle_{\nu(l)}$, $l=0,\dots,d-1$. 
First we see that $\cF^\tw{}' = \cF^\tw$. 
We proceed by induction on the degree. 
The degree zero part $(\cF^\tw)_0$ is 
generated by $\triangle_0 = \triangle_{\nu(0)}$. 
Hence $(\cF^\tw)_0\subset \cF^\tw{}'$. 
Assume by induction that 
$(\cF^\tw)_{\le 2 k}\subset \cF^\tw{}'$ 
for some $k\ge 0$. 
We shall show $(\cF^\tw)_{\le 2 (k+1)} \subset \cF^\tw{}'$. 
By Lemma \ref{lem:generation-over-Rtw}, 
it suffices to show that $zD_v\cdot \triangle_{\nu(l)} \in \cF^\tw{}'$ 
for $0\le l\le d-1$ with $\delta(l) =k$. 
This follows from Lemma \ref{lem:mult-by-zDv} 
and the induction hypothesis. 
Therefore $\cF^\tw = \cF^\tw{}'$. 

As we will see in Proposition \ref{pro:basisofsol-cFtw} below, 
$\cF^\tw$ has $d$ independent solutions. 
This shows that the generic rank (the rank at the generic 
point) of $\cF^\tw$ equals $d$. 
By the previous paragraph, $\cF^\tw$ is generated by 
$\triangle_{\nu(l)}$, $l=0,\dots,d-1$. 
Suppose we have a relation 
$\sum_{l=0}^{d-1} f_l(v,\lambda,z) \triangle_{\nu(l)}=0$ 
with $f_l \in \cO_{\cMo\times \CC_\lambda \times \CC_z}$. 
Then $f_l$ should vanish at the generic point. 
Therefore $f_l =0$. The conclusion follows. 
\end{proof}

\subsection{Refined mirror theorem} 
\label{subsec:refined-mirrorthm}
We construct a basis of hypergeometric solutions of 
the GKZ system $\cF^\tw$. 
Then we relate it to the fundamental solution $L^\tw$ 
of the $e_{\CC^\times}$-twisted quantum connection 
(see \S \ref{subsec:twisted-q-connection}) 
in Theorem \ref{thm:refined-mirrorthm}. 
This shows the analytic continuation of 
\emph{twisted} quantum connections. 
(In this section ``twisted" always means ``$e_{\CC^\times}$-twisted".) 

First we will ``thicken" $\cF^\tw$ by adding 
a new co-ordinate $t^0$. 
Let $\tcMo \to \cMo$ be the minimal abelian cover of 
$\cMo$ such that $\log v$ is single-valued 
(see \eqref{eq:cycliccover-of-M}). 
We set 
\[
\hcM = \CC_{t^0}\times \tcMo. 
\]
where $\CC_{t^0}$ denotes the complex plane with co-ordinate $t^0$. 
Define another co-ordinate 
$s^0\colon \hcM \times \CC_\lambda \to \CC$ by 
\[
s^0 = t^0 - \frac{1}{d}\lambda \log v. 
\]
We shall use $(t^0,v;\lambda)$ and $(s^0,u;\lambda)$ 
as two co-ordinate systems on $\hcM\times \CC_\lambda$; 
$(t^0,v;\lambda)$ is a chart for GW 
theory and $(s^0,u;\lambda)$ is for FJRW theory. 
The co-ordinates $t^0$ and $s^0$ correspond to 
the identity direction of the state space via the mirror map 
we consider below. 
Let $\hcR^\tw$ be the following sheaf of 
algebras over $\hcM\times \CC_z\times \CC_\lambda$: 
\[
\hcR^\tw = 
\cO_{\hcM\times \CC_z \times \CC_\lambda}
\left\langle z D_v, z \parfrac{}{t^0}\right \rangle.  
\]
Here we use the analytic structure sheaf. 
Note that we have 
\begin{equation} 
\label{eq:coordchange-vec}
D_u = -d D_v -\lambda \parfrac{}{t^0}, 
\quad 
\parfrac{}{s^0} = \parfrac{}{t^0}
\end{equation} 
under the co-ordinate change $(t^0,v) \mapsto (s^0,u)$. 
Let $\pr \colon \hcM\times \CC_z \times \CC_\lambda 
\to \tcMo \times \CC_z \times \CC_\lambda 
\to \cMo \times \CC_z \times \CC_\lambda$ 
denote the natural projection. 
The pull back $\hcF^\tw:= \pr^* \cF^\tw$ has the 
structure of an $\hcR^\tw$-module by 
\begin{equation*} 
z \parfrac{}{t^0} \cdot \triangle_{\nu} = \triangle_{\nu}, 
\quad \nu \in B. 
\end{equation*}  
By a \emph{solution} of the $\hcR^\tw$-module 
$\hcF^\tw$,  
we mean an $\hcR^\tw$-module homomorphism 
$\varphi \colon \hcF^\tw|_V \to \cO_{V}$ for an open 
subset $V\subset \hcM \times \CC_z\times \CC_\lambda$. 
We construct a vector-valued solution with values in $\ov{H}$ 
such that all of its components form a basis of solutions. 

\begin{defn} 
The \emph{generalized twisted $I$-functions} $I^{\tw,\nu}$, 
$\nu \in B$  are defined as follows 
(the relevant convergence will be shown in Lemma \ref{lem:convergence} 
below):  

(i) In the FJRW side:
\[
I^{\tw,\nu}_{\FJRW}(s^0, u, z;\lambda) = 
z e^{s^0/z} 
\sum_{k=\nu_0+1}^{\infty} 
u^k 
\frac{\prod_{j=1}^N 
\prod_{0<b<k q_j + \nu_j, \fracp{b} = \fracp{kq_j}}
(-q_j \lambda - b z) } 
{\prod_{0<b<k-\nu_0 , b\in \ZZ} (-bz)} 
\phi_{k-1}
\]
where 
we use the convention of reducing the index $k-1$ 
of $\phi_{k-1}$ modulo $d$. 
This is an $H_{\ext}$-valued power series convergent 
on the region $\{|u| < \vc^{-1/d}\} \times
\CC_z^\times \times  \CC_\lambda$ in $\hcM \times \CC_z \times 
\CC_\lambda$. 
Note that, if $k\ge \nu_0 +1$, then 
$kq_j + \nu_j \ge q_j + (q_j \nu_0+ \nu_j) \ge q_j$. 

(ii) In the GW side (cf.\ \cite[Definition 4.5]{Iritani:period}): 
\[
I^{\tw,\nu}_{\GW}(t^0,v,z;\lambda) 
= z e^{(t^0 + p \log v)/z} 
\sum_{n \in \QQ:  \fracp{n} \in \frF} 
v^n 
\prod_{b=1}^{d n + \nu_0} (d p + \lambda + bz) 
\prod_{i=1}^N 
\frac{\prod_{b\le 0, \fracp{b} = \fracp{w_i n}} 
(w_i p + bz) }
{\prod_{ b\le w_i n - \nu_i, 
\fracp{b} = \fracp{w_in}}(w_i p + bz) } 
\unit_{\fracp{-n}}. 
\]
This is an $H_{\CR}(\PP(\uw))$-valued power series convergent 
on the region 
$\{|v| < \vc\} \times \CC_z^\times \times \CC_\lambda$ in 
$\hcM \times \CC_z \times \CC_\lambda$. 
Note that the term 
\[
\prod_{i=1}^N 
\frac{\prod_{b\le 0, \fracp{b} = \fracp{w_i n}} 
(w_i p + bz) }
{\prod_{ b\le w_i n - \nu_i, 
\fracp{b} = \fracp{w_in}}(w_i p + bz) } 
\unit_{\fracp{-n}} 
\]
vanishes if $w_i n - \nu_i <0$ for all $i$ such that 
$w_i n\in \ZZ$ (because the factor 
$\prod_{\fracp{w_i n} =0} (w_i p)$ in the numerator 
vanishes on $\PP(\uw)_{\fracp{-n}}$ for 
dimensional reason). 
Thus one can assume that there exists 
$i$ such that $w_i n \in \ZZ$ and $w_i n - \nu_i\ge 0$. 
In this case we have $q_i(d n+ \nu_0)  
\ge q_i (d n + \nu_0) + \nu_i - w_i n 
= q_i \nu_0 + \nu_i \ge 0$. 
Hence one can assume $d n + \nu_0 \ge 0$ in the 
summation. 
\end{defn}

\begin{rem} 
For $\nu=0$, $I^{\tw,0}_{\FJRW}(0,u,z)$ and 
$I^{\tw,0}_{\GW}(0,v,z)$ coincide with the original 
twisted $I$-functions in \S \ref{subsubsect:Ifunct}. 
Also note that the generalized $I$-function $I^{\tw, \nu}$ is  
homogeneous of degree $2 + \deg \triangle_{\nu}= 2(1+ \nu_0+\cdots+\nu_N)$ 
with respect to the degree $\deg s^0 = \deg t^0 
= \deg z =\deg \lambda= 2$, $\deg u = \deg v = 0$ and the grading 
on $\ov{H}$. 
\end{rem} 

\begin{lem} 
\label{lem:convergence}
The function $I^{\tw,\nu}_{\FJRW}(s^0,u,z;\lambda)$ is convergent 
on the region $\{|u| < \vc^{-1/d}\} \times
\CC_z^\times \times  \CC_\lambda$ in $\hcM \times \CC_z \times 
\CC_\lambda$; $I^{\tw,\nu}_{\GW}(t^0,v,z;\lambda)$ 
is convergent 
on the region 
$\{|v| < \vc\} \times \CC_z^\times \times \CC_\lambda$ in 
$\hcM \times \CC_z \times \CC_\lambda$. 
\end{lem} 
\begin{proof} 
Write $I^{\tw, \nu}_{\FJRW}(s^0,u,z;\lambda) = z e^{s^0/z} 
\sum_{k=\nu_0+1}^{\infty} u^k \Box_k(z,\lambda)$ with 
$\Box_k(z,\lambda)$ an element of $H_{\ext}$. 
Fix a norm $\|\cdot\|$ on $H_{\ext}$ such that $\|\phi_{k}\|=1$. 
Then we have 
\[
\frac{\|\Box_{k+d}(z,\lambda)\|}{\|\Box_k(z,\lambda)\|} 
= \frac{\prod_{j=1}^N \prod_{a=0}^{w_j-1} 
|q_j \lambda + (kq_j + \nu_j +a)z|}
{\prod_{a=0}^{d-1}|(k-\nu_0 +a)z|}.  
\]
This converges to $\prod_{j=1}^N q_j^{w_j} = \vc$ 
as $k\to \infty$ 
for a fixed $(z,\lambda)\in \CC^\times \times \CC$. 
This implies that the convergence radius of $I^{\tw,\nu}_{\FJRW}$ 
as a power series of $u$ is $\vc^{-1/d}$.  
Similarly, write $I^{\tw,\nu}_{\GW}(t^0,v,z;\lambda) = 
z e^{(t^0+ p \log v)/z} \sum_{n\in \QQ : \fracp{n} \in \frF}
\Diamond_n(z,\lambda)$ with $\Diamond_n(z,\lambda)$ 
an element of $H_{\CR}(\PP(\uw))$.  
Fix a norm $\| \cdot \|$ on $H_{\CR}(\PP(\uw))$. 
We have 
\[
\frac{\|\Diamond_{n+1}(z,\lambda)\|}{\|\Diamond_n(z,\lambda)\|}
\le \left \| 
\frac{
\prod_{a=1}^d (dp + \lambda + (d n + \nu_0+a) z) }
{
\prod_{i=1}^N 
\prod_{a=1}^{w_i} ( w_i p + (w_i n - \nu_i + a)z)
} 
\right \| 
\]
where in the right-hand side $\|\cdot\|$ means the operator norm. 
The right-hand side converges to $d^d/\prod_{i=1}^N w_i^{w_i} = \vc^{-1}$ 
as $n\to \infty$ for a fixed $(z,\lambda) \in \CC^\times \times \CC$. 
Hence the convergence radius of 
$I^{\tw,\nu}_{\GW}(t^0,v,z)$ as a power series of 
$v$ is no less than $\vc$. 
\end{proof}

\begin{pro} 
\label{pro:basisofsol-cFtw} 
For each $\varphi \in \Hom(\ov{H},\CC)$, 
the map 
\[
I^\varphi \colon \hcF^\tw 
\longrightarrow \cO, \quad 
\triangle_\nu \longmapsto z^{-1} \varphi(I^{\tw,\nu}), \quad 
\nu \in B 
\]
defines a solution to the $\hcR^\tw$-module 
$\hcF^\tw$, i.e.\ a homomorphism of 
$\hcR^\tw$-modules. 
Moreover, for a $\CC$-basis $\varphi_1,\dots,\varphi_d$ of 
$\Hom(\ov{H},\CC)$, the corresponding solutions 
$I^{\varphi_1},\dots,I^{\varphi_{d}}$ are linearly independent. 
(In fact, they form a basis of solutions by 
Theorem \ref{thm:cFtw-free}.)
\end{pro} 
\begin{proof} 
For the former statement, it suffices to check that $I^{\tw,\nu} 
= I^{\tw,\nu}_{\FJRW}$ or $I^{\tw,\nu}_{\GW}$ satisfies 
the following differential equations 
(cf.\ \eqref{eq:relations-cFtw}; note also the co-ordinate 
change \eqref{eq:coordchange-vec}): 
\begin{align}
\label{eq:diffeq-generalized-I} 
\begin{split} 
\left(-z D_u + (\nu_0+1) z\right) 
I^{\tw,\nu} & = I^{\tw, \nu + e_0}, \\ 
\left(-q_i z D_u - q_i \lambda -\nu_i z \right)
I^{\tw,\nu} 
& = I^{\tw,\nu+ e_i}, \quad i=1,\dots, N, \\ 
v \cdot I^{\tw,\nu} = I^{\tw,\nu + (-d,w_1,\dots,w_N)},  
&\quad  
z \parfrac{}{s^0} I^{\tw,\nu} = I^{\tw,\nu}. 
\end{split} 
\end{align} 
They follow from a straightforward computation. 
Let $\nu(l)$ be as in Lemma \ref{lem:generation-over-Rtw}. 
For the FJRW $I$-functions, we have 
\begin{equation} 
\label{eq:FJRW-Itw-asympt} 
z^{-1} I^{\tw,\nu(l)}_{\FJRW} \sim 
e^{s^0/z} 
u^{l+1} \left( \phi_l + O(u)\right), \quad 
l=0,\dots, d-1. 
\end{equation} 
Since the leading terms span $H_{\ext}$, 
it follows that the solutions 
$I^{\varphi_{1}}_{\FJRW},\dots,I^{\varphi_{d}}_{\FJRW}$ 
are linearly independent. 
For the GW $I$-functions, we have 
if $\fracp{l/d} \in \frF$, 
\[
z^{-1} I^{\tw, \nu(l)}_{\GW} \sim 
e^{(t^0 + p \log v)/z} 
v^{-l/d} \left ( \unit_{\fracp{\frac{l}{d}}}  
+ O(v^{1/d}) \right).     
\]
Thus 
\begin{equation}
\label{eq:GW-Itw-asympt} 
\left(z D_v + \frac{l}{d}z\right)^i z^{-1} I^{\tw,\nu(l)}_{\GW} 
\sim e^{(t^0 + p \log v)/z} 
v^{-l/d} \left (p^i \unit_{\fracp{\frac{l}{d}}}  
+ O(v^{1/d}) \right). 
\end{equation} 
These leading terms span $H_{\CR}(\PP(\uw))$. 
Hence $I^{\varphi_{1}}_{\GW},\dots,I^{\varphi_{d}}_{\GW}$ are 
linearly independent. 
\end{proof} 

The twisted $I$-function $I^{\tw,0}$ in each theory 
has the $z^{-1}$-asymptotics of the form (cf.\ 
\eqref{eq:FJRW-I-asympt}): 
\begin{equation} 
\label{eq:z-asympt-Itw}
I^{\tw,0} = z F \cdot T_0 + G + O(z^{-1}) 
\end{equation} 
where $F$ and $G$ are functions on a domain 
in $\hcM \times \CC_{\lambda}$ on which $I^{\tw,0}$ converges; 
$F$ takes values in $\CC$ and $G$ takes values in 
the degree $\le 2$ part 
$\ov{H}^{\le 2}=\ov{H}^{0}\oplus \ov{H}^2$. 
More precisely, in the FJRW side, we have
(cf.~\eqref{eq:FandG}): 
\begin{align*}  
F & = F_{\FJRW}(u)  =  
\sum_{k \ge 1 : k \equiv 1 \, (d)} 
u^k 
\frac{
\prod_{j=1}^N 
(kq_j-1)_{\ceil{kq_j}-1}} 
{(k-1)!} \\ 
G & = G_{\FJRW}(s^0,u;\lambda) = s^0 F_{\FJRW}(u) \phi_0 - 
\sum_{k\ge 2: \sum_{j=1}^N \fracp{q_j(k-1)} = 1} 
u^k 
\frac{
\prod_{j=1}^N (kq_j-1)_{\ceil{k q_j}-1}} 
{(k-1)!} 
\phi_{k-1} \\
& \qquad \qquad \qquad \qquad+ 
\lambda\sum_{k\ge d+1: k\equiv 1 \, (d)} 
u^k \left( \sum_{i=1}^N 
\sum_{0<a<k q_i, \fracp{a} = \fracp{kq_i}}
\frac{q_i}{a}\right) 
\frac{
 \prod_{j=1}^N (kq_j-1)_{\ceil{kq_j}-1}} 
{(k-1)!} 
\phi_0  
\end{align*} 
and in the GW side, we have: 
\begin{align*} 
F &= F_{\GW}(v)  =  \sum_{n=0}^\infty  
v^n \frac{(dn)!}{\prod_{j=1}^N (w_j n)! }  
\\ 
G &= G_{\GW}(t^0,v;\lambda) = t^0 F_{\GW}(v) \unit + 
\sum_{n\in \QQ_{>0}: \fracp{n} \in\frF, 
\sum_{j=1}^N \fracp{-w_j n}=1} 
v^n \frac{(dn)!}{\prod_{j=1}^N (w_jn)_{\ceil{w_jn}} } \unit_{\fracp{-n}} 
\\ 
& \qquad \qquad \qquad \quad  + \left [   F_{\GW}(v) \log v + 
\sum_{n=1}^\infty v^n 
\left( 
\sum_{a=1}^{dn} \frac{d}{a} - \sum_{i=1}^N \sum_{a=1}^{w_in} 
\frac{w_i}{a} \right) 
\frac{ (dn)!} {\prod_{j=1}^N (w_j n)!} 
\right ] p \unit \\
& \qquad \qquad \qquad \quad 
+ \lambda \sum_{n=1}^{\infty}  
v^n \left(\sum_{a=1}^{dn} \frac{1}{a} \right) 
\frac{(dn)!}{\prod_{j=1}^N (w_j n)!} \unit 
\end{align*} 
where $(a)_n = a(a-1) \cdots (a-n+1) = \Gamma(a+1)/\Gamma(a-n+1)$ 
denotes the falling factorial. 
We define the mirror map $\varsigma$ to be the 
$\ov{H}^{\le 2}$-valued function: 
\begin{equation} 
\label{eq:mirrormap} 
\varsigma = \frac{G}{F}. 
\end{equation} 
The FJRW mirror map $\varsigma_{\sFJRW}$ 
is defined over 
$\{|u|< \vc^{-1/d}\}\times \CC_\lambda$ 
and the GW mirror map $\varsigma_{\sGW}$ 
is defined over 
$\{|v|< \vc\} \times \CC_\lambda$. 
The following mirror theorem gives a refinement of 
Theorem \ref{thm:FJRW-mirrorthm} 
and \cite{CCLT,CCIT:computing}, 
namely, the special case $\nu=0$ corresponds to 
the original mirror theorem 
(see \eqref{eq:Upsilon0} in the proof). 
A similar refinement was given in \cite[Theorem 4.6]{Iritani:period} 
for GW theory of complete intersections 
in toric orbifolds. 

\begin{thm} 
\label{thm:refined-mirrorthm} 
In both FJRW and GW theories, there exist 
$\ov{H}$-valued complex analytic functions 
$\Upsilon^{\tw,\nu}$, $\nu\in B$ defined 
on an open set $\hU \times \CC_z\times \CC_\lambda
\subset \hcM \times \CC_z \times \CC_\lambda$ 
such that 
\begin{equation} 
\label{eq:Itw-Upsilon}
L^{\rm tw}(\varsigma(x;\lambda),z;\lambda) I^{\tw,\nu}(x, z;\lambda) 
= z \Upsilon^{\tw,\nu}(x,z;\lambda).
\end{equation} 
Here $\varsigma$ denotes the mirror map \eqref{eq:mirrormap} 
in each theory. 
The fundamental solution $L(\varsigma(x;\lambda),z;\lambda)$ is also 
analytic over $\hU\times \CC_z^\times \times \CC_\lambda$. 
The open subset $\hU\subset \hcM$ is of the form $\{|u|<\epsilon\}$ 
for FJRW theory and is of the form $\{|v|<\epsilon\}$ for GW theory.  
For $\nu=0$, we have $\Upsilon^{\tw,0} = F\cdot T_0$ 
where $F$ is the function appearing in \eqref{eq:z-asympt-Itw}. 
\end{thm} 

\begin{proof} 
First we discuss the case of FJRW theory. 
By Theorem \ref{thm:FJRW-mirrorthm}, we have 
\begin{equation} 
\label{eq:Jtw-Itw} 
F(u) J^{\tw}
(\varsigma(s^0, u;\lambda),z;\lambda ) = I^{\tw, 0}(s^0, u, z;\lambda) 
\quad \text{for } \ 
s^0 = 0 
\end{equation} 
where the subscripts ``FJRW" are omitted. 
We have $
\varsigma(s^0,u;\lambda) = s^0 T_0 + 
\varsigma(0,u;\lambda)$ and 
$I^{\tw,0}(s^0, z;\lambda) = e^{s^0/z} I^{\tw,0}(0,z;\lambda)$. 
By the string equation for the twisted invariants 
(see \S \ref{subsubsec:twistedtheory}), 
we have 
\[
J^{\tw}(s^0 T_0 + \varsigma(0, u;\lambda), z;\lambda) 
= e^{s^0/z} J^{\tw}(\varsigma(0, u;\lambda),z;\lambda). 
\]
Hence \eqref{eq:Jtw-Itw} holds for arbitrary $s^0$. 
Therefore, by Proposition \ref{pro:Ltw-Jtw}, we have 
\begin{equation} 
\label{eq:Upsilon0}
L^\tw(\varsigma(s^0,u;\lambda),z;\lambda) I^{\tw,0}(s^0,u,z;\lambda) 
= z F(u) T_0. 
\end{equation} 
This shows that one can take 
$\Upsilon^{\tw,0}(s^0,u,z;\lambda) = F(u) T_0$. 
The other $\Upsilon^{\tw,\nu}$'s are obtained from 
this by differentiation. 
To see this, we use the fact that the generalized 
$I$-functions satisfy \eqref{eq:diffeq-generalized-I} 
and that we have by Proposition \ref{pro:twisted-fundsol} 
\begin{equation*} 
(\varsigma^*\nabla^{\tw}_{z D_u}) \circ  
L^\tw(\varsigma(s^0,u;\lambda),z;\lambda) = 
L^\tw(\varsigma(s^0,u;\lambda),z;\lambda) \circ z D_u, 
\end{equation*} 
where $\varsigma^* \nabla^\tw_{zD_u} = zD_u + 
(D_u \varsigma(s^0,u;\lambda)) \bullet^\tw$. 
For example, one obtains $z \Upsilon^{\tw,e_i}$ as 
\begin{align*} 
L^\tw(\varsigma(s^0,u;\lambda),z;\lambda) I^{\tw,e_i}(s^0,u,z;\lambda) 
& = L^\tw (\varsigma(t^0,u;\lambda),z;\lambda)  (-q_i z D_u - q_i \lambda) \cdot 
I^{\tw,0}(s^0,u,z;\lambda) \\
& = (-q_i \varsigma^* \nabla^{\tw}_{z D_u}-q_i\lambda) \left( 
z F(u) T_0 \right).  
\end{align*} 
To obtain $\Upsilon^{\tw, \nu}$ for a general $\nu$, 
we use the following differential 
operator: 
\[
P_\nu(zD_u) = v^{k} 
\prod_{b=1}^{\nu_0 + kd} (- zD_u + bz) 
\cdot 
\prod_{i=1}^N 
\frac{
\prod_{b=-\infty}^{\nu_i - k w_i -1}(-q_i z D_u - q_i \lambda - b z) } 
{\prod_{b=-\infty}^{-1} (-q_i z D_u - q_i \lambda - b z) } 
\]
where $k$ is an integer such that $\nu_0 + kd \ge 0$. 
When $\nu_i - k w_i < 0$ for some $i$, we expand 
the factor $(-q_i z D_u -q_i \lambda - bz )^{-1}$ 
in the $\lambda^{-1}$-series 
\[
\sum_{n=0}^\infty (-q_i\lambda)^{-n-1} (bz + q_i zD_u)^n.  
\]
Then we have $P_\nu(z D_u) I^{\tw,0} = I^{\tw,\nu}$. 
By applying $P_\nu(\varsigma^* \nabla_{zD_u})$ 
to \eqref{eq:Upsilon0}, one obtains \eqref{eq:Itw-Upsilon}
with $\Upsilon^{\tw,\nu}(s^0,u;\lambda) = 
P_\nu(\varsigma^*\nabla_{z D_u}) F(u) T_0$. 
Note that this expression makes sense as an element of 
$\ov{H}\otimes \CC[z](\!(\lambda^{-1})\!)[\![s^0,u]\!]$. 
This is because $\varsigma^*\nabla_{zD_u} =z D_u + 
(D_u \varsigma)\bullet^\tw = z D_u + O(u)$ 
(note that $\varsigma(s^0,u) = s^0 \phi_0 - u \phi_1 + O(u^2)$). 
A posteriori, we know that 
$\Upsilon^{\tw,\nu}$ belongs to 
$\ov{H}\otimes \CC[z,\lambda][\![s^0,u]\!]$ 
by \eqref{eq:Itw-Upsilon} since 
$L^\tw$, $I^\tw$ and $\varsigma$ are regular at $\lambda=0$. 

Next we show the analyticity of $\Upsilon^{\tw,\nu}$ 
and $L^{\tw}$. It suffices to show the analyticity of $L^\tw$. 
By \eqref{eq:Itw-Upsilon} we have 
\begin{equation} 
\label{eq:Birkhoff} 
z^{-1} 
\begin{bmatrix}
\vert & & \vert \\  
u^{-1} I^{\tw,\nu(0)}& \dots & u^{-d}I^{\tw,\nu(d-1)} \\ 
\vert & & \vert  
\end{bmatrix} 
= 
(L^\tw)^{-1}
\begin{bmatrix} 
\vert &  & \vert \\ 
u^{-1} \Upsilon^{\tw,\nu(0)} & \dots & u^{-d} \Upsilon^{\tw,\nu(d-1)} \\
\vert &  & \vert 
\end{bmatrix} 
\end{equation} 
where $\nu(l)$ is as in Lemma \ref{lem:generation-over-Rtw}. 
Using the basis $\phi_0,\dots,\phi_{d-1}$ of $\ov{H}$, one can 
view this as an equality of $(d,d)$ matrices. 
Write $\gamma(z)$ for the left-hand side. 
It is invertible near $u=0$ 
because of the asymptotics \eqref{eq:FJRW-Itw-asympt}. 
Hence $S^1 \ni z\mapsto \gamma(z)$ 
defines an element of the loop group $LGL_d=C^\infty(S^1,GL_d)$. 
As observed by Coates-Givental \cite{Coates-Givental} and 
Guest \cite{Guest:D-module}, we can regard \eqref{eq:Birkhoff} 
as a Birkhoff factorization  
of $\gamma(z)$ because $(L^\tw)^{-1} = \id + O(z^{-1})$ and 
$\Upsilon^{\tw,\nu}$ is regular at $z=0$. 
Birkhoff's theorem \cite{Pressley-Segal} says that 
the multiplication map 
$L_1^-GL_d \times L^+ GL_d \to LGL_d$ is an isomorphism onto 
an open and dense subset called the ``big cell". 
Here $L_1^- GL_d$ is the subgroup consisting of the boundary values of 
holomorphic maps $\gamma_- \colon 
\{z\in \CC\cup \{\infty\} \mid |z|< 1\} \to GL_d$ 
satisfying $\gamma_-(\infty) =\id$ and 
$L^+GL_d$ is the subgroup consisting of the boundary values of 
holomorphic maps $\gamma_+ \colon 
\{z\in \CC \mid |z|>1\} \to GL_d$. 
The asymptotics \eqref{eq:FJRW-Itw-asympt} 
ensures that $\gamma(z)$ is in the big cell for 
$|u|<\epsilon$ and $|\lambda|\le 1$ for sufficiently 
small $\epsilon>0$. 
Hence its Birkhoff factor $\gamma_-(z) =
L^{\tw}(\varsigma(s^0,u;\lambda),z;\lambda)^{-1}$ 
is analytic on the region $\{|u|<\epsilon, |z|> 1, |\lambda|<1\}$.  
The homogeneity of $(L^\tw)^{-1}$ implies that 
it is in fact analytic on $\{|u|<\epsilon, z\in \CC^\times, 
\lambda\in \CC\}$.

For GW theory, the theorem 
follows from the proof of \cite[Theorem 4.6]{Iritani:period}. 
When $\fracp{l/d} \in \frF$, 
the function $I^{\tw,\nu(l)}_{\GW}$ coincides with 
$z v^{-l/d} {\mathbf I}_\lambda^{\fracp{l/d}}|_{Q=1}$ there 
and we can take $\Upsilon^{\tw,\nu(l)} = v^{l/d} 
\widetilde{\mathbf \Upsilon}_{\fracp{l/d}}|_{Q=1}$ 
in the notation of \emph{loc}.\ \emph{cit}. 
We can get the other $I^{\tw,\nu}_{\GW}$ 
from $I^{\tw,\nu(l)}_{\GW}$, $\fracp{l/d}\in \frF$ 
by applying differential operators in $\hcR^\tw$, 
so the other $\Upsilon^{\tw,\nu}$ as well. 
\end{proof} 

Let $\hU_{\FJRW}= \{|u| < \epsilon \}$, 
$\hU_{\GW}=\{|v|< \epsilon \}$ be sufficiently small 
open subsets of $\hcM$ 
as in Theorem \ref{thm:refined-mirrorthm}. 
The following corollary gives a twisted version 
of Theorem \ref{thm:main}.

\begin{cor} 
\label{cor:anconti-twisted} 
Via the $\hcR^\tw$-module $\hcF^\tw$ over $\hcM$, 
the $e_{\CC^\times}$-twisted quantum connections $\nabla^\tw$ 
of FJRW theory and of GW theory 
are analytically continued to each other. 
More precisely, we have a local trivialization of $\hcF^\tw$ 
over $\hU_\heartsuit$ 
\begin{align*} 
\Mir_\heartsuit \colon 
\hcF^\tw|_{\hU_{\heartsuit} \times \CC_z\times \CC_\lambda} 
& \cong \ov{H}_\heartsuit 
\otimes \cO_{\hU_{\heartsuit} \times \CC_z\times \CC_\lambda}, 
\quad 
\heartsuit = \FJRW \text{ or } \GW \\ 
\triangle_\nu & \longmapsto \Upsilon^{\tw,\nu}_\heartsuit  
\end{align*}  
such that, under the trivialization, 
the action of $\hcR^{\tw}$ is 
given by the $e_{\CC^\times}$-twisted quantum connection 
\[
z D \longmapsto \varsigma_\heartsuit^* \nabla^{\tw,\heartsuit}_{zD}, \quad 
\text{$D$ is a vector field on $\hU_{\heartsuit}$,}
\]
where $\varsigma_\heartsuit \colon \hU_\heartsuit\times \CC_\lambda
\to \ov{H}^{\le 2}_\heartsuit$ 
is the mirror map \eqref{eq:mirrormap} in the respective theory. 
\end{cor} 
\begin{proof} 
We omit the subscript ``FJRW" or ``GW" throughout the proof. 
By Proposition \ref{pro:basisofsol-cFtw}, 
the generalized twisted $I$-functions 
define an $\ov{H}$-valued solution: 
\begin{equation} 
\label{eq:I-solution} 
\hcF^\tw|_{\hU\times \CC_z^\times \times \CC_\lambda}
\to \ov{H}\times \cO_{\hU\times \CC_z^\times \times \CC_\lambda}, \quad 
\triangle_\nu \mapsto z^{-1} I^{\tw,\nu}.  
\end{equation} 
which is an isomorphism (see the asymptotics 
\eqref{eq:FJRW-Itw-asympt}, \eqref{eq:GW-Itw-asympt}). 
On the other hand, the twisted quantum connection 
$\varsigma^* \nabla^\tw$ 
also has an $\ov{H}$-valued solution 
(Proposition \ref{pro:twisted-fundsol})
\[
L^\tw(\varsigma(\cdot\,;\lambda), z;\lambda)^{-1} \colon 
(\ov{H}\otimes \cO_{\hU\times \CC_z^\times \times \CC_\lambda},
\varsigma^*\nabla^\tw) \to 
\ov{H}\otimes \cO_{\hU\times \CC_z^\times \times \CC_\lambda}, 
\]
which sends $\Upsilon^{\tw, \nu}$ to $z^{-1} I^{\tw,\nu}$ 
by Theorem \ref{thm:refined-mirrorthm}. This is also an isomorphism. 
Therefore we have an isomorphism $\Mir \colon 
\hcF^\tw|_{\hU\times \CC_z^\times \times \CC_\lambda}\cong 
\ov{H}\otimes \cO_{\hU\times \CC_z^\times \times \CC_\lambda}$ 
such that $\Mir(\triangle_\nu) = \Upsilon^{\tw,\nu}$.  
It extends across $z=0$ as $\Upsilon^{\tw,\nu}$ is regular 
at $z=0$. 
Now it suffices to show that 
$\Upsilon^{\tw, \nu}$, $\nu\in B$ 
generate $\ov{H}$ along $z=0$. 
In the case of FJRW theory, this follows 
from the fact that the factor 
$[u^{-1}\Upsilon^{\tw,\nu(0)},\dots,u^{-d}\Upsilon^{\tw,\nu(d-1)}]$ 
in the Birkhoff factorization \eqref{eq:Birkhoff} 
is invertible at $z=0$.   
The discussion is similar for GW theory. 
\end{proof} 
\begin{rem} 
Mann-Mignon \cite[Theorem 1.2]{Mann-Mignon} described explicitly the 
twisted quantum $D$-module (with $\lambda=0$) for a 
smooth nef complete intersection in a toric manifold. 
\end{rem} 

\subsection{Analytic continuation $\UU^\tw$ revisited} 
\label{subsec:ancont-revisited}
\begin{lem}
\label{lem:almostgenerate} 
The submodule $\cR^\tw \triangle_0$ of $\cF^\tw$ coincides 
with $\cF^\tw$ at the generic point on $\cMo \times \CC_z\times 
\CC_\lambda$. 
\end{lem} 
\begin{proof} 
In view of the isomorphism \eqref{eq:I-solution}, 
it suffices to show that 
$\prod_{b=1}^l(zD_u-bz) z^{-1} I^{\tw,0}_{\FJRW}$, 
$l=0,\dots,d-1$ form a basis of $\ov{H}=H_{\ext}$ 
for generic $(z,\lambda)$ and sufficiently small $|u|$. 
This follows from the asymptotics near $u=0$: 
\[
\prod_{b=1}^l(zD_u-bz) z^{-1} I^{\tw,0}_{\FJRW} 
\sim e^{s^0/z} u^{l+1} (-1)^l 
\left( 
\prod_{j=1}^N 
\prod_{\substack{0<b<(l+1)q_j \\ \fracp{b} = \fracp{(l+1)q_j}}}
(-q_j \lambda - b z)  \right)
\phi_l + O(u^{l+2}). 
\]
\end{proof} 

We calculated in Proposition \ref{pro:Utw} 
a linear transformation $\UU^\tw_l$ \eqref{eq:U}
\[
\UU^{\tw}_l\colon 
H_{\ext}\otimes \cO_{\Delta_\xi}  
\to H_{\CR}(\PP(\uw))\otimes \cO_{\Delta_\xi}
\]
from analytic continuation of the ``$\Hf$-functions", 
where $\Delta_\xi=\{|\xi|< \varepsilon\}$ 
denotes a sufficiently small disc in the 
$\xi:= (\lambda/z)$-plane. 
We give an interpretation of $\UU^{\tw}_l$ as 
analytic continuation of flat sections 
of the global $D$-module $\hcF^\tw$. 
Using the trivialization $\Mir_\heartsuit$ 
in Corollary \ref{cor:anconti-twisted}, 
we define a flat section $\frf^\tw_\heartsuit(\alpha)$ of 
$\hcF^\tw$ over $\hU_\heartsuit 
\times \{(z,\lambda)\in \CC^\times \times \CC\;|\; 
|\xi| = |\lambda/z| < \varepsilon\}$ 
parametrized by $\alpha\in \ov{H}_\heartsuit \otimes \cO(\Delta_\xi)$: 
\begin{equation} 
\label{eq:flatsec-tw}
\frf^\tw_\heartsuit(\alpha)(x,z;\lambda) 
:= L^\tw_{\heartsuit}(\varsigma_\heartsuit(x;\lambda),z;\lambda) z^{-\grading} 
\hGamma^\tw_\heartsuit \left((2\pi\iu)^{\frac{\deg_0}{2}} \alpha \right), 
\quad \alpha \in \ov{H}_\heartsuit \otimes \cO(\Delta_\xi)   
\end{equation} 
where $L^\tw_\heartsuit$, $\hGamma^\tw_\heartsuit$ are the 
fundamental solution and the twisted Gamma class 
(\S \ref{subsubsec:Hfunct}) in each theory 
and $\heartsuit = \GW \text{ or } \FJRW$. 
We extend the $\Hf$-functions in the $s^0$- or $t^0$-direction as follows: 
\begin{align*} 
\Hf^\tw_{\FJRW}((s^0,u),z;\lambda) 
&= e^{s^0/z} \Hf^{\tw}_{\FJRW}(u,z;\lambda)  \\
\Hf^\tw_{\GW}((t^0,v),z;\lambda) 
&= e^{t^0/z} \Hf^{\tw}_{\GW}(v,z;\lambda) 
\end{align*} 
so that we have $I^\tw_\heartsuit = z^{-\grading} \hGamma_\heartsuit^\tw 
\left ((2\pi\iu)^{\frac{\deg_0}{2}} \Hf^\tw_\heartsuit\right)$ 
(cf.\ \eqref{eq:Ifunct-Hfunct}). 
By this relation and Theorem \ref{thm:refined-mirrorthm}, 
we have  
\begin{equation} 
\label{eq:Hfcn-meaning}
\Mir_\heartsuit(z \triangle_0) =  
z\Upsilon^{\tw,0}_{\heartsuit}(x,z;\lambda) 
=  \frf^\tw_{\heartsuit}
(\Hf^\tw_{\heartsuit}(x,z;\lambda))(x,z;\lambda). 
\end{equation} 
Namely \emph{the $\Hf$-function represents the section 
$z \triangle^0$ in the flat frame $\frf^{\tw}_\heartsuit$}. 
The relationships between $\Upsilon^{\tw,0}_\heartsuit$, 
$z \Delta_0$, $I^\tw_{\heartsuit}$, $\Hf^\tw_{\heartsuit}$ 
are given 
in the following diagram: 
\begin{align*} 
\xymatrix@R=1pc{
\cF^\tw \Bigr|_{\cB_\heartsuit} \ar[r]^-{\Mir_\heartsuit} & 
(\ov{H}_{\heartsuit}\otimes \cO_{\cB_\heartsuit},\varsigma_{\heartsuit}^*\nabla^\tw) 
& \ar[l]_-{\varsigma_\heartsuit^* L^\tw_{\heartsuit}}
 (\ov{H}_{\heartsuit}\otimes \cO_{\cB_\heartsuit},\dd)  
& & & 
\ar@/^2pc/[llll]_{\frf^\tw_\heartsuit} 
\ar[lll]_-{z^{-\grading}\hGamma^\tw_\heartsuit
(2\pi\iu)^{\frac{\deg_0}{2}}}   
(\ov{H}_\heartsuit \otimes \cO_{\cB_\heartsuit},\dd)  \\ 
z \Delta_0 \ar@{|->}[r] &  z \Upsilon^{\tw,0}_\heartsuit &  
\ar@{|->}[l] I^\tw_{\heartsuit} & & & 
\ar@{|->}[lll]\Hf^\tw_\heartsuit 
}
\end{align*} 
where $\cB_\heartsuit = \hU_\heartsuit \times \{(z,\lambda) 
\in \CC^\times \times \CC \mid |\lambda/z| <\epsilon\}$ 
and $\dd$ stands for the trivial connection 
(all the connections here are only defined in 
the $\hU_\heartsuit$-direction). 

Recall the path $\gamma_l$ in $\tcMo$
(\S \ref{subsubsec:MB}, Figure \ref{fig:ancont-path}) 
defined for each integer $l$.  
It can be lifted to a path $\hgamma_l$ in $\hcM$ starting 
from the GW base point $\log v \ll 0$, $t^0=0$ 
and ending at the FJRW base point $\log u\ll 0$, $s^0=0$.  
The homotopy type of the lift $\hgamma_l$ is unambiguous. 
For convenience, we take the following lift $\hgamma_l$: 
\begin{align*} 
(\log v\ll 0, t^0 =0) 
& \rightsquigarrow 
( \log u\ll 0, t^0 =0) 
= (\log u\ll 0, s^0 =\lambda \log u)  
& &\text{(moving along $\gamma_l$)} 
\\
& \rightsquigarrow  
(\log u \ll 0, s^0 =0) 
& &\text{(shifting $s_0$).} 
\end{align*} 
Because the shift of $s^0$ has an effect of 
the multiplication by the factor 
$u^{\lambda/z} = u^{\xi}$ 
on the $\Hf$-function $\Hf_{\FJRW}$, 
we have from Proposition \ref{pro:Utw} that 
\begin{equation} 
\label{eq:H-and-U}
(\Hf_{\GW}^\tw)_{\rm continued} = \UU_l^\tw (\Hf_{\FJRW}^\tw) 
\end{equation} 
where the left-hand side now 
denotes the analytic continuation 
of $\Hf^\tw_{\GW}$ along $\hgamma_l$.  
\begin{pro} 
\label{pro:Utw-flatsection} 
Along the path $\hgamma_l^{-1}$, 
the FJRW flat section 
$\frf^\tw_{\sFJRW}(\alpha)$, $\alpha \in H_{\ext}$ 
is analytically continued to the GW flat section 
$\frf^\tw_{\sGW}(\UU^\tw_l \alpha)$.  
Here the values of $z$ and $\lambda$ are fixed during 
the analytic continuation and chosen so that 
$|\xi| = |\lambda/z|$ is sufficiently small. 
\end{pro} 
\begin{proof} 
Note that $\hGamma^\tw_{\heartsuit}$ is invertible for 
sufficiently small $|\xi|$.  
Therefore $\{\frf^\tw_\heartsuit(T_i)\}_{i=0}^{d-1}$ forms  
a basis of flat sections for sufficiently small $\xi=\lambda/z$.  
Hence for a fixed such $(z,\lambda)$, 
there exists an invertible linear transformation 
$\VV_l \colon H_{\ext} \to H_{\CR}(\PP(\uw))$ 
such that $\frf^\tw_{\sFJRW}(T_i)$ 
is analytically continued to $\frf^\tw_{\sGW}(\VV_l T_i)$ 
along $\hgamma_l^{-1}$. 
Because $\frf^\tw_\heartsuit(\Hf^\tw_\heartsuit) 
= \Mir_\heartsuit(z \triangle_0)$ \eqref{eq:Hfcn-meaning} 
and $z \triangle_0$ is a global section of $\hcF^\tw$, we have 
\[
(\Hf^\tw_{\GW})_{\rm continued} = \VV_l (\Hf^\tw_{\FJRW}). 
\]
Because $z\triangle_0$ is a generator of $\hcF^\tw$ 
at the generic point (Lemma \ref{lem:almostgenerate}), 
this relation uniquely determines $\VV_l$ 
for a generic $(z,\lambda)$. 
By \eqref{eq:H-and-U}, we know that $\VV_l = \UU_l^\tw$. 
\end{proof} 

\subsection{The non-equivariant limit and its reduction}
\label{subsec:reduction}
Here we prove Theorem \ref{thm:main}. 
By taking the non-equivariant limit $\lambda =0$ 
in Corollary \ref{cor:anconti-twisted}, we obtain 
analytic continuation between $e$-twisted quantum connections. 
(Recall that $e$ stands for the non-equivariant 
Euler class.)
We shall show 
that it reduces to analytic continuation 
between ambient and narrow part quantum $D$-modules. 
This reduction was described more explicitly in terms 
of the Picard-Fuchs ideal 
in a recent paper of Mann-Mignon \cite[Theorem 1.2]{Mann-Mignon} 
for the quantum cohomology of a smooth nef complete intersection 
in a toric manifold. 

\vspace{5pt} 
\noindent 
{\bf (Step 0)} 
Note that $\tcMo\times \CC_z$ is contained in 
$\hcM\times \CC_z \times \CC_\lambda$ 
as the locus $\{\lambda=t^0=0\} = \{\lambda = s^0=0\}$. 
We consider the restriction 
\[
\tcG := \hcF^\tw|_{\lambda=t^0=0}
\]
of $\hcF^\tw$ to $\tcMo\times \CC_z$. 
This is also identified with the pull-back of 
\[
\cG: = \cF^\tw|_{\lambda=0}
\] by 
$\tcMo\times \CC_z \to \cMo\times\CC_z$. 
Let $\tU_\heartsuit$ denote the open subset 
of $\tcMo$ given by 
$\tU_{\GW} = \{|v|<\epsilon \}$ or 
$\tU_{\FJRW} = \{|u|<\epsilon\}$ 
where $\epsilon$ is the same as in 
Corollary \ref{cor:anconti-twisted}. 
Over $\tU_\heartsuit \times \CC_z$, 
$\tcG$ is identified with the 
$e$-twisted quantum connection $\nabla^\tw$ on 
$\ov{H}_\heartsuit \times (\tU_\heartsuit \times \CC_z) 
\to \tU_\heartsuit \times \CC_z$ 
by Corollary \ref{cor:anconti-twisted}. 
By Proposition \ref{pro:noneq-reduction}, 
under the natural projection $\pr \colon \ov{H}
\twoheadrightarrow H'$, the $e$-twisted quantum 
connection projects to the quantum connection 
of the respective theory: 
\begin{align} 
\label{eq:mirrorisom-lambdazero} 
\begin{CD} 
\tcG|_{\tU_\heartsuit \times \CC_z} @>{\cong}>>  
\left( \ov{H}_\heartsuit \otimes \cO_{\tU_\heartsuit \times \CC_z}, 
\varsigma_\heartsuit^* \nabla^\tw
\right)  \\ 
@. @VV{\pr}V \\ 
@. \left( 
H'_\heartsuit \otimes \cO_{\tU_\heartsuit \times \CC_z},  
(\pr \circ \varsigma_\heartsuit)^* \nabla\right).   
\end{CD} 
\end{align} 
where $\varsigma_\heartsuit\colon \tU_\heartsuit \to \ov{H}^2_\heartsuit$ 
denotes the mirror map 
\eqref{eq:mirrormap} restricted to $\lambda = t^0 =0$. 
Here the meromorphic flat connection $\nabla$ 
on $\tcG$ (or $\cG$) is given by 
the action of $z D_v\in \cR^\tw|_{\lambda=0}$, i.e.\ 
we define $\nabla_{D_v} := z^{-1} (\text{the action of $zD_v$})$ 
on $\tcG$ (or $\cG$).  

\vspace{5pt} 
\noindent 
{\bf (Step 1)} 
Let $U_\heartsuit\subset \cMo$ be the image of $\tU_\heartsuit$ 
under the projection $\tcMo \to \cMo$. 
We show that the diagram \eqref{eq:mirrorisom-lambdazero} descends to 
the quotient $\tU_\heartsuit \twoheadrightarrow U_\heartsuit$. 
First notice that the Galois symmetry in Propositions 
\ref{pro:Galois-FJRW}, \ref{pro:Galois-GW} extends to 
the twisted theory. The map $G\colon H' \to H'$ there 
is extended to $\ov{H}$ as 
\begin{alignat*}{2}  
G(\phi_k) & = e^{-2\pi\iu k/d} \phi_k  
& \quad & \text{for FJRW theory;} \\ 
G(\unit_f) & = e^{2\pi\iu f} \unit_f - 2\pi\iu p 
& & \text{for GW theory}. 
\end{alignat*} 
Then the conclusions of Propositions \ref{pro:Galois-FJRW}, 
\ref{pro:Galois-GW} hold for this $G$ 
(except that we do not have the connection in the $z$-direction
in the twisted theory). The proof is similar. 
This shows that the fundamental solution $L^\tw$ 
in the twisted theory (see Proposition \ref{pro:twisted-fundsol}) 
has the following symmetry: 
\begin{align*} 
e^{-2\pi\iu/d} G \circ L^\tw_{\FJRW}(G^{-1}(t),z;\lambda) 
& = L^\tw_{\FJRW}(t,z;\lambda) \circ e^{-2\pi\iu/d} G  \\
\dd G \circ L^\tw_{\GW}(G^{-1}(t),z;\lambda) & 
= L^\tw_{\GW}(t,z;\lambda) \circ e^{-2\pi\iu p/z} \dd G.   
\end{align*} 
On the other hand, the deck transformation of 
$\tU_\heartsuit \twoheadrightarrow U_\heartsuit$ acts 
on $I^{\tw,\nu}$ as 
\begin{align*} 
e^{-2\pi\iu/d} G \left( 
I^{\tw, \nu}_{\FJRW}(\log u + (2\pi\iu/d), z)\Bigr|_{s^0=\lambda =0} 
\right)  
& = I^{\tw,\nu}_{\FJRW}(\log u, z)\Bigr|_{s^0= \lambda=0} \\ 
e^{-2\pi\iu p/z} \dd G\left( 
I^{\tw,\nu}_{\GW}(\log v + 2\pi\iu, z)\Bigr|_{t^0=\lambda=0}\right) 
& = I^{\tw,\nu}_{\GW} (\log v, z)\Bigr|_{t^0=\lambda=0}  
\end{align*} 
Hence the mirror maps (with $t^0=\lambda=0$) satisfy 
\begin{equation} 
\label{eq:mirrormap-Galois}
G\left(\varsigma_{\sFJRW}(\log u + (2\pi\iu/d))\right)
= \varsigma_{\sFJRW}(\log u), 
\quad 
G\left(\varsigma_{\sGW}(\log v + 2\pi\iu)\right) = 
\varsigma_{\sGW}(\log v).  
\end{equation} 
This shows that the deck transformation of $\tU_\heartsuit$ 
is conjugate to the Galois action on $\ov{H}^2$ 
via the mirror maps. 
By the relation \eqref{eq:Itw-Upsilon} 
and the above calculations, we find that 
(again over the locus $\lambda=t^0=0$) 
\begin{align*}
\begin{split}  
e^{-2\pi\iu/d} G \left( 
\Upsilon_{\FJRW}^{\tw,\nu}(\log u + (2\pi\iu)/d, z) 
\right) 
& = \Upsilon_{\FJRW}^{\tw,\nu}(\log u,z) \\
\dd G \left ( 
\Upsilon_{\GW}^{\tw,\nu} ( \log v + 2\pi\iu ,z) 
\right) 
& = \Upsilon_{\GW}^{\tw,\nu}(\log v,z)  
\end{split} 
\end{align*} 
This shows that the induced Galois symmetry on 
the sheaf $(\ov{H} \times \cO_{\tU_\heartsuit \times \CC_z}, 
\varsigma_\heartsuit^* \nabla^\tw)$ is compatible 
with the deck transformation on $\tcG|_{\tU_\heartsuit\times \CC_z}$ 
because the deck-transformation-invariant section 
$\triangle_\nu\in \tcG$ corresponds to $\Upsilon^{\tw,\nu}$.  
Moreover the projection $\pr \colon \ov{H} \to H'$ is 
compatible with the Galois action, so the diagram 
\eqref{eq:mirrorisom-lambdazero} descends to 
\begin{equation} 
\label{eq:mirrorisom-lambdazero-Galoisquot} 
\begin{CD} 
\cG|_{U_\heartsuit \times \CC_z} @>{\cong}>>  
( \ov{H}_\heartsuit \times \cO_{\tU_\heartsuit \times \CC_z}, 
\varsigma_\heartsuit^*\nabla^\tw)/\langle G  \rangle  \\ 
@. @VV{\pr}V \\ 
@. (H'_\heartsuit \otimes \cO_{\tU_\heartsuit \times \CC_z},  
(\pr \circ \varsigma_\heartsuit)^* \nabla)/
\langle G \rangle. 
\end{CD} 
\end{equation} 
Notice that the bundle in the second line is the 
pull-back of the quantum $D$-module 
$(F, \nabla)/\langle G\rangle$ in Definition \ref{defn:QDM} 
by the mirror map 
\[ 
\tau_\heartsuit := 
[\pr \circ \varsigma_\heartsuit]  
\colon U_\heartsuit \to {H^{\prime \, 2}_\heartsuit}/\langle G\rangle. 
\]
In the diagram \eqref{eq:mirrorisom-lambdazero-Galoisquot}, 
we do not consider the flat connection $\nabla_z$ 
in the $z$-direction and the pairing $P$. 
However, we can introduce $\nabla_z$ for $\cG$ and make 
the diagram compatible with $\nabla_z$ as follows. 
Recall that (the module of global sections of) $\cF^\tw$ is $2\ZZ$-graded by 
$\deg u = \deg v =0$, $\deg \triangle_\nu = 2 \sum_{i=0}^N \nu_i$, 
$\deg z = \deg \lambda =2$. 
Thus $\cG=\cF^\tw|_{\lambda=0}$ is also graded. 
The grading defines the meromorphic flat connection $\nabla_z$ on $\cG$ 
(with logarithmic poles along $z=0$) as 
\[
\nabla_z \triangle_\nu = \frac{1}{z} 
\frac{\deg \triangle_\nu}{2} 
\triangle_\nu. 
\] 
Because all the morphisms in the diagram 
\eqref{eq:mirrorisom-lambdazero-Galoisquot} 
preserve the grading and the Euler vector field 
vanishes on the image of the mirror map 
$\pr \circ \varsigma_\heartsuit$, 
the projection $\pr \colon \cG|_{U_\heartsuit \times \CC_z} 
\to (\tau_\heartsuit)^* 
(F, \nabla)/\langle G \rangle$ 
induced from the diagram \eqref{eq:mirrorisom-lambdazero-Galoisquot}
preserves the connection $\nabla_z$ as well. 

\vspace{5pt} 
\noindent 
{\bf (Step 2)} 
The diagram \eqref{eq:mirrorisom-lambdazero-Galoisquot}  
defines for each $(x,z)\in U_\heartsuit\times \CC_z$ 
a projection $\cG_{(x,z)} \twoheadrightarrow H'$, i.e.\ 
an element of the Grassmannian $\Grass(\cG_{(x,z)})$. 
The kernel of the projection is flat for $\nabla$  
(including the $z$-direction). 
We show that this section of the Grassmannian bundle 
$\Grass(\cG)$ extends globally over $\cMo\times \CC_z$. 

Recall the flat section $\frf^\tw_\heartsuit(\alpha)$ 
of the twisted theory in \eqref{eq:flatsec-tw}. 
When restricted to the locus $\lambda=t^0=0$, 
this defines a flat section of $\cG$. 
On the other hand, we can define a flat 
section of the quantum $D$-module 
$(H'_\heartsuit \otimes 
\cO_{\tU_\heartsuit \times \CC_z}, 
(\pr \circ \varsigma_\heartsuit)^* \nabla)$ 
by an analogous formula: 
\begin{equation} 
\label{eq:f-section} 
\frf_\heartsuit(\alpha) = 
L_\heartsuit(\pr \circ \varsigma_\heartsuit(x),z) z^{-\grading} 
\hGamma_\heartsuit 
\left( 
(2\pi\iu)^{\frac{\deg_0}{2}} \alpha\right), 
\quad \alpha \in H'_\heartsuit,  
\end{equation} 
where $L_\heartsuit(t,z)$ and $\hGamma_\heartsuit$ 
are the fundamental solution 
and the Gamma class in the respective theory 
(as appear in Definition \ref{defn:intstr}). 
By Proposition \ref{pro:Ltw-L} and the definitions of 
$\frf^\tw_\heartsuit$ and $\frf_\heartsuit$, we have 
\begin{equation} 
\label{eq:iota-frf} 
\pr\left( \frf^\tw_\heartsuit(\alpha)|_{\lambda=t^0=0} \right) 
= \frf_\heartsuit(\pr(\alpha)) 
\end{equation} 
for $\alpha \in \ov{H}$. 

\begin{lem} 
The section of $\Grass(\cG)$ over 
$(U_{\GW} \cup U_{\FJRW} )\times \CC_z$ given by 
the diagram \eqref{eq:mirrorisom-lambdazero-Galoisquot}
extends to $((U_{\GW}\cup U_{\FJRW})\times \CC_z) \cup 
(\cMo\times \CC_z^\times)$. 
\end{lem} 
\begin{proof} 
By the flat connection $\nabla$ on $\cG$, 
the section of $\Grass(\cG)$ over $U_{\GW}\times \CC_z^\times$ 
can be extended along any path in $\cMo\times \CC_z^\times$.  
We see that the given section of 
$\Grass(\cG)|_{U_{\GW}\times \CC_z^\times}$ 
is analytically continued to the given section of 
$\Grass(\cG)|_{U_{\FJRW}\times \CC_z^\times}$ 
along the path $\gamma_l$ in \S \ref{subsubsec:MB}. 
By considering the $\lambda=0$ limit 
in Proposition \ref{pro:Utw-flatsection}, we know that 
$\frf^\tw_{\sFJRW}(\alpha)|_{\lambda=t^0=0}$ 
is analytically continued to 
$\frf_{\sGW}^\tw(\UU_l^\tw \alpha)|_{\lambda=t^0=0}$ 
along $\gamma_l^{-1}$, for $\alpha \in \ov{H}_\FJRW = H_\ext$. 
By \eqref{eq:iota-frf}, the projections of these flat 
sections by $\pr$ are $\frf_{\sFJRW}(\pr(\alpha))$ 
and $\frf_{\sGW}(\pr(  \lim_{\lambda \to 0}\UU_l^\tw \alpha) )$. 
Diagrammatically: 
\begin{equation} 
\label{eq:ancont-frf}
\begin{CD}
\frf^\tw_{\sFJRW}(\alpha)\bigr |_{\lambda=t^0=0} 
@>{\text{analytic continuation}}>{\text{along $\gamma_l^{-1}$}}> 
\frf_{\sGW}^\tw(\UU_l^\tw \alpha)\bigr |_{\lambda=t^0=0} \\ 
@V{\pr}VV @V{\pr}VV \\
\frf_{\sFJRW}(\pr(\alpha))  @. \frf_{\sGW}(\UU_l \pr( \alpha) )  
\end{CD} 
\end{equation} 
Here we used the fact (Corollary \ref{cor:Udescends}) 
that there exists a unique operator 
$\UU_l \colon H'_{\FJRW} \to H'_{\GW}$ such that 
$\pr \circ (\lim_{\lambda \to 0} \UU^\tw_l) = \UU_l \circ \pr$. 
The existence of such an operator shows that 
the sections $\pr$ of $\Grass(\cG)|_{U_{\GW}\times \CC_z^\times}$ and 
$\Grass(\cG)|_{U_{\FJRW}\times \CC_z^\times}$ coincide under analytic 
continuation along $\gamma_l$. 
Because this holds for all the paths $\gamma_l$ with $l \in \ZZ$, 
the conclusion follows. 
\end{proof} 

\begin{lem} 
The section of $\Grass(\cG)$ in the previous lemma  
extends to $\cMo\times \CC_z$. 
\end{lem} 
\begin{proof} 
The section of $\Grass(\cG)$ 
here is flat for $\nabla$ on $\cG$. 
Therefore, the corresponding element of 
$\Grass(\cG_{(x,z)})$ at $(x,z)\in \cMo\times \CC^\times_z$ 
can be represented by a matrix independent of $z$ 
when we write it in terms of the homogeneous 
basis $z^{-\deg \triangle_{\nu(l)}/2}\triangle_{\nu(l)}$, 
$l=0,\dots,d-1$ of $\cG_{(x,z)}$. 
Therefore, via the basis $\triangle_{\nu(l)}$, 
$\nu=0,\dots,d-1$, the section $\{x\}\times \CC^\times_z \to 
\Grass(\cG|_{\{x\} \times \CC_z^\times})$ 
can be represented by an algebraic map 
$\CC^\times_z \to \Grass(\CC^d)$,  
which extends across $z=0$ by the completeness of 
$\Grass(\CC^d)$. 
This proves the lemma. 
\end{proof} 

\vspace{5pt} 
\noindent 
{\bf (Step 3)} The previous step shows that there exists a 
projection $\cG \twoheadrightarrow \cF$ 
to a locally free sheaf $\cF$ over $\cMo\times \CC_z$. 
The sheaf $\cF$ is equipped with a meromorphic 
flat connection with simple poles along $z=0$. 
\[
\nabla \colon \cF \to \cF(\cMo\times\{0\}) 
\otimes \Omega_{\cMo\times \CC_z}^1.   
\]
Also $\cF$ is isomorphic to the pulled-back quantum $D$-module 
$(H'_{\heartsuit} \otimes \cO_{\tU_{\heartsuit} \times \CC_z}, 
(\pr \circ \varsigma_\heartsuit)^* \nabla)/
\langle G\rangle$ over the open subset $U_\heartsuit \times \CC_z$. 
In particular, 
$\cF$ extends across the orbifold point $u=0$ 
as an orbi-sheaf with flat connection 
(i.e.\ $\bmu_d$-equivariant flat bundle on a $d$-fold cover). 
We denote this extension over $\cM\times \CC_z$ 
by the same symbol $\cF$. 

We claim that there is a global $\ZZ$-local 
subsystem $F_\ZZ$ of $(\cF|_{\cM\times\CC_z^\times},\nabla)$ 
such that it coincides with the $\hGamma$-integral 
structure over $U_{\heartsuit}\times \CC_z^\times$. 
By \eqref{eq:ancont-frf}, the flat section 
$\frf_{\sFJRW}(\alpha)$, $\alpha\in H_{\nar}(W,\bmu_d)$ 
is analytically continued to 
$\frf_{\sGW}(\UU_l \alpha)$ along the path $\gamma_l^{-1}$. 
Note that $\frf(\alpha)$ \eqref{eq:f-section} is related 
to the flat section $\frs(\cE)$ \eqref{eq:K-framing} 
defining the $\hGamma$-integral structure by 
\[
\frs(\cE) = 
\frac{1}{(2\pi\iu)^{\hc}} 
\frf(\inv^* \ch(\cE)) (x,z) 
\]
where $\cE$ is an object of $D^b(X_W)$ 
or $\MFhom^{\rm gr}_{\bmu_d}(W)$ 
such that $\ch(\cE) \in H'$. 
Therefore by Theorem \ref{thm:UisOrlov} we know that 
\begin{equation} 
\label{eq:ancont-frs} 
\frs_{\sFJRW}(\cE) \text{ is analytically continued to } 
\frs_{\sGW}(\Phi_l(\cE)) 
\text{ along }\gamma_l^{-1} 
\end{equation} 
for $\cE\in \MFhom^{\rm gr}_{\bmu_d}(W)$ with 
$\ch(\cE) \in H_{\nar}(W,\bmu_d)$. 
This shows the existence of a global $\ZZ$-local system 
and that the analytic continuation along $\gamma_l^{-1}$ 
corresponds to the Orlov equivalence $\Phi_l$.  

Finally we show that $\cF$ admits a global $\nabla$-flat pairing 
\[
P \colon (-)^*\cF \otimes \cF \to z^{\hc} \cO_{\cM\times \CC_z}, 
\quad \hc = N-2, 
\]
which coincides with the pairings $P_{\GW}$, $(-1)^{N-1}P_{\FJRW}$
of the quantum $D$-modules. 
In order to see that the global pairing exists over 
$\cM\times \CC^\times$, in view of \eqref{eq:ancont-frs}, 
it suffices to check that 
\begin{equation} 
\label{eq:pairingmatch-flatsec} 
(-1)^{N-1} P_{\FJRW}((-)^*\frs_{\sFJRW}(\cE_1), \frs_{\sFJRW}(\cE_2)) = 
P_{\GW}((-)^*\frs_{\sGW}(\Phi_l \cE_1), \frs_{\sGW}(\Phi_l \cE_2))  
\end{equation}
for $\cE_1,\cE_2\in \MFhom^{\rm gr}_{\bmu_d}(W,\bmu_d)$ 
such that $\ch(\cE_i) \in H_{\nar}(W,\bmu_d)$. 
Recall that the pairing between the flat sections $\frs(\cE)$
coincides with the Euler form up to sign 
(Proposition \ref{pro:Gamma-intstr-properties}). 
Because the categorical equivalence preserves the 
Euler pairing $\chi(\cE,\cF) = \chi(\Phi_l \cE, \Phi_l \cF)$, 
\eqref{eq:pairingmatch-flatsec} follows. 
The global pairing $P$ over $\cM\times \CC_z^\times$ 
extends across $z=0$ (with zeros of order $\hc$) 
by Hartog's principle 
because it already extends over $U_\heartsuit \times \CC_z$. 
The non-degeneracy of $P/(2\pi\iu z)^{\hc}$ along $z=0$ 
holds for the same reason. 

Now the proof of Theorem \ref{thm:main} is complete. 

\begin{rem} 
We described the global $D$-module $\cF$  
as a quotient of $\cG = \cF^\tw|_{\lambda=0}$. 
In \cite[Theorem 6.13]{Iritani:period}, 
with the aid of mirror symmetry, 
it was described as a \emph{submodule} of another multi-GKZ system. 
We can translate this result in our setting as follows. 
Define the shift map $S \colon \cMo\times \CC_z \times 
\CC_\lambda \to \cMo\times \CC_z\times \CC_\lambda$ 
by 
$S(x,z,\lambda) = (x,z,\lambda-z)$. 
Then the map 
\[
\sigma \colon \cF^\tw \to S^* \cF^\tw, \quad 
\triangle_\nu \longmapsto S^* \triangle_{\nu+e_0}
\]
is a morphism of $\cR^\tw$-modules 
by the relations \eqref{eq:relations-cFtw}. 
This is an isomorphism at the generic point because 
both $\cR^\tw \triangle_0$ and $\cR^\tw \triangle_{e_0}$ 
equal $\cF^\tw$ at the generic point  
(see Lemma \ref{lem:almostgenerate}; the proof there  
applies also to $\cR^\tw \triangle_{e_0}$). 
However, $\sigma$ is not an isomorphism over $\lambda=0$ 
and we have $\cF = \Image (\sigma|_{\lambda=0} 
\colon \cF^\tw|_{\lambda=0} \to (S^*\cF^\tw)|_{\lambda=0} )$. 
See also \cite{Mann-Mignon}. 
\end{rem} 

\subsection{Reconstruction of the big quantum $D$-module} 
\label{subsec:reconstruction}  
Here we prove Theorem \ref{thm:big}. 
When $X_W$ is a manifold, the orbifold cohomology 
consists only of untwisted sectors. In particular 
$H_{\amb}(X_W)$ is spanned 
$\unit, p \unit, \cdots, p^{\dim X_W} \unit$. 
This allows us to use the reconstruction theorem 
\cite{Kontsevich-Manin, Hertling-Manin, Iritani:genmir, 
Rose, Reichelt} to obtain the big quantum cohomology 
from the small one. 

More specifically, we apply the reconstruction theorem of 
a (TE) structure by Hertling-Manin \cite[Theorem 2.5]{Hertling-Manin} 
to the global $D$-module $(\cF,\nabla)$ over $\cM$ 
(which is itself a (TE) structure). 
For this, one has to check the injectivity condition (IC) and 
the generation condition (GC) for $(\cF,\nabla)$. 
More concretely, (IC) means  
\[
z D_v \triangle_0\Bigl |_{z =0} \neq 0,  
\]
and (GC) means 
\[
\{(z D_v)^n \triangle_0 \, |\, n \ge 0\}  
\text{ generates } \cF|_{z=0} \text{ over } \cO_\cM. 
\]

We claim that $(z D_v)^n \triangle_0$, 
$n=0,\dots, \rank \cF -1$  
is a basis of $\cF|_{z=0}$ over the open subsets 
$U_{\FJRW}$ and $U_{\GW}$. 
(Here $U_{\FJRW}$ does not contain $u=0$.)
We work over the cyclic cover $\tU_{\heartsuit}\subset \tcMo$ 
of $U_{\heartsuit}$.  
First observe that we have $D$-module 
isomorphisms  
(cf.\ \eqref{eq:I-solution}): 
\begin{align*} 
(\cF,\nabla)|_{\tU_\heartsuit\times \CC^\times_z} 
& \xrightarrow{\Mir_\heartsuit} & 
& \hspace{-5pt} 
(H'_\heartsuit \otimes \cO_{\tU_\heartsuit \times \CC_z^\times}, 
(\pr \circ \varsigma_\heartsuit)^* \nabla) & 
& \hspace{-10pt} 
\xrightarrow{L(\pr\circ \varsigma(x),z)^{-1}} &  
& \hspace{-5pt} 
(H'_{\heartsuit} \otimes \cO_{\tU_\heartsuit \times \CC_z^\times},d) \\ 
\triangle_\nu 
& \longmapsto &  
& \hspace{-5pt} 
\Upsilon^\nu := \pr( \Upsilon^{\tw,\nu}) & 
& \hspace{-10pt} \longmapsto & 
& \hspace{-30pt} 
z^{-1} I^\nu_{\heartsuit} := 
z^{-1} \pr( 
I^{\tw,\nu}_{\heartsuit}|_{\lambda=t^0=0}).  
\end{align*} 
Here the first map is induced from the mirror 
isomorphism in Corollary \ref{cor:anconti-twisted} 
(see also \eqref{eq:mirrorisom-lambdazero}) 
and the second map is given by the inverse of 
the fundamental solution $L(t,z)$ in each theory. 
The relation 
$L(\pr\circ \varsigma(x),z)^{-1} \Upsilon^\nu(x,z) 
= z^{-1} I^\nu(x,z)$ follows from 
\eqref{eq:Itw-Upsilon} and Proposition \ref{pro:Ltw-L}. 
Similarly to \eqref{eq:Birkhoff}, 
the two maps $\Mir_\heartsuit$, $L^{-1}$ can be viewed as 
the Birkhoff factors of the composition 
since $\Mir_\heartsuit$ extends regularly to $z=0$ 
and $L^{-1}$ extends regularly to $z=\infty$. 
We want to check that $(z D_v)^i \triangle_0$, 
$i=0,\dots,\rank \cF-1$ form a basis. 
Under the above map, these sections map to 
\begin{align*} 
(zD_v)^i z^{-1} I^0_{\GW} &= e^{p\log v/z} 
\left( p^i \unit + O(v^{1/d}) \right)
\end{align*} 
over $\tU_{\GW}$. From these asymptotics, 
we know that the matrix with the column vectors 
$(zD_v)^i z^{-1} I^0_{\GW}$, $i=0,\dots,\rank \cF-1$ 
is Birkhoff factorizable (i.e.\ in the ``big cell") 
for sufficiently small $|v|$; this means that $(zD_v)^i \triangle_0$, 
$i=0,\dots,\rank \cF-1$ is a basis of $\cF|_{z=0}$ 
over $\tU_{\GW}$. 

Over $\tU_{\FJRW}$, the calculation is a little more involved. 
Instead of $(zD_v)^i$, $i=0,\dots,\rank \cF-1$, 
we consider the differential operator 
$P_i$, $i=0,\dots,\rank \cF-1$ defined inductively by 
\[
P_0 = u^{-1}, \quad P_i := u^{-\ord_i} (z \partial_u)P_{i-1} 
\]
where $\ord_i\in \NN$ is determined by 
$(z\partial_u) P_{i-1} I_{\FJRW}^0 = O(u^{\ord_i})$. 
It suffices to show that $P_i \triangle_0$, $i=0,\dots,\rank \cF-1$ 
is a basis of $\cF|_{z=0}$ since 
$\{P_i \triangle_0\}$ and $\{(zD_v)^i \triangle_0\}$ 
are related by an invertible matrix along $z=0$. 
We have   
\[
P_i z^{-1} I^0_{\FJRW} = c_{i} \frac{\phi_{k_i-1}}{z^{l_i-i}} 
+ O(u) 
\]
where $k_i$ is the $(i+1)$-th smallest element 
of the set $\Nar\subset \{1,\dots,d-1\}$, 
$c_i \neq 0$ 
and $l_i := \deg (\phi_{k_i-1})/2$. 
It is not difficult to show that $l_i = i$ 
when $X_W$ is a manifold. 
Therefore the matrix having the column vectors 
$P_i z^{-1} I_{\FJRW}^0$, $i=0,\dots,\rank \cF-1$ 
is Birkhoff factorizable for small $|u|$. 
The claim now follows 
also over $\tU_{\FJRW}$.  

Because (IC) and (GC) are open conditions, 
they hold in a Zariski open subset $\cM'$ of $\cM$ 
containing $U_\GW$ and $U_{\FJRW}$. 
At each point $x\in \cM'$,  
we have a universal unfolding \cite[Definition 2.3]{Hertling-Manin} 
of $(\cF,\nabla)|_{(\cM,x)\times \CC_z}$ 
over the analytic germ 
$(\cM,x) \times (\CC^{\rank \cF-1}, 0)\times \CC_z$.  
By the universality, they will patch together to 
form a global (TE) structure $(\cF^\ext, \nabla^\ext)$ 
over $\cM_\ext \supset \cM'$. 
By \cite[Lemma 3.2]{Hertling-Manin}, the pairing 
$P$ over $\cM\times \CC_z$ extends to $\cM_\ext \times \CC_z$ 
and we have a (TEP) structure $(\cF^\ext, \nabla^\ext,P^\ext)$. 
The extension of the $\ZZ$-local system $F_\ZZ$ is automatic.  

Next we show that $(\cF^\ext,\nabla^\ext,P^\ext)$ coincides 
with the ``big" quantum $D$-module over a neighbourhood 
of $U_{\FJRW}$ or $U_{\GW}$. 
We review the reconstruction of the big FJRW quantum cohomology. 
Over $U_{\FJRW}$, we already identified 
$(\cF,\nabla,P)$ with the quantum $D$-module over 
the image of the mirror map $\tau = \pr \circ \varsigma$. 
We take a basis $\{T_i\}_{i=0}^r$ of $H'$ such that 
$T_0 = \phi_0$, $T_1=\phi_1$ and 
write the big quantum product as 
$T_i \bullet T_j = \sum_{k=0}^r C_{ij}^k(t) T_k$, 
where $t=(t^0,\dots,t^r)$ is the co-ordinates of 
$H'=H_{\nar}(W,\bmu_d)$ 
dual to $\{T_i\}_{i=0}^r$. 
Using the frame $\{T_i\}_{i=0}^r$, one can write 
the connection $\nabla$ of $\cF|_{U_{\FJRW}}$ as 
\[
\nabla_u = \parfrac{}{u} + \frac{1}{z} \sum_{i=0}^r \parfrac{\tau^i(u)}{u} 
\left(C_{i\alpha}^\beta(\tau(u))\right)_{\alpha,\beta}.  
\]
Here $\tau(u) = \sum_{i=0}^r \tau^i(u) T_i$ denotes the mirror map. 
The structure constants $C_{ij}^k(t)$ are a priori formal power 
series in $t$, but we know from the mirror theorem that 
the above connection $\nabla_u$ is convergent. 
Because $\tau(u) = -u\phi_1 + O(u^2)$, 
we can use $(u,t^0,t^2,\dots,t^r) 
\mapsto \tau(u) + \sum_{j\neq 1} t^j T_j$ as 
a co-ordinate patch of $H'$ near the origin. 
We want to reconstruct the connection operators 
\[
\nabla_u^\ext = 
\parfrac{}{u} + \frac{1}{z} A(u,t),\quad 
\nabla_i^\ext = \parfrac{}{t^i} + \frac{1}{z} 
\left( C_{i\alpha}^\beta
\left(\tau(u) + \textstyle 
\sum_{j\neq 1} t^j T_j\right)\right)_{\alpha,\beta}, \quad i\neq 1  
\]
satisfying $\nabla_u^\ext|_{t=0} = \nabla_u$, 
$[\nabla^\ext_i,\nabla^\ext_u] 
= [\nabla^\ext_i,\nabla^\ext_j] =0$  
and $\nabla_i T_0 = T_i$. 
Following the method of \cite[Lemma 2.9]{Hertling-Manin}, 
\cite[\S 4.4]{Iritani:genmir}, one can solve for such 
$C_{i\alpha}^\beta(\tau(u)+ \sum_{j\neq 1} t^j T_j)$ 
uniquely as a power series in $t$. 
This is because $T_0 = \phi_0$ 
is asymptotic to $u^{-1}\triangle_0$ as $u\to 0$,  
so is also a cyclic vector of the action of $[z \nabla_u]|_{z=0}$ 
for a sufficiently small $u\neq 0$. 
This reconstruction can be done either over 
the formal Laurent series ring $\CC(\!(u)\!)$ or 
for a fixed small $u\neq 0$. 
In the former case, we recover the big quantum product 
as a formal power series in $(u,t)$; 
in the latter case, we get 
$C_{i\alpha}^\beta(\tau(u)+ \sum_{j\neq 1} t^j T_j)$ 
as a convergent power series of $t$ 
(\cite[Lemma 2.9]{Hertling-Manin}). 
Therefore $C_{i\alpha}^\beta(\tau(u)+\sum_{j\neq 1}t^j T_j)$ 
is a formal power series in $t$ whose 
coefficients are analytic functions on $\{u\in \CC\,|\,|u|<\epsilon\}$. 
Moreover for each $u$ with $0<|u|<\epsilon$, it is convergent 
as a power series in $t$.  
By \cite[Lemma 6.5]{Iritani:convergence}, such a function
is holomorphic in a neighbourhood of $(u,t)=(0,0)$. 
This shows the convergence of the big quantum product 
and that $(\cF^\ext,\nabla^\ext)$ is isomorphic to  
the big quantum $D$-module in a neighbourhood of $U_{\FJRW}$. 
The discussion on the GW side 
is similar and omitted. 

\subsection{Monodromy and autoequivalences} 
\label{subsec:monodromy}
Here we prove Theorem \ref{thm:monodromyrep}. 
We study the relationships between monodromy of 
the global quantum $D$-module $\cF$ 
and category equivalences.  

An object $E$ of $D^b(X_W)$ is said to be \emph{spherical} 
\cite[Definition 1.1]{Seidel-Thomas} if 
$\Hom^n(E,E) = \Hom(E,E[n])$ is isomorphic to the cohomology of 
a sphere, i.e.\ 
\[
\Hom^n(E, E) = \begin{cases} 
\CC & n = 0 \text{ or } \dim X_W \\ 
0 & \text{otherwise}.  
\end{cases} 
\]
Seidel-Thomas \cite{Seidel-Thomas} introduced a functor  
$T_E\colon D^b(X_W)\to D^b(X_W)$, called \emph{spherical twist}, 
for a spherical object $E$. 
This gives an auto-equivalence with the following property:  
\[
T_E(F) \cong \Cone( \Hom^\bullet(E,F)\otimes E \to F).  
\]
\begin{exa} 
A line bundle $\cO(i)$ on $X_W$ is a spherical object 
(since $X_W$ is Calabi-Yau). 
\end{exa} 

By Proposition \ref{pro:Gamma-intstr-properties} 
and \eqref{eq:mirrormap-Galois}, 
the monodromy of flat sections $\frs(\cE)$ 
around the paths $\gamma_{\CY}$, $\gamma_{\LG}$ 
(Figure \ref{fig:paths_in_M}) 
comes from the autoequivalences $\cO(-1)$, $(1)$ 
of $D^b(X_W)$ and $\MFhom^{\rm gr}_{\bmu_d}(W)$ 
respectively. 
We already saw in \eqref{eq:ancont-frs} that 
the analytic continuation along $\gamma_l^{-1}$ 
(Figure \ref{fig:ancont-path}) is 
induced by the Orlov equivalence $\Phi_l$. 
Thus the monodromy along $\gamma_{\con}^{-1}$ 
corresponds to the composition $\Phi_0 \circ \Phi_{1}^{-1}$. 
The following proposition shows that the monodromy 
around $\gamma_{\con}^{-1}$ comes from the spherical twist $T_\cO$. 

\begin{pro} 
\label{pro:spherical} 
For $E\in D^b(X_W)$ such that $\ch(E) \in H_{\amb}(X_W)$, 
we have 
$[\Phi_l \Phi_{l+1}^{-1}(E) ] = [T_{\cO(l)}(E)]$ 
in the numerical $K$-group.  
\end{pro} 
\begin{proof} 
Let $\{\ua,\ub\}_q$ be the graded Koszul matrix factorization 
in Example \ref{exa:ch-Koszul}. 
Recall that $\ch(\{\ua,\ub\}_q)$, $q\in \ZZ$ 
span $H_{\nar}(W,\bmu_d)$. 
Hence by Theorem \ref{thm:UisOrlov},  
$\ch(\Phi_l(\{\ua,\ub\}_q))$, $q\in \ZZ$ 
also span $H_{\amb}(X_W)$ 
since $\UU_l \colon H_{\nar}(W,\bmu_d) \cong 
H_{\amb}(X_W)$. 
Therefore, it suffices to check that 
$[T_{\cO(l)} \Phi_{l+1}(\{\ua,\ub\}_q)] = [\Phi_l(\{\ua,\ub\}_q)]$ 
in the $K$-group. 
By Proposition \ref{pro:OrlovKoszul}, we have 
\begin{align*} 
[\Phi_{l+1}(\{\ua,\ub\}_q)] & = 
\sum_{\substack{{j_1<\dots<j_r}\\ 
\sum_{a=1}^r w_{j_a}\le m'}} 
(-1)^{r+1} 
[\cO\left(l+ 1 + \textstyle m' - 
\sum_{a=1}^r w_{j_a}\right) ] \\ 
[\Phi_{l}(\{\ua,\ub\}_q)] & = 
\sum_{\substack{{j_1<\dots<j_r}\\ 
\sum_{a=1}^r w_{j_a}\le m }} 
(-1)^{r+1}
[\cO\left(l+ \textstyle m - 
\sum_{a=1}^r w_{j_a}\right)]  
\end{align*} 
where $m$ (resp.\ $m'$) 
is the remainder of $q-l$ (resp.\ $q-l-1$) divided by $d$. 
Because $[T_{\cO(l)} E] = [E] - \chi(E(-l)) [\cO(l)]$, we have 
\begin{align} 
\label{eq:twist-Phi_l+1}
\begin{split} 
[T_{\cO(l)} \Phi_{l+1}(\{\ua,\ub\}_q)] 
= & \sum_{\substack{{j_1<\dots<j_r}\\ 
\sum_{a=1}^r w_{j_a}\le m'}} 
(-1)^{r+1} 
[\cO\left(l+ 1 + \textstyle m' - 
\sum_{a=1}^r w_{j_a}\right) ] \\ 
& + \sum_{\substack{{j_1<\dots<j_r}\\ 
\sum_{a=1}^r w_{j_a}\le m'}} 
(-1)^{r} 
\chi(\cO\left( 1+ \textstyle m' - 
\sum_{a=1}^r w_{j_a}\right)) 
[\cO(l)].  
\end{split} 
\end{align} 
Here we use the following fact: For $1\le i\le d$, 
we have 
\begin{align*} 
\chi(\cO(i)) & = \dim H^0(X_W,\cO(i)) \\ 
& = 
\begin{cases} 
\sharp\left\{
k_1\le \cdots \le k_s \,|\, s\ge 0, \sum_{b=1}^s w_{k_b} = i
\right\} 
& \text{ if } 1\le i\le d-1 \\ 
\sharp\left\{
k_1 \le \cdots \le k_s \,|\, s\ge 0, \sum_{b=1}^s w_{k_b} =i
\right\} -1 
& \text{ if } i = d. 
\end{cases} 
\end{align*} 
Therefore the second term of the right-hand side 
of \eqref{eq:twist-Phi_l+1} 
gives 
\begin{equation} 
\label{eq:coeff-O(l)} 
[\cO(l)] 
\left( -\delta_{m', d-1} + 
\sum_{\substack{{j_1<\dots<j_r, \ k_1\le \cdots \le k_s }\\ 
\sum_{a=1}^r w_{j_a} + \sum_{b=1}^s w_{k_b}= m'+1, \ 
\sum_{a=1}^r w_{j_a} \le m'}} 
(-1)^{r} \right). 
\end{equation} 
We claim that for $m'\ge 0$ 
\[
\sum_{\substack{{j_1<\dots<j_r, \ k_1\le \cdots\le k_s }\\ 
\sum_{a=1}^r w_{j_a} + \sum_{b=1}^s w_{k_b}= m'+1}}   
(-1)^{r} = 0.  
\]
The claim follows from the comparison of the coefficient of $t^{m'+1}$ 
in the following equality: 
\[
1 = 
\frac{(1-t^{w_1})(1-t^{w_2}) \cdots (1-t^{w_N}) }
{(1-t^{w_1})(1-t^{w_2}) \cdots (1-t^{w_N})} 
= \sum_{p,q\ge 0} 
\left( 
\sum_{\substack{j_1 < \cdots < j_r \\ \sum_{a=1}^r w_{j_{a}} = p}} 
(-1)^r t^p   \right) 
\left( 
\sum_{\substack{k_1 \le \cdots \le k_s \\ 
\sum_{b=1}^s w_{k_b} = q}} t^q 
\right). 
\]
By the above claim, \eqref{eq:coeff-O(l)} can be rewritten as  
\[
[\cO(l)] 
\left( -\delta_{m', d-1} - 
\sum_{\substack{j_1<\dots<j_r \\ 
\sum_{a=1}^r w_{j_a} = m'+1}} 
(-1)^{r} \right).  
\]
This gives the second term of the right-hand side of 
\eqref{eq:twist-Phi_l+1}. 

First consider the case where $m' < d-1$. 
In this case, by the above calculation, 
$[T_{\cO(l)} \Phi_{l+1}(\{\ua,\ub\}_q)]$ 
equals $[\Phi_{l}(\{\ua,\ub\}_q)]$ 
because $m = m' +1 $. 
Next consider the case where $m'= d-1$. 
In this case, we have $m=0$ and 
\[
[T_{\cO(l)} \Phi_{l+1}(\{\ua,\ub\}_q)] = 
\sum_{j_1<\cdots<j_r} (-1)^{r+1} 
[\cO(\textstyle l+d - \sum_{a=1}^r w_{j_a})] -[\cO(l)]. 
\]
We know from the Koszul complex $\cE_d$ \eqref{eq:exactseq} 
that the first term in the right-hand side vanishes. 
Because $m=0$ we have $[\Phi_l(\{\ua,\ub\}_q)]=-[\cO(l)]$. 
The conclusion follows.  
\end{proof} 

\begin{rem} 
\label{rem:Orlov-spherical} 
We should have an isomorphism of functors 
$T_{\cO(l)} \cong \Phi_l \circ \Phi_{l+1}^{-1}$,  
but this does not seem to be proved in the literature. 
E.\ Segal \cite[Theorem 3.13]{Segal:equivalence} 
showed a similar (object-wise) relationship  
in the category of B-branes on the LG model 
$(K_{\PP(\uw)}, \tW)$ (which should be equivalent to $D^b(X_W)$).  
\end{rem} 


We speculate that the relations in the fundamental groupoid of $\cM$ 
\begin{align*}
\gamma_{l+1} 
& = \gamma_{\LG} \circ \gamma_l \circ \gamma_{\CY}, \\ 
\gamma_{\con} 
& = \gamma_1^{-1} \circ\gamma_0 \\ 
\gamma_{\LG}^d 
& =  \id 
\end{align*}  
should be lifted to category equivalences as 
\begin{align*} 
\Phi_{l+1}^{-1} & \cong  (1)  \circ \Phi_{l}^{-1} \circ \cO(-1)  \\ 
T_\cO^{-1} & \cong \Phi_1 \circ \Phi_0^{-1}, \\ 
(d) & \cong [2].   
\end{align*} 
The second relation is conjectural 
(see Remark \ref{rem:Orlov-spherical}) but the other two 
are easy to show. 
Note that the identity in the fundamental groupoid 
is lifted to the 2-shift $[2]$ 
in the third relation.  
This is the reason why we have to mod out by $[2]$ 
in the statement of Theorem \ref{thm:monodromyrep}. 

Finally we check the last statement in Theorem \ref{thm:monodromyrep}. 
The fundamental group of $\cM$ is 
generated by $\gamma_{\CY}$, $\gamma_{\con}$ 
and is defined by the relation 
\[
(\gamma_{\CY} \circ \gamma_{\con})^d =\id.  
\]
We define the lift $\hrho \colon \pi_1(\cM,b_0) \to \Auteq(D^b(X))/[2]$ 
by sending $\gamma_{\CY}$ to $\cO(-1)$ and 
$\gamma_{\con}$ to $T_\cO^{-1}$ as we speculated above. 
It suffices to check the relation: 
\[
(O(-1) \circ T_\cO^{-1})^d \cong [2].     
\]
This was proved by Canonaco-Karp \cite{Cano-Karp}. 
The proof of Theorem \ref{thm:monodromyrep} 
(hence of Theorem \ref{thm:overview-autoeq}) 
is now complete. 

\appendix 
\section{Proof of Proposition \ref{pro:FJRWstatespace_Jacobi}}
\label{sec:proof-Jacobi} 
When $N_k=0$, both sides of \eqref{eq:FJRWstatespace_Jacobi}
are one-dimensional and their pairings match. 
When $N_k=1$, both sides are zero. 
Assume that $N_k\ge 2$. 
The relative cohomology exact sequence 
identifies $H^{N_k}((\CC^N)_k, W_k^{+\infty})$ with 
$H^{N_k-1}(W_k^{+\infty})\cong H^{N_k-1}(W_k^{-1}(1))$. 
Therefore 
\[
H(W,\bmu_d)_k \cong H^{N_k-1}(W_k^{-1}(1))^{\bmu_d}. 
\]
We use the following result of Steenbrink: 
\begin{thm}[{\cite[Theorem 1]{Steenbrink}}]  
The Deligne weight filtration $\rsW_\bullet$ 
on $H^{N_k-1}(W_k^{-1}(1))$ is of the form 
\[
0= \rsW_{N_k-2} \subset \rsW_{N_k-1} 
\subset \rsW_{N_k}= H^{N_k-1}(W_k^{-1}(1)). 
\]
Take a set $\{\varphi_1,\dots,\varphi_L\} \subset 
\Omega_{(\CC^N)_k}^{N_k}$ 
of homogeneous $N_k$-forms which gives 
a basis of $\Omega(W_k)$. 
Let $|i|$ denote the degree of $\varphi_i$ divided by $d$. 
Define $\eta_i \in H^{N_k-1}(W_k^{-1}(1))$ 
by 
\begin{equation} 
\label{eq:residue_form} 
\eta_i := c_i \Res_{W_k(x)=1} \left( 
\frac{\varphi_i}{(W_k(x)-1)^{\ceil{|i|}}}
\right) 
\end{equation}
with $c_i = \Gamma(1-\fracp{-|i|}) 
(\ceil{|i|}-1)!$. 
Then the set  $\{\eta_i \,|\, N_k-1-p<|i| < N_k-p\}$ 
gives a basis of 
$\Gr_{\rsF}^{p} (\rsW_{N_k-1})$; 
the set $\{\eta_i\, |\, |i|=N_k-p\}$ 
gives a basis of $\Gr_{\rsF}^{p} 
(\rsW_{N_k}/\rsW_{N_k-1})$. 
\end{thm} 
There is a typo in the statement of 
\cite[Theorem 1]{Steenbrink} about the 
index of the Hodge filtration and we 
corrected it above. 
The prefactor $c_i$ is not important in 
the above statement, but is chosen for our later purpose. 
Since $\{\eta_i \,|\, |i|\in \ZZ\}$ 
gives a basis of the $\bmu_d$-invariant part 
of $H^{N_k-1}(W_k^{-1}(1))$, by the theorem, 
the $\bmu_d$-invariant part splits the 
weight filtration: 
\begin{equation*} 
\rsW_{N_k}/\rsW_{N_k-1} \cong H^{N_k-1}(W_k^{-1}(1))^{\bmu_d}.  
\end{equation*} 
Therefore the sector $H(W,\bmu_d)_k\cong 
H^{N_k-1}(W_k^{-1}(1) )^{\bmu_d}$ 
has a pure Hodge structure of weight $N_k$. 
Moreover the theorem gives an isomorphism 
\[
\Omega(W_k)^{\bmu_d} \cong 
\Gr_{\rsF}^\bullet  H^{N_k-1}(W_k^{-1}(1))^{\bmu_d}, \quad 
 [\varphi_i]   \longmapsto [\eta_i]  
\]  
independent of the choice of representatives $\varphi_i$. 
The isomorphism \eqref{eq:FJRWstatespace_Jacobi}  
is defined by the Hodge decomposition 
which splits the above isomorphism:  
\begin{equation} 
\label{eq:Hodge-decomp-Jacobi}
H(W,\bmu_d)_k \cong H^{N_k-1}(W_k^{-1}(1))^{\bmu_d} 
= \bigoplus_{p=0}^{N_k} \rsF^{\,p} \cap \ov\rsF^{\,N_k-p} 
\cong \, \Omega(W_k)^{\bmu_d}.  
\end{equation} 

Next we study the pairing on the FJRW state space.  
The form $e^{-W_k} \varphi_i$ 
defines a cohomology class in 
$H^{N_k}((\CC^N)_k,W_k^{+\infty})$ 
via the integration over non-compact 
Lefschetz thimbles $\Gamma
\in H_{N_k}((\CC^N)_k,W_k^{+\infty})$ of $W_k$: 
\[
\Gamma \longmapsto \int_{\Gamma} e^{-W_k} \varphi_i.  
\]
The following lemma shows that the set 
$\{e^{-W_k} \varphi_i \,|\, |i|\in \ZZ, \, 
|i|\le N_k-p\}$ 
of relative cohomology classes 
forms a basis of $\rsF^p H(W,\bmu_d)_k$. 
It also shows that $[\varphi_i]\in \Omega(W_k)^{\bmu_d}$ 
corresponds to an element of the form 
$[(\varphi_i + \sum_{|j|<|i|} a_{ji} \varphi_j) e^{-W_k}] 
\in H(W,\bmu_d)_k$ 
under \eqref{eq:Hodge-decomp-Jacobi}. 
\begin{lem} 
Under the isomorphism 
$H^{N_k}((\CC^N)_k,W^{+\infty}_k) \cong 
H^{N_k-1}(W_k^{-1}(1))$, 
the class represented by $e^{-W_k} \varphi_i$  
corresponds to the class 
$\eta_i$ in \eqref{eq:residue_form}. 
\end{lem} 
\begin{proof} 
Let $\Gamma$ be a Lefschetz thimble of $W_k$ 
in $H_{N_k}((\CC^N)_k,W^{+\infty}_k)$ 
and $C \in H_{N_k-1}(W_k^{-1}(t))$ be the 
corresponding cycle. 
(Note that $H_{N_k}((\CC^N)_k,W^{+\infty}_k)
\cong  H_{N_k-1}(W_k^{-1}(1))$.) 
The image of $\Gamma$ under $W_k$ is assumed to be 
the positive real line. 
Then we have 
\begin{equation}
\label{eq:Laplace-trans}  
\int_\Gamma e^{-W_k} \varphi_i 
= \int_0^\infty e^{-t} P(t) \dd t.  
\end{equation} 
Here we set 
\begin{equation} 
\label{eq:def-P} 
P(t):= \int_{\Gamma \cap \{W_k(x)=t\}} \frac{\varphi_i}{dW_k} 
= \frac{1}{2\pi\iu} \int_{T} 
\frac{\varphi_i}{W_k(x)-t}
\end{equation} 
where $T$ is a circle bundle over $\Gamma\cap \{W_k(x)=t\}$. 
Using the homogeneity, one can deduce 
from the co-ordinate change $x_i \mapsto t^{-w_i/d} x_i$ 
that 
\begin{equation*} 
P(t) = t^{|i|-1} P(1). 
\end{equation*} 
Therefore by \eqref{eq:Laplace-trans}, 
\begin{equation} 
\label{eq:osci-periodder} 
\int_\Gamma e^{-W_k} \varphi_i = \Gamma(|i|) P(1) 
= \Gamma(1-\fracp{-|i|})P^{(\ceil{|i|}-1)}(1).  
\end{equation} 
By differentiating \eqref{eq:def-P} and setting $t=1$, we find 
\begin{equation} 
\label{eq:periodder-period} 
\Gamma(1-\fracp{-|i|})
P^{(\ceil{|i|}-1)}(1) =  \int_C \eta_i.  
\end{equation} 
The lemma follows from \eqref{eq:osci-periodder} and 
\eqref{eq:periodder-period}. 
\end{proof} 
Consider the tame deformation $W_{k,s}$ of $W_k$: 
\[
W_{k,s}(x) = W_k(x) + \sum_{i\in F_k} s_i x_i, 
\]
where $F_k := \{1\le j\le N\,|\, \zeta^{k w_j} =1\}$ 
is the index set of co-ordinates on $(\CC^N)_k$. 
For generic values of $s$, $W_{k,s}$ has only 
non-degenerate critical points (i.e.\ it is 
a Morse function). 
Let $z\in \CC^\times$ and 
$\pair{\cdot}{\cdot} 
\colon H^{N_k}((\CC^N)_k, (W_{k,s}/z)^{+\infty}) 
\times H^{N_k}((\CC^N)_k, (W_{k,s}/z)^{-\infty}) 
\to \CC$ denote the 
intersection pairing (cf.\ \eqref{eq:intersection-pairing}).  
Set 
\[
G_{ij}(s,z):= 
\Pair{[e^{-W_{k,s}/z} \varphi_i]}{[e^{W_{k,s}/z} \varphi_j]}.  
\]
This is a presentation of 
K.~Saito's higher residue pairing \cite{SaitoK:higherresidue} 
by Pham \cite{Pham:onglet}. 
The invariance of the pairing under the co-ordinate 
change $x_j \mapsto \lambda^{w_j/d} x_j$ shows the 
following. 
\begin{lem} 
\label{lem:degree_G_ij} 
With respect to the degree $\deg s_i := 1- (w_i/d)$ and $\deg z: =1$,  
the function $G_{ij}(s,z)$ is homogeneous of 
degree $|i| + |j|$. 
\end{lem} 
\begin{lem}[] 
\label{lem:z-expansion_G_ij} 
The function $G_{ij}(s,z)$ is regular at $z=0$. Moreover 
\[
G_{ij}(s,z) = (-1)^{\frac{N_k(N_k-1)}{2}}  
(2\pi\iu z)^{N_k} 
\left( \Res_{W_{k,s}} ([\varphi_k], [\varphi_l]) 
+ O(z) \right).  
\]
\end{lem} 
\begin{proof} 
This is remarked in \cite[2\`{e}me Partie, \S 4.3, Remarque]{Pham:onglet}, 
but we include a proof 
for the convenience of the reader. 
Suppose that $s$ is generic so that $x\mapsto \Re(W_{k,s}(x)/z)$ 
is a Morse function. 
Let $\Gamma_1^+,\dots,\Gamma_L^+$ 
(resp.\ $\Gamma_1^-,\dots,\Gamma_L^-$) 
denote the Lefschetz thimbles emanating from the critical 
points $\sigma_1,\dots,\sigma_L$ of $\Re(W_{k,s}/z)$ 
given by the upward (resp.\ downward) 
gradient flow. Choose an orientation of $\Gamma_i^\pm$ such 
that $\Gamma_a^+ \cdot \Gamma_b^- = \delta_{ab}$. 
We have 
\[
G_{ij}(s,z) = \sum_{a=1}^L 
\left ( \int_{\Gamma_a^+} e^{-W_{k,s}/z} \varphi_i \right) 
\cdot 
\left( \int_{\Gamma_a^-} e^{W_{k,s}/z} \varphi_j \right). 
\]
For a fixed argument of $z$, we have the 
stationary phase expansion as $z\to 0$. 
\[
\int_{\Gamma_a^+} e^{-W_{k,s}/z} \varphi_i 
\sim 
\pm 
\frac{(2\pi z)^{N_k/2}}{\sqrt{\Hess W_{k,s}(\sigma_a)}} 
(f_i(\sigma_a) + O(z)).  
\]
Here we set $\varphi_i = f_i(x) 
\bigwedge_{j\in F_k} \dd x_{j}$, 
$\Hess W_{k,s}(\sigma_a):= \det \left (
(\partial_{x_{i}} \partial_{x_{j}} W_{k,s})_{i,j\in F_k}
\right)$ is the Hessian of $W_{k,s}$ 
at $\sigma_a$ and $\pm$ is the sign depending on the 
orientation of $\Gamma_a^+$. Therefore 
\begin{align*} 
G_{ij}(s,z) & \sim (-1)^{\frac{N_k(N_k-1)}{2}} (2\pi\iu z)^{N_k} 
\left ( \sum_{a=1}^N 
\frac{f_i(\sigma_a) f_j(\sigma_a)}{\Hess W_{k,s}(\sigma_a)} + O(z)
\right )  
\end{align*} 
where the lowest order term in the right-hand side 
equals the Grothendieck residue. 
The sign $(-1)^{\frac{N_k(N_k-1)}{2}}$ comes from 
a local computation\footnote
{This comes from 
$\bigwedge_{j=1}^{N_k} \dd u_j \wedge 
\bigwedge_{j=1}^{N_k} \dd v_j 
= (-1)^{\frac{N_k(N_k-1)}{2}} 
\bigwedge_{j=1}^{N_k} (\dd u_j\wedge \dd v_j)$ 
where $\{u_j + \sqrt{-1} v_j\,|\, j=1,\dots, N_k\}$ is a 
local co-ordinate system centered at a critical point.} 
of the orientation.  
Since this holds for an arbitrarily fixed argument 
of $z$, and $G_{ij}(s,z)$ is holomorphic 
in $z\in \CC^\times$, the conclusion follows 
for a generic $s$. By analytic continuation, 
the same holds for all $s$. 
\end{proof} 

By Lemma \ref{lem:degree_G_ij} and 
Lemma \ref{lem:z-expansion_G_ij}, 
we have 
\begin{equation} 
\label{eq:G_ij_values} 
G_{ij}(0,z)  = 
\begin{cases} 
0  & \text{if }|i| + |j| < N_k \\ 
(-1)^{\frac{N_k(N_k-1)}{2}}  
(2\pi\iu z)^{N_k} 
\Res_{W_k} \left( [\varphi_i], [\varphi_j] \right)
& \text{if } |i|+|j| =N_k. 
\end{cases} 
\end{equation} 
This shows the Hodge-Riemann bilinear relation: 
\begin{equation} 
\label{eq:Hodge-Riemann} 
\left
(\rsF^p H(W,\bmu_d)_k, \rsF^q H(W,\bmu_d)_{d-k}
\right) =0 \quad 
\text{if } p+q > N_k.
\end{equation} 
For $i,j$ such that $|i|, |j| \in \ZZ$, 
we take lifts 
\[
[e^{-W_k} \hvarphi_i] \in \rsF^{\, p} \cap \ov{\rsF}^{\,N_k-p}, 
\quad 
[e^{-W_k} \hvarphi_j] \in \rsF^{\,q} \cap \ov{\rsF}^{\, N_k-q}
\]
which correspond to $[\varphi_i], [\varphi_j] 
\in \Omega(W_k)^{\bmu_d}$ under 
the isomorphism \eqref{eq:Hodge-decomp-Jacobi}. 
When $p+q>N_k$, the pairing 
$([e^{-W_k}\hvarphi_i], [e^{-W_k}\hvarphi_j]) =0$ 
vanishes by \eqref{eq:Hodge-Riemann}. 
When $p+q < N_k$, the pairing again vanishes 
because of the Hodge-Riemann bilinear 
relation \eqref{eq:Hodge-Riemann} for $\ov\rsF$.  
When $p+q = N_k$, we have 
\begin{align*} 
\left( [e^{-W_k}\hvarphi_i], [e^{-W_k}\hvarphi_j] 
\right) & =  \left([e^{-W_k} \varphi_i], [e^{-W_k} \varphi_j]
\right) \quad 
\text{by \eqref{eq:Hodge-Riemann}} 
\\ 
& = \frac{1}{d} 
\Pair{[e^{-W_k} \varphi_i]}{(-1)^{|j|}
[e^{W_k} \varphi_j]} \quad 
\text{by \eqref{eq:FJRW-pairing}} 
\\
& = 
(-1)^{\frac{N_k(N_k-1)}{2}}
(2\pi\iu)^{N_k} 
\frac{1}{d} 
\Res_{W_k} \left( [\varphi_i], (-1)^{|j|}[\varphi_j]  \right) 
\quad 
\text{by \eqref{eq:G_ij_values}}.  
\end{align*} 
The factor $(-1)^{|j|}$ comes 
from the map $I^*$ in \eqref{eq:Istar}. 
The proof of Proposition \ref{pro:FJRWstatespace_Jacobi} 
is complete. 

\begin{rem} 
In the proof we observed that $H(W,\bmu_d)_k$ 
has a Hodge structure of weight $N_k$. 
In order to make the weight compatible with the FJRW grading, 
we consider the Tate twist by $\sum_{i=1}^N \fracp{kq_i} -1$ so that 
$H(W,\bmu_d)_k$ is of weight  
$N_k + 2(\sum_{i=1}^N \fracp{kq_i} -1)$. 
\end{rem} 

\section{Compatibility with $\FJRW$ setup}
\label{sec:factorofd}
In \cite{FJR1} a factor $f_g=\abs{G}^g/\deg({\rm st})$ multiplies 
all genus-$g$ $n$-pointed invariants as well as all 
the homomorphisms 
\begin{align*}\label{eq:CohFT}
\Lambda_{g,n}^{W,G} \colon H(W,G)^{\otimes n} 
&\to H^*(\ol {\mathcal M}_{g,n};\QQ),\\ 
(\al_1,\dots,\al_n) &\mapsto  f_g {\rm st}_*
\left([\ol{\mathcal W}_{g,n,G}(W, (k_1,\dots k_n))]^{\vir}\cap 
\prod_{i=1}^n \al_i \right)
\end{align*}
defining a cohomological field theory via Poincar\'e duality. 
These operators embody 
all the relevant invariants via the definition 
\[
\langle \tau_{b_1}(\al_1),\dots,\tau_{b_n}(\al_n)\rangle^{W,G}_{g,n}
=\int_{\ol {\mathcal M}_{g,n}} \Lambda_{g,n}^{W,G}(\al_1,\dots,\al_n) 
\prod_{i=1}^n \psi_i^{b_i}.
\]
The cycle $[\ol{\mathcal W}_{g,n,G} 
(W,(k_1,\dots k_n))]^{\vir}\cap 
\prod_{i=1}^n \al_i$ 
is a cycle in the moduli space of $(W,G)$-curves; 
in this paper $G$ equals $\bmu_d$ and, in genus zero and
for narrow state space entries, 
we may regard this  as 
the top Chern class of the obstruction bundle.
The degree of $\rm st$ is simply the degree 
of the map forgetting the $(W,G)$-structure and retaining 
only the underlying coarse stable curve; 
for $(W,G)=(W,\bmu_d)$ the morphism $\rm st$ 
is the natural forgetful map 
$\Spin^d_{g,n}(k_1,\dots,k_n)\to \ol{\mathcal M}_{g,n}$.
We have $\deg({\rm st})=\abs{G}^{2g-1}$ in general; 
therefore, $f_g$ equals $1/\abs{G}^{g-1}$,
and the setup of \cite{FJR1} is consistent 
with that of Witten's original tentative 
treatment \cite{Witten:original} of quantum singularity theory. 
In genus zero and for $G=\bmu_d$, 
this amounts to an overall factor $d$ 
appearing also in \cite[(14)]{ChiR}. 

We point out that all these different  
factors $f_g$ can be removed once we take into 
account that the pairing used \eqref{eq:FJRW-pairing}  comes from orbifold 
Chen-Ruan cohomology (in its relative version) and acquires 
an overall factor $1/\abs{G}$ equal to the degree of  
$BG$ over $\Spec \CC$ (we recall that the pairing of \cite{FJR1} 
maps the pair $(\phi_k, \phi_l)$ to $\delta_{d-1,k+l}$ 
without any factor). 
In particular, removing the factor $f_0=d$ in the definition of 
the genus-zero invariants does not change the quantum product: 
in the definition \eqref{eq:quantum-prod} of $T_i\bullet T_j$,  
the factor $f_0$ is absorbed into $g^{k,l}=d\delta_{d-2,k+l}$. 
Furthermore, removing the factors $f_g$ from the cohomological 
field theory homomorphisms $\Lambda_{g,n}^{W,G}$ 
does not affect the composition axioms \cite[(62),(64)]{FJR1}. 
Let $g = g_1 + g_2$;  let $n = n_1 + n_2$; and let
$\rho_{\rm tree} \colon  \ol {\mathcal M}_{g_1,n_1+1}\times  
\ol{\mathcal M}_{g_2,n_2+1}
\to \ol {\mathcal M}_{g,n}$ 
be the gluing morphism. Then the forms 
$$\wt \Lambda^{W,G}_{g,n}(\al_1,\dots,\al_n)= 
{\rm st}_*\left([\ol{\mathcal W}_{g,n,\bmu_d} 
(W,(k_1,\dots k_n))]^{\vir}\cap 
\prod_{i=1}^n \al_i \right)=\frac{\Lambda_{g,n}^{W,G}}{f_g}$$
satisfy the composition
property stated in \cite[(62)]{FJR1}
\[
\rho_{\rm tree}^*
\wt\Lambda^{W,G}_{g,n}(\alpha_1, \alpha_2,\dots, \alpha_n) =
\sum_{\mu,\nu} g^{\mu,\nu}
\wt\Lambda^{W,G}_{g_1,n_1+1}(\alpha_{i_1},\dots,\alpha_{i_{n_1}},\mu) \otimes 
\wt\Lambda^{W,G}_{g_2,n_2+1}(\alpha_{i_{n_1+1}},\dots,\alpha_{i_{n_1+n_2}},\nu).
\]
for all $\al_i\in H(W,G)$, for $\mu$ and $\nu$ 
running through a basis of $H(W,G)$, and for 
$g^{\mu,\nu}$ denoting the inverse of the pairing 
$(\ , \ )$ with respect to the chosen basis. 
This happens because, by rescaling the pairing, we have 
multiplied $g^{\mu,\nu}$ by $\abs{G}$; 
the cancelled factors on the two sides of the above identity match
\[
f_g = \frac{1}{\abs{G}^{g-1}}
=\frac{1}{\abs{G}}\frac{1}{\abs{G}^{g_1-1}}\frac{1}{\abs{G}^{g_2-1}}
=\frac{1}{\abs{G}}f_{g_1}f_{g_2}.
\]
The same happens for the gluing morphism 
$\rho_{\rm loop} \colon  \ol {\mathcal M}_{g-1,n+2}\to \ol{\mathcal M}_{g,n}$. 
We have  
\[
\rho_{\rm loop}^*
\wt\Lambda^{W,G}_{g,n}(\alpha_1, \alpha_2,\dots, \alpha_n) =
\sum_{\mu,\nu} g^{\mu,\nu}
\wt\Lambda^{W,G}_{g-1,n+2}(\alpha_{1},\dots,\alpha_{n},\mu,\nu).
\]
Taking again into account that, in this paper, the matrix $(g^{\mu,\nu})$ has been 
multiplied by an overall factor 
$\abs{G}$, the cancellation of factors from the analogue identity \cite[(64)]{FJR1} 
yields the same quantity on both sides 
$$f_g= \frac{1}{\abs{G}}\frac{1}{\abs{G}^{g-2}}=\frac{1}{\abs{G}} f_{g-1}.$$

\end{document}